\documentclass[10pt,final]{amsart}

\usepackage{docmute}

\usepackage[utf8]{inputenc}
\usepackage[T1]{fontenc}
\usepackage{amsmath, amssymb, amsthm}
\usepackage{mathrsfs}      % For fancy script letters
\usepackage{enumitem}      % Customizable lists
\usepackage{geometry}      % Page layout
\usepackage{graphicx}      % For figures
\usepackage{tikz}
\usepackage{tikz-cd}
\usepackage{mathtools}
\usepackage{stmaryrd}%\mapsfromのため
\usepackage{mathdots} %\iddotsのため

\usepackage[backend=biber,style=alphabetic,sorting=nyt,url=false]{biblatex}
\usepackage{comment}

\usepackage{hyperref}      % Clickable links and references
\usepackage{aliascnt}
\usepackage{cleveref}

\usepackage{xurl}
\theoremstyle{plain}
\newtheorem{theorem}{Theorem}[section]

\newaliascnt{lemma}{theorem}
\newtheorem{lemma}[lemma]{Lemma}
\aliascntresetthe{lemma}

\newaliascnt{corollary}{theorem}
\newtheorem{corollary}[corollary]{Corollary}
\aliascntresetthe{corollary}

\newaliascnt{proposition}{theorem}
\newtheorem{proposition}[proposition]{Proposition}
\aliascntresetthe{proposition}

\newtheorem{atheorem}{Theorem}

\theoremstyle{definition}
\newaliascnt{definition}{theorem}
\newtheorem{definition}[definition]{Definition}
\aliascntresetthe{definition}

\newaliascnt{example}{theorem}
\newtheorem{example}[example]{Example}
\aliascntresetthe{example}

\newaliascnt{remark}{theorem}
\newtheorem{remark}[remark]{Remark}
\aliascntresetthe{remark}
\newtheorem*{ackn}{Acknowledgements}
\newtheorem*{notation}{Notation}

\crefname{theorem}{Theorem}{Theorems}
\crefname{lemma}{Lemma}{Lemmas}
\crefname{corollary}{Corollary}{Corollaries}
\crefname{proposition}{Proposition}{Propositions}
\crefname{definition}{Definition}{Definitions}
\crefname{example}{Example}{Examples}
\crefname{remark}{Remark}{Remarks}
\crefname{section}{Section}{Sections}
\crefname{subsection}{Subsection}{Subsections}
\crefname{figure}{Figure}{Figures}

\DeclareDelimFormat{multicitedelim}{\addcomma\space} 

\newcommand{\Acal}{\mathcal{A}}

\newcommand{\Fcal}{\mathcal{F}}

\newcommand{\Lcal}{\mathcal{L}}

\newcommand{\Bb}{\mathbb{B}}

\newcommand{\Nbb}{\mathbb{Z}_{\geq 0}}

\newcommand{\Rbb}{\mathbb{R}}

\newcommand{\Tbb}{\mathbb{T}}

\newcommand{\Zbb}{\mathbb{Z}}

\newcommand{\Lsf}{\mathsf{L}}
\newcommand{\Msf}{\mathsf{M}}

\newcommand{\Fsf}{\mathsf{F}}

\newcommand{\pol}{\mathbb{T}[X_1,\dots,X_n]}
\newcommand{\lau}{\mathbb{T}[X_1^{\pm},\dots,X_n^{\pm}]}
\newcommand{\Bpol}{\mathbb{B}[X_1,\dots,X_n]}
\newcommand{\Blau}{\mathbb{B}[X_1^{\pm},\dots,X_n^{\pm}]}

\let\Im\relax

\newcommand{\Ker}{\operatorname{Ker}}
\newcommand{\Im}{\operatorname{Im}}
\newcommand{\Hom}{\operatorname{Hom}}

\newcommand{\CSpec}{\operatorname{CSpec}}
\newcommand{\Aff}{\operatorname{Aff}}

\newcommand{\het}{\operatorname{ht}}
\newcommand{\coht}{\operatorname{coht}}

\newcommand{\cone}{\operatorname{cone}}

\newcommand{\relint}{\operatorname{relint}}

\newcommand{\trop}{\operatorname{trop}}
\newcommand{\con}{\operatorname{cong}}

\newcommand{\Mp}{M_{\geq 0}}

\newcommand{\rk}{\operatorname{rk}}

\DeclareMathOperator{\supp}{supp}

\newcommand{\<}{\langle}
\renewcommand{\>}{\rangle}

\title{Space of prime congruences in tropical geometry}
\author{Kentaro Tanaka}
\address{Graduate School of Mathematics, Nagoya University\\
Furo-cho, Chikusa-ku, Nagoya 464-8602, Japan}
\email{kentaro.tanaka.h5@math.nagoya-u.ac.jp}

\begin{document}

\begin{abstract}
We investigate spaces of prime congruences on tropical algebras and study their geometry, 
inspired by classical scheme theory.
Our main strategy is to use tropical algebras associated to ordered monoids, which play the role of 
the monomial structure of these algebras.
Using the space of prime congruences as local models, 
we introduce a topological space which contains the usual tropical toric variety as a subspace.
We show how the separatedness and properness of this space are captured by scheme-theoretic points.
As an application of our framework, we obtain a necessary and sufficient condition 
for a prime congruence to be finitely generated.
\end{abstract}

\maketitle

\tableofcontents

\section{Introduction} \label{sec:intro}

\subsection{Background}
Tropical geometry has found an increasing number of applications in
algebraic geometry, providing combinatorial shadows of algebro-geometric
objects. Its primary geometric objects are rational polyhedral complexes
called tropical varieties, which are often obtained via the
tropicalization of algebraic varieties.

Recently, attempts have been made to develop a theory of tropical geometry 
from an algebraic perspective.
The basic algebras used in tropical geometry are the semifields 
\begin{align*}
  &\Bb=(\{-\infty,0\},\operatorname{max},+),\\
  &\Tbb=(\Rbb\cup\{-\infty\},\operatorname{max},+),
\end{align*}
and (semi)algebras over these semifields.
The set of functions on a tropical variety is considered to form such an algebra.
One of the main interests in algebraic tropical geometry is
what geometric data these algebras contain and how they can be extracted.

In \cite{joo2018prime}, 
the authors defined prime congruences and studied their properties.
A \emph{prime congruence} on a semiring $S$ is an equivalence relation on $S$
which gives a surjective semiring homomorphism from $S$ to 
a semiring with suitable properties (see \cref{definition:prime congruence} for the precise definition).
Note that this definition is slightly different from that of a prime congruence 
used in several other papers (\cite{lorscheid2012geometry,BertramEaston2017TropicalNullstellensatz,jarra2025strong}).
They also proposed the definition of the Krull dimension of an additively idempotent semiring $S$
and showed that it works well for some semirings which appear in tropical geometry.
\begin{theorem}[{\cite[parts of Theorems 4.6, 4.9, 4.14]{joo2018prime}}]
  The following equalities hold:
  \begin{itemize}
    \item $\dim \Blau=\dim \Bpol=n$,
    \item $\dim \lau=\dim \pol=n+1$.
  \end{itemize}
\end{theorem}
In order to compute the Krull dimension,
they showed that any prime congruence on $\Blau$ or $\lau$
can be represented by a matrix with coefficients in $\Rbb$, called an \emph{admissible matrix}.
Furthermore, in \cite{song2024geometricinterpretationkrulldimensions} and \cite{joo2026varieties}, 
the authors calculated the Krull dimension of semirings of the form $\lau/C$
where $C$ is a congruence satisfying appropriate conditions, using these matrices.

Admissible matrices are useful for expressing prime congruences as numerical data.
However, for a given prime congruence, the corresponding matrix is not uniquely determined.
This non-uniqueness makes it difficult to decide inclusion relations between prime congruences.

\subsection{Our contribution}

Our aim is to treat the set of prime congruences on semirings as a geometric object.
To do this in a way analogous to classical scheme theory, in this paper, 
we begin with a sufficiently abstract setting.
Some notions in our framework not only lead to several new results,
but also shed new light on previously known results.
Note that our setting is different from the one in \cite{giansiracusa2016equations}, 
where the authors construct schemes by gluing sets of prime ideals of semirings.
%Unfortunately, in this paper $\CSpec S$ is treated merely as a topological space, without additional structures such as a structure sheaf.

For a semiring $S$, we denote by $\CSpec S$ the set of all prime congruences on $S$.
Let $M$ be a free abelian group generated by $X_1,\dots,X_n$.
While previous work used matrices to represent an element of $\CSpec\Blau$ or $\CSpec\lau$, 
we employ \emph{flags} instead:
\begin{definition} (= \cref{definition:RM-flag})
  An \emph{$(\Rbb\times M)$-flag} is a collection of the following data:
  \begin{itemize}
    \item $l\in\Zbb_{\geq 0}$,
    \item A family of linear subspaces of $\Rbb\times M_{\Rbb}$
    \[\Rbb\times M_{\Rbb}=H_0\supsetneq H_1\supsetneq \dots\supsetneq H_l,\]
    such that $H_i$ is a hyperplane of $\Rbb(H_{i-1}\cap (\Rbb\times M))$ 
    for each $i\in \{1,\dots,l\}$,
    \item An orientation $H_{i,+}$ of $H_i$ as a hyperplane for each $i\in \{1,\dots,l\}$ such that 
  \[\Rbb_{\geq 0}\times \{0\} \subseteq \bigcup_{i=1}^l H_{i,+} \cup H_l.\]
  \end{itemize}
  The number $l$ is called the \emph{length} of the $(\Rbb\times M)$-flag. 
  Denote the set of all $(\Rbb\times M)$-flags by $\Fcal_{\Rbb\times M}$.
\end{definition}

We give an example of an $(\Rbb\times M)$-flag of length 2 where the rank $n$ of $M$ is equal to 1:
\begin{align*}
  &H_1    = \{(r,m)\in \Rbb\times M_{\Rbb}\mid r=0.5m\},\\
  &H_{1,+}= \{(r,m)\in \Rbb\times M_{\Rbb}\mid r>0.5m\},\\
  &H_2    = \{(0,0)\},\\
  &H_{2,+}=H_1\cap \{(r,m)\in \Rbb\times M_{\Rbb}\mid r>0\}.
\end{align*}
This $(\Rbb\times \Zbb)$-flag is illustrated in \cref{figure:RM-flag} 
where the "positive part" $(H_{1,+}\cup H_{2,+})\cap (\Rbb\times M)$ is represented by the red region.

\begin{figure}[h]
  \centering

  \begin{tikzpicture}[scale=0.5]
      \draw[->,>=stealth,thick] (-4.5,0)--(4.5,0)node[above]{$m$}; %x軸
      \draw[->,>=stealth,thick] (0,-4.5)--(0,4.5)node[right]{$r$}; %t軸
      \filldraw[black] (0,0) circle (3pt) node[below right]{O}; %原点

      \foreach \x/\y in {-4/-2, -3/-1.5, -2/-1, -1/-0.5, 0/0, 1/0.5, 2/1, 3/1.5, 4/2}{
          \draw[red, thick] (\x,\y) -- (\x,4.5);
          \draw[blue, thick] (\x,\y) -- (\x,-4.5);
          
      }

      \filldraw[black] (0,0) circle (3pt);

      \foreach \x/\y in {1/0.5, 2/1, 3/1.5, 4/2}{
            \filldraw[red] (\x,\y) circle (3pt);
      }

      \foreach \x/\y in {-4/-2, -3/-1.5, -2/-1, -1/-0.5}{
            \filldraw[blue] (\x,\y) circle (3pt);
      }
  \end{tikzpicture}

  \caption{An example of an $(\Rbb\times\Zbb)$-flag of length 2}
  \label{figure:RM-flag}
\end{figure}
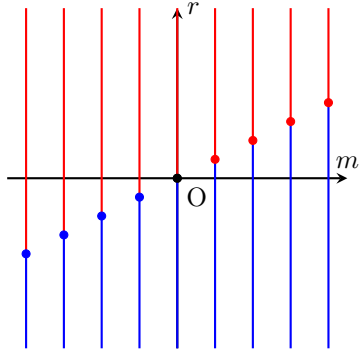

The notion of a flag arises naturally 
from the general framework developed in the first half of this paper.
As a consequence, 
an $(\Rbb\times M)$-flag provides a unique representation of prime congruences on $\lau$ 
in the following sense:
\begin{atheorem}[= \cref{theorem:flag corr group T}]\label{theorem:A}
  There is a one-to-one correspondence between $\CSpec \lau$ and 
  the set of all $(\Rbb\times M)$-flags.
\end{atheorem}
Similarly, the set of prime congruences on various tropical algebras, such as
\begin{itemize}
  \item $\CSpec \Blau$\ (\cref{theorem:flag corr group B}),
  \item $\CSpec \Bpol$\ (\eqref{eq:flag corr monoid B}),
  \item $\CSpec \pol$\ (\eqref{eq:flag corr monoid T}),
\end{itemize}
can be identified with sets of flags.

An inclusion relation between prime congruences corresponds to 
a certain relation between flags (see \cref{proposition:flag inclusion group T}).
Thus we can describe the set of prime congruences on several algebras, 
as shown in \cref{example:affine line}.
In particular, we give alternative proofs of some results in \cite{joo2018prime} 
(see \cref{corollary:JM result,corollary:JM result2,corollary:JM result3}).
See \cref{remark:difference} for the difference 
between our approach and computations using admissible matrices.

In \cite{friedenberg2025geometric}, the authors have obtained a related correspondence,
which classifies the continuous spectrum of a topological semiring using flags of polyhedra.
The two approaches share similar underlying ideas.
The previous work is developed from the viewpoint of Gröbner theory and adic geometry, 
whereas our approach is algebraic and scheme-theoretic.
Furthermore, in Section 4 of \cite{mikami2025tropicalcohomologysmoothalgebraic}, 
the author gives another interpretation of the set of homomorphisms from $M$ to totally ordered groups
where $M$ is a free abelian group of rank $n$.

\medskip
 
Following classical scheme theory,
we attempt to construct a global object using the space of prime congruences.
We use techniques of toric geometry, where the process of gluing is given 
by the combinatorial data of a fan.
For a fan $\Sigma$ in $N_{\Rbb}\cong \Rbb^n$, we define a topological space $X_{\Sigma}^{\con}$ by
\begin{align*}
  X_{\Sigma}^{\con}=\bigcup_{\sigma\in\Sigma}\CSpec_{\Tbb}\Tbb[\sigma^{\vee}\cap M],
\end{align*}
where the spaces are glued along the inclusions induced by face relations of cones.
Here, $\CSpec_{\Tbb}\Tbb[\sigma^{\vee}\cap M]$ is the subset of all points 
in $\CSpec\Tbb[\sigma^{\vee}\cap M]$ which are \emph{lying over $\Tbb$} (see \cref{definition:P over T}).

On the other hand, analogues of toric varieties 
have already been constructed in \cite{Kajiwara2008TropicalToric} and \cite{payne2009analytification}.
The \emph{tropical toric variety} associated with a fan $\Sigma$ is the topological space
\begin{align*}
  X_{\Sigma}^{\trop}=\bigcup_{\sigma\in\Sigma}\Hom(\Msf_{\sigma},\Tbb).
\end{align*}

Regarding $X_{\Sigma}^{\trop}$ as the set of $\Tbb$-valued points of $X_{\Sigma}^{\con}$, 
we have the following result.
\begin{atheorem}[= \cref{theorem:trop is height n toric,corollary:closed points}]\label{theorem:B}
  The tropical toric variety $X_{\Sigma}^{\trop}$ can be topologically embedded into
  the topological space $X_{\Sigma}^{\con}$ as 
  the subset consisting of points $P$ satisfying $\dim_{P}X_{\Sigma}^{\con}=\dim X_{\Sigma}^{\con}$.
  If $\Sigma$ is complete, the image is equal to the subset of closed points.
\end{atheorem}
Note that $\dim_{P}X_{\Sigma}^{\con}$ and $\dim X_{\Sigma}^{\con}$ are locally equal to 
the length of maximal chains of prime congruences.
This subspace is also related to 
the separatedness and properness of $X_{\Sigma}^{\con}$, 
although we will not give a precise definition of the notions in this paper.
See \cref{theorem:separatedness,theorem:properness} for the details.

\medskip

Our algebraic framework is also useful to study purely algebraic problems.
Somewhat surprisingly, even for standard tropical algebras such as $\Blau$ or $\lau$, 
many prime congruences fail to be finitely generated.
This phenomenon is illustrated by the following theorem.
\begin{atheorem}[= \cref{theorem:fg for T}]\label{theorem:C}
  For a prime congruence $P$ on $\lau$, the following are equivalent:
  \begin{enumerate}
    \item $P$ is finitely generated,
    \item $\Lsf(P)$ is a submonoid of $\Rbb\times M$ generated by finitely many elements and $\Rbb_{\geq 0}\times \{0\}$,
    \item $(\het P,\coht P)$ is equal to $(n+1,0)$ or $(n,1)$,
    \item $P$ is one of the following three: (i) the maximum congruence, (ii) a geometric congruence, 
    or (iii) a congruence which corresponds to an $(\Rbb\times M)$-flag $H_{\bullet}$ of length 1 
    and $H_1=\Rbb\times H$ where $H$ is a rational hyperplane of $M_{\Rbb}$.
  \end{enumerate}
\end{atheorem}
The monoid $\Lsf(P)$ appearing in the condition (2) is a submonoid of an ordered group $\Rbb\times M$ 
corresponding to $P$ 
(see \cref{eq:right arrow}).
The equivalence of (1) and (2) holds in a more general setting (see \cref{proposition:new fg}).
The conditions (3) and (4) arise from describing the submonoid $\Lsf(P)$ as a flag.
Note that finitely generated congruences play important roles in the Nullstellensatz 
for congruences in \cite{BertramEaston2017TropicalNullstellensatz}.
We also note that in \cite{ito2026nonfinitegeneratednesscongruencesdefined}, the author gives another criterion for 
a congruence to be finitely generated.

\subsection{Structure of the paper}

\begin{figure}[h]
  \begin{tikzpicture}

  \draw (0,0) rectangle (14,6.5);
  \node[fill=white] at (7,6.5) {Semirings (\S\ref{subsection:semirings})};

  \draw (0.5,0.5) rectangle (13.5,6);
  \node[fill=white] at (7,6) {$\mathbb{B}$-algebra (\S\ref{subsection:2.2})};

  \draw (1,1.5) rectangle (12.5,5);
  \node[fill=white] at (3.5,5) {$\mathbb{B}[\Msf]$ (\S\ref{subsection:B[M]})};

  \draw (6,1) rectangle (13,5.5);
  \node[fill=white] at (9.5,5.5) {$\Tbb$-algebra (\S\ref{subsection:2.3})};

  \draw (7,2) rectangle (12,4.5);
  \node[fill=white] at (9.5,4.5) {$\Tbb[\Msf]$ (\S\ref{subsection:T[M]})};

  \node at (3.5,3.8) {$\Blau$};
  \node at (3.5,3.2) {$\Bpol$};
  \node at (3.5,2.7) {$\vdots$};

  \node at (9.5,3.8) {$\lau$};
  \node at (9.5,3.2) {$\pol$};
  \node at (9.5,2.7) {$\vdots$};

  \end{tikzpicture}

\caption{Hierarchy of some classes of algebras introduced in this paper}
\label{figure:class}
\end{figure}
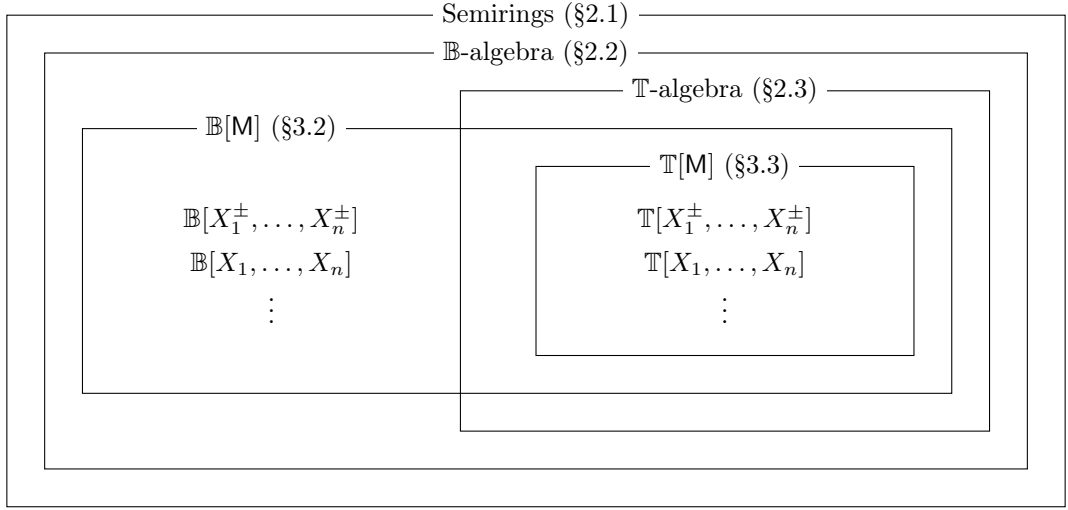

We begin in \cref{section:Preliminaries} by reviewing some properties of semirings.
In \cref{section:B[M]}, we introduce a \emph{$\Bb$-algebra associated to an ordered monoid}, 
denoted by $\Bb[\Msf]$, which is a key ingredient of this paper. 
\[\Bb[\Msf]\coloneqq \left(\bigoplus_{m\in\Msf}\Bb\right)/\sim.\]
This algebra is not merely the direct sum of copies of $\Bb$ indexed by $\Msf$.
That is, the equivalence relation $\sim$ is non-trivial in general depending on 
the order of $\Msf$. 
\cref{figure:class} summarizes the hierarchy of some classes of algebras introduced in this paper.

It is important that several tropical algebras such as $\Blau$ and $\lau$ can be treated as 
$\Bb$-algebras $\Bb[\Msf]$ associated to suitable ordered monoids $\Msf$ 
(by forgetting $\Tbb$-algebra structure if necessary).

The main assertions are stated in \cref{section:B[M]cong} and \cref{section:Mflag}, 
where we describe the details of the space of prime congruences.
In \cref{section:toric}, we discuss the relation between our framework and tropical toric varieties,
and in \cref{section:fg}, we characterize when a prime congruence is finitely generated.

\begin{ackn}
The author would like to thank Shu Kawaguchi, JuAe Song, and Junta Kamiya for many helpful discussions.
The author is also grateful to Yusuke Nakamura, Kosuke Mizuno, and Ryosuke Murooka
for useful comments and suggestions, and to Yuki Tsutsui for providing an overview of relevant prior work.

This work was supported by JSPS KAKENHI Grant Number JP26KJ1347.
The author used ChatGPT (GPT-5.6 Thinking) for proofreading, language editing, and improving the exposition.
All mathematical arguments were independently verified by the author.
\end{ackn}

\begin{notation}
\leavevmode
\begin{itemize}
  \item $\<\Acal\>$ : a congruence generated by a subset $\Acal\subseteq S\times S$
  \item $\Msf$ : an ordered monoid
  \item $G$ : an ordered group
  \item $\Lsf$ : a submonoid
  \item $\Nbb A$ : the submonoid generated by a subset $A$ of a monoid, namely, 
  \[\Nbb A=\left\{\sum_{i=0}^{k}n_ia_i\ \middle|\  n_i\in \Nbb,a_i\in A,k\in\Nbb\right\}.\]
  \item $\<A\geq 0\>$ : a congruence on $\Bb[G]$ generated by $\{(m\oplus 0,m)\mid m\in A\}$ 
  where $A$ is a subset of $G$
  \item $\Rbb A$ : the linear subspace spanned by a subset of an $\Rbb$-linear space
  \item $\cone(A)$ : the convex cone spanned by a subset $A$ of an $\Rbb$-linear space, namely, 
  \[\cone (A)=\left\{\sum_{i=0}^{k}r_ia_i\ \middle|\  r_i\in \Rbb_{\geq 0},a_i\in A,k\in\Nbb\right\}.\]
  \item $M$ : a free abelian group generated by $X_1,\dots,X_n$ ($M\cong \Zbb^n$)
  \item $\Mp$ : a free (commutative) monoid generated by $X_1,\dots,X_n$ ($\Mp\cong (\Nbb)^n$)
\end{itemize}
\end{notation}

\section{Preliminaries} \label{section:Preliminaries}

\subsection{Semirings and prime congruences}\label{subsection:semirings}

We briefly recall some basic properties of semirings.
See \cite{hebisch1998semirings} and \cite{golan2013semirings} for more details. 
When a specific statement is difficult to locate in the literature, 
we include a short proof for completeness.

In this paper, a \emph{semiring} $S$ is always assumed to be commutative and nontrivial.
Namely, a semiring is a nonempty set $S$ with two binary operations $(+,\cdot)$ satisfying: 
\begin{itemize}
  \item $(S,+)$ is a commutative monoid with the identity $0_S$,
  \item $(S,\cdot)$ is a commutative monoid with the identity $1_S$,
  \item For any $a,b,c\in S$, $a(b+c)=ab+ac$,
  \item $0_S\neq1_S$ and $a\cdot 0_S=0_S$ for any $a\in S$.
\end{itemize}
A semiring $S$ is said to be \emph{cancellative} if whenever $ab=ac$ for some $a,b,c\in S$, 
then either $a=0_S$ or $b=c$. 
A semiring $S$ is called a \emph{semifield} if every nonzero element of $S$ is 
multiplicatively invertible.
If a subset $I\subseteq S$ containing $0_S$ satisfies 
$a+b\in I$ and $sa\in I$ for any $a,b\in I$ and $s\in S$,
then $I$ is called an \emph{ideal} of $S$, which rarely appears in this paper.

A map $\phi\colon S_1\to S_2$ between semirings is a \emph{semiring homomorphism} if it satisfies
\[\phi(a+b)=\phi(a)+\phi(b),\ \phi(a\cdot b)=\phi(a)\cdot \phi(b),\ \phi(0_{S_1})=0_{S_2}, \text{ and }\phi(1_{S_1})=1_{S_2},\]
for any $a,b\in S_1$.
A bijective homomorphism is called an \emph{isomorphism}.

A \emph{congruence} $C$ on a semiring $S$ is a subsemiring of $S\times S$ 
that defines an equivalence relation on $S$.
Note that if $C$ is not equal to $S\times S$, 
the set of all equivalence classes $S/C$ naturally inherits a semiring structure 
and the quotient map $S\to S/C$ is a semiring homomorphism,
which is often denoted by $\pi_C$.
We call the diagonal set $\Delta=\Delta_S$ of $S\times S$ the \emph{diagonal congruence}.

Let $\{C_{\lambda}\}$ be a family of congruences on a semiring $S$.
Since the intersection $\bigcap_{\lambda}C_{\lambda}$ is also a congruence on $S$,
for any subset of $\Acal$ of $S\times S$, 
we can define the \emph{congruence generated by $\Acal$} as the minimal congruence containing $\Acal$, denoted by $\<\Acal\>$.
We denote by $\sum_{\lambda} C_{\lambda}$ the congruence generated by the union $\bigcup_{\lambda}C_{\lambda}$.
For some subsets $\Acal_1,\dots,\Acal_k$ of $S\times S$, 
we often denote a congruence $\<\bigcup_i \Acal_i\>$ by $\<\Acal_1,\dots,\Acal_k\>$.
A congruence on $S$ is called \emph{finitely generated} 
if it is generated by a finite subset of $S\times S$.

Let $\phi\colon S_1\to S_2$ be a semiring homomorphism.
The \emph{pull-back} of a congruence $C_2$ on $S_2$ along $\phi$ is defined by
\begin{align}
  \phi^{*}C_2\coloneqq C_2\cap S_1\coloneqq \{(a,b)\in (S_1)^2\mid (\phi(a),\phi(b))\in C_2 \}.\label{eq:pull-back}
\end{align}
Note that the notation $C_2\cap S_1$ is just a symbol and does not mean the set-theoretical intersection.
In particular, the pull-back of the diagonal congruence on $S_2$ 
is called the \emph{kernel (congruence)} of $\phi$, denoted by $\Ker\phi$.

\begin{proposition}[{\cite[Theorem 7.5]{hebisch1998semirings}}]\label{proposition:thm of hom}
  Let $\phi\colon S_1\to S_2$ be a homomorphism of semirings.
  Then $\phi$ induces an isomorphism $S_1/\Ker\phi\to \Im\phi$.
\end{proposition}

The \emph{push-out} of a congruence $C_1$ on $S_1$ along $\phi$ is defined by
\begin{align}
  \phi_{*}C_1\coloneqq C_1\cdot S_2\coloneqq \<\{(\phi(a),\phi(b))\in S_2\times S_2\mid (a,b)\in C_1\}\>.\label{eq:push-out}
\end{align}
It is straightforward to see that if $\phi$ is surjective,
$\phi_{*}C_1$ is equal to the set-theoretical image of $C_1$:
\begin{align}
  \phi_{*}C_1=\{(\phi(a),\phi(b))\in S_2\times S_2\mid (a,b)\in C_1\}.\label{eq:surj push-out}
\end{align}

\begin{lemma}\label{lemma:pushpush}
  Let $C_1,C_2$ be congruences on $S$. Then there is an isomorphism
  \[S/\< C_1,C_2\>\cong (S/C_1)/(C_2\cdot (S/C_1)).\]
\end{lemma}

\begin{proof}
  Let $\pi=\pi_{C_1}:S\to S/C_1$ be the quotient map by $C_1$ and $C$ the kernel congruence of 
  \[S\to S/C_1 \to (S/C_1)/(C_2\cdot(S/C_1)).\]
  Clearly $C$ contains $C_1$ and $C_2$. Take $(a,b)\in C$. 
  By the property \eqref{eq:surj push-out}, 
  the push-out $C_2\cdot (S/C_1)$ is the set-theoretical image of $C_2$ by $\pi\times \pi$.
  Then there exists a pair $(a',b')\in C_2$ such that $(\pi(a),\pi(b))=(\pi(a'),\pi(b'))$.
  Since $(a,a'),(b,b')\in\Ker\pi=C_1$, then $(a,b)\in \<C_1,C_2\>$. Thus we have $C=\<C_1,C_2\>$.
  By \cref{proposition:thm of hom}, we have $S/\<C_1,C_2\> \cong S/C\cong (S/C_1)/(C_2\cdot (S/C_1))$.
\end{proof}

Let $S$ be a semiring. The \emph{twisted product} of the pairs 
$\alpha=(\alpha_1,\alpha_2)$ and $\beta=(\beta_1,\beta_2)$ of $S\times S$ is
\[\alpha\rtimes \beta=(\alpha_1\beta_1+\alpha_2\beta_2,\alpha_1\beta_2+\alpha_2\beta_1).\]
Note that this product is associative and commutative.
For two congruences $C_1,C_2$ on $S$, define $C_1\rtimes C_2$ as a congruence
generated by $\{\alpha\rtimes \beta\mid \alpha\in C_1,\beta\in C_2\}$.
Note that $C_1\rtimes C_2\subseteq C_1\cap C_2$ holds.

\begin{definition}[{\cite[Definitions 2.3, 2.4]{joo2018prime}}]\label{definition:prime congruence}
  A congruence $C$ on $S$ is \emph{prime} if the complement $(S\times S)\setminus C$ is 
  a nonempty multiplicatively closed set under the twisted product. 
  In the definitions below, all congruences appearing in chains are assumed to be prime congruences on $S$.
  The \emph{dimension} of $S$ 
  is the supremum of all integers $l\geq 0$ such that there exists a chain
  \[P_0\subsetneq P_1\subsetneq\dots\subsetneq P_l.\]
  For $P\in\CSpec S$, the \emph{height} of $P$ 
  is the supremum of all integers $l\geq 0$ such that there exists a chain
  \[P_0\subsetneq P_1\subsetneq\dots\subsetneq P_l=P,\]
  and the \emph{coheight} of $P$ 
  is the supremum of all integers $l\geq 0$ such that there exists a chain
  \[P=P_0\subsetneq P_1\subsetneq\dots\subsetneq P_l.\]
  These are denoted by $\dim S$, $\het P$, and $\coht P$, respectively. 
\end{definition}

In this paper, we denote by $\CSpec S$ the set of all prime congruences on $S$.
For a congruence $C$ on $S$, we define the subset $V(C)$ of $\CSpec S$ as follows:
\[V(C)\coloneqq \{P\in\CSpec S\mid C\subseteq P\}.\]

\begin{proposition}\label{proposition:topology}
  \begin{enumerate}
    \item $V(S\times S)=\emptyset$ and $V(\Delta)=\CSpec S$ hold. 
    \item For two congruences $C_1,C_2$ on $S$, $V(C_1\rtimes C_2)=V(C_1)\cup V(C_2)$ holds.
    \item For a family of congruences $\{C_i\}_i$ on $S$, $V(\sum_i C_i)=\bigcap_i V(C_i)$ holds, 
    where $\sum_i C_i$ is a congruence generated by $\bigcup_i C_i$.
  \end{enumerate}
\end{proposition}

\begin{proof}
  \begin{enumerate}
    \item It follows directly.
    \item For two congruences $C_1,C_2$ on $S$, $V(C_1)\cup V(C_2)\subseteq V(C_1\rtimes C_2)$ clearly holds
    from $C_1\rtimes C_2\subseteq C_1\cap C_2$. 
    To show the opposite inclusion, assume $P\notin V(C_1)\cup V(C_2)$. 
    Then there exist pairs $(a,b)\in C_1\setminus P$ and $(c,d)\in C_2\setminus P$.
    Since $P$ is prime, we have $(a,b)\rtimes (c,d)\in (C_1\rtimes C_2)\setminus P$. 
    Thus $P\notin V(C_1\rtimes C_2)$.
    \item Since $\sum C_i\subseteq P$ if and only if $C_i\subseteq P$ for each $i$, the claim is straightforward. 
  \end{enumerate}
\end{proof}

By \cref{proposition:topology}, the set $\CSpec S$ has a topology 
whose closed sets are represented as $V(C)$ for some congruence $C$ on $S$.
We always consider $\CSpec S$ endowed with the topology defined above.

For a semiring homomorphism $\phi:S_1\to S_2$, 
the pull-back $\phi^{*}P_2$ of a prime congruence $P_2$ on $S_2$ is a prime congruence on $S_1$.
So we obtain the map $\phi^{*}:\CSpec S_2\to\CSpec S_1$.
For a closed set $V(C_1)$ of $\CSpec S_1$,
\begin{align*}
  (\phi^{*})^{-1}(V(C)) &= \{P\in\CSpec S_2\mid C\subseteq \phi^{*}P\}\\
  &=\{P\in\CSpec S_2\mid \phi_{*}(C)\subseteq P\}\\
  &=V(\phi_{*}(C))
\end{align*}
Therefore $\phi^{*}$ is continuous.

%top spaceとしての整理．
For the discussion in \cref{section:toric}, 
where we glue topological spaces $\CSpec S$ together,
we introduce terminology for topological spaces.
A closed subset $Z$ of a topological space is said to be \emph{reducible} 
if $Z$ can be written as a union $Z=Z_1\cup Z_2$ 
where $Z_1,Z_2\subsetneq Z$ are nonempty closed subsets.
A closed subset $Z$ of $X$ is said to be \emph{irreducible} if $Z$ is not reducible.

\begin{definition}\label{definition:dim of top space}
  Let $X$ be a topological space and $x$ be a point of $X$.
  The \emph{dimension of $X$} is the supremum of all integers $l$ such that there exists a chain
  \[Z_0 \subsetneq Z_1 \subsetneq \dots \subsetneq Z_l\]
  of irreducible closed subsets of $X$.
  The \emph{dimension of $X$ at $x$} is the supremum of all integers $l$ such that there exists a chain
  \[x\in Z_0 \subsetneq Z_1 \subsetneq \dots \subsetneq Z_l\]
  of irreducible closed subsets of $X$.
  These are denoted by $\dim X$ and $\dim_x X$, respectively.
\end{definition}

Note that the dimension at a point can be calculated in an open neighborhood of the point, that is,
\begin{align}\label{eq:dim at point open nbhd}
  \dim_x X =\dim_x U, \text{ for any open neighborhood }U\text{ of }x.
\end{align}
The following lemma implies 
that the notions such as height, coheight, and dimension are topological invariant.

\begin{lemma}\label{lemma:dim height}
  Let $X=\CSpec S$ be the space of prime congruences on a semiring $S$ and $P$ a point of $X$.
  Then the following equalities hold:
  \begin{itemize}
    \item $\dim X=\dim S$,
    \item $\dim_P X=\het P$.
  \end{itemize}
\end{lemma}

\begin{proof}
  It is enough to show that a nonempty closed subset $Z=V(C)$ of $X$ is irreducible if and only if
  there is a prime congruence $P_0$ on $S$ such that $Z=V(P_0)$.

  First we assume that $Z$ is irreducible. 
  It suffices to show that $P_0\coloneqq \bigcap_{P'\in V(C)}P'$ is prime.
  Take pairs $(a_1,b_1),(a_2,b_2)\in S^2$ such that $(a_1,b_1)\rtimes (a_2,b_2)\in P_0$.
  For $i=1,2$, let $C_i$ be a congruence generated by $(a_i,b_i)$.
  Then we have $V(P_0)\subseteq V(C_1)\cup V(C_2)$ and hence $V(P_0)= (V(C_1)\cap V(P_0))\cup (V(C_2)\cap V(P_0))$.
  Since $Z=V(P_0)$ is irreducible, we may assume $V(P_0)=V(C_1)\cap V(P_0)\subseteq V(C_1)$.
  Thus we have
  \begin{align*}
    (a_1,b_1)\in \bigcap_{P'\in V(C_1)}P' \subseteq \bigcap_{P'\in V(P_0)}P'=P_0,
  \end{align*}
  which implies that $P_0$ is prime.

  Conversely, we assume that $Z$ is reducible and 
  there is a prime congruence $P'$ on $S$ such that $Z=V(P')$.
  By \cref{proposition:topology}, there exist congruences $C_1,C_2$ such that
  $V(P')=V(C_1\rtimes C_2)$ and $V(C_1),V(C_2)\subsetneq V(P')$.
  Then there exist $(a_i,b_i)\in C_i\setminus P'$ for $i=1,2$.
  Since $P'\in V(P')=V(C_1\rtimes C_2)$ holds, we have 
  $(a_1,b_1)\rtimes (a_2,b_2)\in P'$, which contradicts the assumption that $P'$ is prime.
\end{proof}

In the remainder of this subsection, 
we collect several technical propositions.
A \emph{topological embedding}
is an injective continuous map $f:X\to Y$ such that the induced map $X\to f(X)$ is a homeomorphism.

\begin{lemma}\label{lemma:phiphi}
  Let $\phi:S_1\to S_2$ be a surjective homomorphism of semirings and $P$ be a prime congruence on $S_1$.
  If $\Ker\phi\subseteq P$, then $\phi_{*}P$ is prime and $\phi^{*}\phi_{*}P=P$.
\end{lemma}

\begin{proof}
  Assume $\Ker\phi\subseteq P$. 
  Let $(a,b),(c,d)\in S_2\times S_2$ be pairs satisfying $(a,b)\rtimes(c,d)\in \phi_{*}P$. 
  Since $\phi$ is surjective, there exist $a',b',c',d'\in S_1$ such that
  \[\phi(a')=a,\ \phi(b')=b,\ \phi(c')=c,\ \phi(d')=d.\]
  By \eqref{eq:surj push-out}, there exists $(s,t)\in P$ such that 
  \begin{align*}
    \phi(a')\phi(c')+\phi(b')\phi(d')=\phi(s),\\
    \phi(a')\phi(d')+\phi(b')\phi(c')=\phi(t).
  \end{align*}
  By the assumption $\Ker\phi\subseteq P$, we have $(a',b')\rtimes (c',d')\in P$.
  Since $P$ is prime, $(a',b')\in P$ or $(c',d')\in P$.
  Therefore we have $(a,b)\in\phi_{*}P$ or $(c,d)\in \phi_{*}P$, so $\phi_{*}P$ is prime.

  One inclusion $P\subseteq \phi^{*}\phi_{*}P$ is clear from the definition.
  To show the opposite inclusion, take $(a,b)\in\phi^{*}\phi_{*}P$. 
  Then there exists a pair $(a',b')\in P$ such that $(\phi(a),\phi(b))=(\phi(a'),\phi(b'))$.
  Since $\Ker\phi\subseteq P$, we have $(a,b)\in P$.
\end{proof}

\begin{proposition}\label{proposition:quotient induces inj}
  If $\phi\colon S_1\to S_2$ is a surjective semiring homomorphism, 
  then 
  \[\phi^{*}\colon \CSpec S_2\to \CSpec S_1,\quad P_2\mapsto \phi^{*}P_2\] 
  is a topological embedding and its image is a closed subset
  \[V(\Ker\phi)=\{P\in\CSpec S\mid \Ker\phi\subseteq P\}.\]
\end{proposition}

\begin{proof}
  First note that $\phi_{*}\phi^{*}$ is the identity map on $\CSpec S_2$.
  Namely, for $P\in\CSpec S_2$, we show that $\phi_{*}\phi^{*}P=P$. 
  One inclusion $\phi_{*}\phi^{*}P\subseteq P$ is clear.
  The other inclusion is implied from the surjectivity of $\phi$.
  In particular, $\phi^{*}$ is injective.
  
  We next show that the image of $\phi^{*}$ is equal to $V(\Ker\phi)$.
  Take $P\in\CSpec S_1$. 
  If $P$ is in the image, then there exists $P'\in\CSpec S_2$ such that $\phi^{*}P'=P$.
  Then $\Ker\phi=\phi^{*}\Delta_{S_2}\subseteq \phi^{*}P'=P$ holds.
  Conversely, assume $\Ker \phi\subseteq P$.
  By \cref{lemma:phiphi}, $P$ is the image of $\phi_{*}P\in\CSpec S_2$. 

  Finally, we show that $\phi_{*}:V(\Ker\phi)\to\CSpec S_2$ is continuous.
  For a closed set $V(C)$ of $\CSpec S_2$, we have
  \begin{align*}
    (\phi_{*})^{-1}(V(C))&=\{P\in V(\Ker\phi)\mid C\subseteq \phi_{*}P\}\\
    &=\{P\in V(\Ker\phi)\mid \phi^{*}C\subseteq P\}\\
    &=V(\phi^{*}C)\cap V(\Ker\phi).
  \end{align*}
  The second equality follows from \cref{lemma:phiphi} and the surjectivity of $\phi$.
  Thus $\phi_{*}$ is continuous and hence $\phi^{*}$ is a topological embedding.
\end{proof}

Let $T$ be a multiplicatively closed subset of a semiring $S$.
In the following, we always assume $0_S\notin T$ and $1_S\in T$.
We define an equivalence relation $\sim$ on $S\times T$ as follows:
\[(s,t)\sim(s',t')\overset{\mathrm{def}}{\iff} st'u=s'tu\text{  for some $u\in T$.}\]
We denote by $\frac{s}{t}$ the equivalence class of $(s,t)$.
The set $T^{-1}S\coloneqq (S\times T)/\sim$ is a semiring with the following operation.
\begin{align*}
  &\frac{s}{t}+\frac{s'}{t'}=\frac{st'+s't}{tt'}\\
  &\frac{s}{t}\cdot\frac{s'}{t'}=\frac{ss'}{tt'}
\end{align*}
There is a canonical homomorphism of semirings $S\to T^{-1}S,s\mapsto\frac{s}{1}$.
We call this homomorphism the \emph{localization} of $S$ by $T$.

\begin{lemma}\label{lemma:localcong}
  Let $\phi:S\to T^{-1}S$ be the localization.
  For a congruence $C$ on $S$, $\phi_{*}C$ is equal to 
  \[C'\coloneqq \left\{\left(\frac{a}{t},\frac{b}{t}\right)\;\middle|\; (a,b)\in C,t\in T\right\}.\]
\end{lemma}

\begin{proof}
  Since $\phi(C)\subseteq C'\subseteq \phi_{*}C$ clearly holds, 
  it is enough to show that $C'$ is a congruence.
  In particular, it suffices to show the transitivity of $C'$.

  Assume that $(a/t,b/t),(a'/t',b'/t')\in C'$ and $b/t=a'/t'$, where $(a,b),(a',b')\in C$ and $t,t'\in T$.
  Then there exists $u\in T$ such that $bt'u=a'tu$. 
  By the transitivity of $C$, we have $(at'u,b'tu)\in C$.
  Thus $(a/t,b'/t')=(\frac{at'u}{tt'u},\frac{b'tu}{tt'u})\in C'$.
\end{proof}

\begin{lemma}\label{lemma:local phiphi}
  Let $\phi:S\to T^{-1}S$ be the localization and $C$ a congruence on $T^{-1}S$.
  Then $\phi_{*}\phi^{*}(C)=C$ holds.
\end{lemma}

\begin{proof}
  The inclusion $\phi_{*}\phi^{*}(C)\subseteq C$ is straightforward.
  We show the opposite inclusion.

  Take $(a,b)=(a_1/a_2,b_1/b_2)\in C$.
  Since $(\phi(a_1b_2),\phi(a_2b_1))=(a_1b_2/1,a_2b_1/1)\in C$, 
  we have $(a_1b_2,a_2b_1)\in\phi^{*}(C)$.
  Therefore we have $(a_1b_2/1,a_2b_1/1)\in \phi\phi^{*}(C)\subseteq \phi_{*}\phi^{*}(C)$.
  Thus $(a_1/a_2,b_1/b_2)\in\phi_{*}\phi^{*}(C)$ holds.
\end{proof}

\begin{lemma}\label{lemma:push is prime}
  Let $\phi:S\to T^{-1}S$ be the localization.
  If $P\in\CSpec S$ satisfies $P\cap (T\times \{0_S\})=\emptyset$, then $\phi_{*}P$ is prime.
\end{lemma}

\begin{proof}
  Assume $(a_1/a_2,b_1/b_2)\rtimes (c_1/c_2,d_1/d_2)\in\phi_{*}P$.
  Then there exist $(p,q)\in P$ and $t\in T$ such that
  \begin{align*}
    \frac{a_1c_1b_2d_2+b_1d_1a_2c_2}{a_2b_2c_2d_2}=\frac{p}{t},\\
    \frac{a_1d_1b_2c_2+b_1c_1a_2d_2}{a_2b_2c_2d_2}=\frac{q}{t}
  \end{align*}
  Since $(a_1c_1b_2d_2+b_1d_1a_2c_2,a_1d_1b_2c_2+b_1c_1a_2d_2)=(a_1b_2,a_2b_1)\rtimes (c_1d_2,c_2d_1)$,
  there exists $u\in T$ such that 
  $(ut,0)\rtimes (a_1b_2,a_2b_1)\rtimes (c_1d_2,c_2d_1)=(ua_2b_2c_2d_2,ua_2b_2c_2d_2)\cdot(p,q)\in P$.
  Since $(T\times \{0_S\})\cap P=\emptyset$ and $P$ is prime, 
  either $(a_1b_2,a_2b_1)\in P$ or $(c_1d_2,c_2d_1)\in P$.
  Thus we have $(a_1/a_2,b_1/b_2)\in \phi_{*}P$ or $(c_1/c_2,d_1/d_2)\in \phi_{*}P$.
\end{proof}

We define a subset of $\CSpec S$ by
\begin{align}
  \begin{split}\label{eq:D(M)}
    D(T) &\coloneqq \{P\in\CSpec S\mid P\cap (T\times \{0_S\})=\emptyset\}\\
    &=\{P\in\CSpec S\mid (t,0_S)\not\in P\text{ for all }t\in T\}.
  \end{split}
\end{align}

\begin{proposition}\label{proposition:loc induces inj}
  Let $T$ be a multiplicatively closed subset of a semiring $S$ and $\phi:S\to T^{-1}S$ the localization.
  Then $\phi^{*}:\CSpec T^{-1}S\to \CSpec S$ is a topological embedding onto $D(T)$.
  If $T$ is finitely generated as a monoid, then the image of $\phi^{*}$ is an open subset of $\CSpec S$.
\end{proposition}

\begin{proof}
  Take $P\in \phi^{*}(\CSpec T^{-1}S)$. Then there exists $P'\in\CSpec T^{-1}S$ such that $P=\phi^{*}P'$.
  Assume that $(T\times \{0_S\})\cap P\neq\emptyset$. 
  Take $(t,0)\in (T\times \{0_S\})\cap P$.
  By \cref{lemma:local phiphi}, $(t/1,0/1)\in \phi_{*}P=\phi_{*}\phi^{*}P'=P'$ holds. 
  Since $t\in T$ is invertible in $T^{-1}S$, we have $(1,0)\in P'$, 
  contradicting the assumption that $P'$ is prime.
  Thus $\Im\phi^{*}\subseteq D(T)$.

  Conversely, assume that $P$ satisfies $(T\times \{0_S\})\cap P=\emptyset$.
  It is enough to show that $\phi^{*}\phi_{*}(P)=P$ from \cref{lemma:push is prime}.
  One inclusion $P\subseteq \phi^{*}\phi_{*}P$ is clear.
  To show the opposite inclusion, take $(a,b)\in \phi^{*}\phi_{*}(P)$.
  Then there exist $(c,d)\in P$ and $t\in T$ such that $(a/1,b/1)=(c/t,d/t)$ from \cref{lemma:localcong}.
  Thus $(tt',0)\rtimes (a,b)=(ct',dt')\in P$ holds for some $t'\in T$. 
  Since $P$ is prime, we have $(a,b)\in P$.

  Thus $\phi^{*}$ and $\phi_{*}$ are bijections between $\CSpec T^{-1}S$ and $D(T)$.
  The continuity of $\phi_{*}:D(T)\to\CSpec T^{-1}S$ is shown as follows:
  For a closed set $V(C)$ of $\CSpec T^{-1}S$, we have
  \begin{align*}
    (\phi_{*})^{-1}(V(C))&=\{P\in D(T)\mid C\subseteq \phi_{*}P\}\\
    &=\{P\in D(T)\mid \phi^{*}C\subseteq P\}\\
    &=V(\phi^{*}C)\cap D(T).
  \end{align*}
  The second equality follows from \cref{lemma:local phiphi} and the surjectivity of $\phi$.

  Finally assume that $T$ is generated by finitely many elements $t_1,\dots,t_k$.
  Let $C_{t}$ be a congruence generated by a single pair $(t,0_S)$.
  Then the complement of the image $\phi^{*}(\CSpec T^{-1}S)$ is equal to
  \[\{P\in\CSpec S\mid P\cap (T\times \{0_S\})\neq \emptyset\} = \bigcup_{t\in T} V(C_t)=\bigcup_{i=1}^{k} V(C_{t_i}),\]
  where the second equality follows from the definition of prime congruences.
  Thus the image of $\phi^{*}$ is open.
\end{proof}

\begin{proposition}\label{proposition:used where}
  Let $C$ be a congruence on a semiring $S$ and $T$ a multiplicatively closed subset of $S$ which 
  satisfies $C\cap (T\times \{0_S\})=\emptyset$.
  Then there is an isomorphism $T^{-1}S/(C\cdot T^{-1}S)\cong \pi_C(T)^{-1}(S/C)$,
  where $\pi_C\colon S\to S/C$ is the quotient.
  Furthermore, $\CSpec\pi_C(T)^{-1}(S/C)\to\CSpec S$ is a topological embedding onto $V(C)\cap D(T)$.
\end{proposition}

\begin{proof}
  Since the image of $t$ in $\pi_C(T)^{-1}(S/C)$ is invertible for any $t\in T$, 
  then there exists a homomorphism $\phi:T^{-1}S\to \pi_C(T)^{-1}(S/C)$, which is surjective from the construction.

  For a pair $(s_1/t_1,s_2/t_2)\in T^{-1}S\times T^{-1}S$, 
  \[(s_1/t_1,s_2/t_2)\in\Ker\phi\iff  \exists u\in T, (s_1t_2u,s_2t_1u)\in C\iff (s_1/t_1,s_2/t_2)\in C\cdot (T^{-1}S)\]
  holds by \cref{lemma:localcong}.
  Thus, applying \cref{proposition:thm of hom} to $\phi$, we obtain the desired isomorphism.

  Let $\psi:S\to \pi_C(T)^{-1}(S/C)$ be the composition map.
  Then
  \begin{align*}
    \Im\psi^{*}&=\Im(\CSpec \pi_C(T)^{-1}(S/C)\to\CSpec(S/C)\to\CSpec S)\\
    &=\pi_C^{*}\{P'\in\CSpec(S/C)\mid (\pi_C(t),0_S)\notin P'\text{ for all }t\in T\}\\
    &=\{P\in\CSpec S\mid C\subseteq P,(\pi_C(t),0_S)\notin (\pi_{C})_{*}P\text{ for all }t\in T\}.
  \end{align*}
  The last equality follows from the proof of \cref{proposition:quotient induces inj}.
  Thus it is enough to show that for any $P\in V(C)$ and $t\in T$,
  \[(\pi_C(t),0_{S/C})\notin (\pi_{C})_{*}P \text{  if and only if  } (t,0_S)\notin P\]
  holds. The only if part is clear. The if part follows from \cref{lemma:phiphi}.
\end{proof}

\subsection{$\Bb$-algebras} \label{subsection:2.2}

In the following, we focus on additively idempotent semirings.
Then the addition is denoted by $\oplus$, the multiplication is denoted by $\odot$,
and the additive and multiplicative identity elements are often denoted by $-\infty$ and 0, respectively.
Note that we often omit the multiplication symbol $\odot$.

Recall that the \emph{boolean algebra} $\Bb=(\{-\infty,0\},\oplus,\odot)$ 
is the additively idempotent semifield which consists of only the identity elements $-\infty,0$.
An additively idempotent semiring is called a \emph{$\Bb$-algebra}. 
Note that $S$ is a $\Bb$-algebra if and only if 
there exists a semiring homomorphism from $\Bb$ to $S$. 
Such homomorphism $\Bb\to S$ is uniquely determined and injective.

A $\Bb$-algebra has a natural partial order defined by:
\begin{align}\label{eq:B-alg poset}
  a\geq b \overset{\mathrm{def}}{\iff} a\oplus b=a.
\end{align}
If this is a total order, the $\Bb$-algebra is said to be \emph{totally ordered}.

The following assertion is crucial for the discussion in \cref{section:B[M]cong}.

\begin{proposition}[{\cite[Corollary 2.6]{joo2026varieties}}]\label{proposition:JM}
  Let $S$ be a $\Bb$-algebra and $P$ be a congruence on $S$.
  Then $P$ is prime if and only if $S/P$ is totally ordered and cancellative. 
\end{proposition}

\subsection{$\Tbb$-algebras}\label{subsection:2.3}

Recall that $\Tbb=\Rbb\cup\{-\infty\}$ is an additively idempotent semifield 
with the addition $\oplus=\operatorname{max}$ and the multiplication $\odot=+$.
A \emph{$\Tbb$-algebra} is a pair of a semiring $S$ and 
an injective homomorphism $\iota\colon \Tbb\to S$, called the \emph{$\Tbb$-algebra structure} of $S$.
For a $\Tbb$-algebra $S$, we often regard the semifield $\Tbb$ as a subset of $S$.

\begin{definition}\label{definition:P over T}
  Let $S$ and $S'$ be $\Tbb$-algebras.
  A homomorphism $\phi\colon S\to S'$ is called a \emph{homomorphism of $\Tbb$-algebras}
  if it is compatible with the $\Tbb$-algebra structures $\iota\colon \Tbb\to S$ and $\iota\colon\Tbb\to S'$,
  that is, $\phi\circ \iota=\iota'$ holds.
  A congruence $C$ on $S$ is said to be 
  \begin{itemize}
    \item \emph{lying over $\Tbb$} if the composition $\Tbb\to S\to S/C$ is injective, and
    \item \emph{geometric} if the composition $\Tbb\to S\to S/C$ is an isomorphism. 
  \end{itemize}
  Note that if $C$ is lying over $\Tbb$, then the quotient $S/C$ can be seen as a $\Tbb$-algebra 
  and the quotient map $S\to S/C$ can be seen as a homomorphism of $\Tbb$-algebras.
  The set of all prime congruences on $S$ lying over $\Tbb$ is denoted by $\CSpec_{\Tbb}S$, 
  and the set of all geometric congruences on $S$ is denoted by $\CSpec^{\mathrm{geo}}S$.
\end{definition}

\begin{remark}
  We introduce the notion above because the dimension of semirings often exhibits undesirable behavior. 
  For example, $\dim\lau$($=\dim \CSpec \lau$) is equal to $n+1$, not $n$.
  This is due to the fact that a surjective map from a semifield need not be injective.
  However, as we will see in \cref{cor:good dimension}, 
  the dimension of $\CSpec_{\Tbb}\lau$ is equal to $n$. 
  For this reason, in \cref{section:toric}, we will mainly work with $\CSpec_{\Tbb}S$ 
  instead of $\CSpec S$.
\end{remark}

By \cref{proposition:JM}, any geometric congruence is prime.
Thus we have
\[\CSpec^{\mathrm{geo}}S\subseteq \CSpec_{\Tbb} S\subseteq \CSpec S.\]

It is easy to show that the pull-back (see \eqref{eq:pull-back}) of a prime congruence over $\Tbb$ 
by a homomorphism of $\Tbb$-algebras is also a prime congruence over $\Tbb$.
Thus a homomorphism $S_1\to S_2$ of $\Tbb$-algebras induces
a continuous map $\CSpec_{\Tbb}S_2\to \CSpec_{\Tbb}S_1$.

We will see some properties of the subset $\CSpec_{\Tbb} S$.
For a $\Tbb$-algebra $S$, the $\Tbb$-algebra structure $\iota\colon \Tbb\to S$ gives a continuous map
\begin{align}\label{eq:fiberation}
  \phi=\iota^{*}\colon\CSpec S\to \CSpec \Tbb.
\end{align}
The base space $\CSpec\Tbb$ contains two points: 
the diagonal $\Delta$, and the kernel $P_0$ of the surjective map $\Tbb\to \Bb$.
Note that the singleton $\{\Delta\}$ is an open subset of $\CSpec\Tbb$
whose complement is a closed subset $V(P_0)=\{P_0\}$.

\begin{proposition}\label{proposition:T open}
  For the continuous map of \eqref{eq:fiberation}, 
  the subset $\CSpec_{\Tbb} S$ is equal to the fiber $\phi^{-1}(\Delta)$ of the point $\Delta$.
  In particular, $\CSpec_{\Tbb} S$ is an open subset of $\CSpec S$.
\end{proposition}

\begin{proof}
  For $P\in \CSpec S$, let $S'$ be the image of $\Tbb$ by the quotient map $\pi_P$. 
  Then the pull-back $\iota^{*}P$ induces the composition map
  \[\Tbb\xrightarrow{\iota} S \xrightarrow{\pi_P} S'.\]
  Thus we have
  \begin{align*}
    \phi^{-1}(\Delta)&=\{P\in\CSpec S \mid \phi(P)=\iota^{*}P=\Delta \}\\
    &=\{P\in\CSpec S \mid \iota^{*}P \text{ induces }\Tbb\cong S'\}\\
    &=\CSpec_{\Tbb}S
  \end{align*}

  The latter assertion holds since the singleton $\{\Delta\}$ is an open subset.
\end{proof}

With respect to the fibration $\CSpec S\to\CSpec \Tbb$, 
the special fiber $\phi^{-1}(P_0)$ corresponds to what is often called the \emph{recession} in tropical geometry
(see also \cref{proposition:special fiber}).

%separatedness of affine scheme

\begin{definition}\label{definition:limit}
  Let $S$ be a $\Tbb$-algebra and $P\in\CSpec_{\Tbb} S$.
  Define a map 
  \[\lim_{P}\colon S\to \Tbb\cup \{+\infty\}=\Rbb\cup \{\pm\infty\} \] 
  as follows:
  If $\pi_P(s)\in S/P$ is larger than any element of $\Rbb\subseteq S/P$, then define $\lim_P s$ as $+\infty$.
  If $\pi_P(s)\in S/P$ is smaller than any element of $\Rbb\subseteq S/P$, then define $\lim_P s$ as $-\infty$.
  Otherwise, $\lim_Ps$ is defined as the supremum of the subset 
  $\{t\in \Rbb\mid t\leq \pi_P(s) \text{ in }S/P\}$ of $\Rbb$.
  Note that $\lim_Ps$ is also the infimum of the subset 
  $\{t\in \Rbb\mid t\geq \pi_P(s) \text{ in }S/P\}$ of $\Rbb$.
  We call $\lim_Ps$ the \emph{limit} of $s$ under $P$.
\end{definition}

\begin{lemma}\label{lemma:limit equal}
  If $P_1,P_2\in\CSpec_{\Tbb} S$ and $P_1\subseteq P_2$, 
  then $\lim_{P_1}s=\lim_{P_2}s$ for any $s\in S$.
\end{lemma}

\begin{proof}
  Assume $P_1\subseteq P_2$. Then there is a surjective homomorphism $S/P_1\to S/P_2$.
  Since these are totally ordered $\Tbb$-algebras, then the following are equivalent:
  \begin{itemize}
    \item $t-\epsilon < \pi_{P_1}(s) < t+\epsilon$ for all $\epsilon\in \Rbb_{>0}$,
    \item $t-\epsilon < \pi_{P_2}(s) < t+\epsilon$ for all $\epsilon\in \Rbb_{>0}$.
  \end{itemize}
  Thus $\lim_{P_1}s=\lim_{P_2}s$ holds for any $s\in S$ if these limits are finite.
  The same holds if these limits are $+\infty$ or $-\infty$.
\end{proof}

\begin{lemma}\label{lemma:geom condition}
  Let $P$ be a prime congruence on a $\Tbb$-algebra $S$ lying over $\Tbb$.
  Then there exists a geometric congruence on $S$ containing $P$
  if and only if $\lim_P s\in\Tbb$ holds for any $s\in S$.
\end{lemma}

\begin{proof}
  Assume that there exists a geometric congruence $P_0$ which contains $P$.
  By the definition of the limit, we have $\lim_{P_0} s=\pi_{P_0}(s)\in\Tbb$ for any $s\in S$.
  By \cref{lemma:limit equal}, we have $\lim_{P} s\in \Tbb$ for any $s\in S$.

  Conversely, assume that $\lim_P s\in\Tbb$ holds for any $s\in S$.
  Then the map $\lim_P:S\to \Tbb$ is a homomorphism of semirings 
  such that the composition with $\Tbb\hookrightarrow S$ is identity.
  Since $\lim_P$ factors through $S\to S/P$, 
  the kernel congruence of $\lim_P$ is a geometric congruence containing $P$.
\end{proof}

We would like to call the property of \cref{proposition:abstract geom} 
the \emph{separatedness} of $\CSpec_{\Tbb} S$.
In this paper, this is not defined in a precise way, 
but it is motivated by the valuative criterion for separatedness of usual schemes 
(see \cite[Theorem 4.3]{hartshorne2013algebraic} for the details).

\begin{proposition}\label{proposition:abstract geom}
  For a prime congruence $P$ on a $\Tbb$-algebra $S$, 
  there is at most one geometric congruence which contains $P$.
\end{proposition}

\begin{proof}
  Assume that there exists a geometric congruence $P_0$ which contains $P$.
  In particular, $P$ is lying over $\Tbb$.
  Fix $s\in S$. By the definition of $\lim_Ps$, we have
  $\lim_Ps-\epsilon < \pi_P(s) < \lim_Ps+\epsilon$ for any $\epsilon\in \Rbb_{>0}$.
  Sending to $S/P_0$, we have $\lim_Ps-\epsilon < \pi_{P_0}(s) < \lim_Ps+\epsilon$ for any $\epsilon\in \Rbb_{>0}$.
  Since $P_0$ is geometric, $\pi_{P_0}(s)=\lim_Ps$ holds in $S/P_0$.
  Thus $P_0$ is equal to the kernel congruence of $\lim_P$.
\end{proof}

\begin{remark}
  Any geometric congruence is maximal in $\CSpec_{\Tbb}S$.
  However, as we will see in \cref{example:affine line},
  there is a prime congruence on a $\Tbb$-algebra which is maximal in $\CSpec_{\Tbb}\Tbb[X]$ but not geometric.
  This corresponds to the fact that the usual affine line is not proper 
  (see \cite[Theorem 4.7]{hartshorne2013algebraic} for the details).
\end{remark}

\section{Tropical algebras with monomial sets}\label{section:B[M]}

We now proceed gradually toward a tropical setting.
When we consider tropical algebras, additive idempotency is
one of the minimal requirements which we should impose on algebras.
However, for simplicity, 
we impose not only additive idempotency but also a \emph{monomial structure} on algebras.

\subsection{Ordered monoids}\label{section:o-monoid}

Recall that a \emph{(commutative) monoid} is 
a set equipped with an associative and commutative binary operation $+$ and an identity element 0.
We write the monoid operation additively, 
although the operation $+$ will be identified 
with the product $\odot$ of a $\Bb$-algebra in the subsequent discussion.

An \emph{ordered monoid} is
a pair $(\Msf,\geq)$ consisting of a commutative monoid $\Msf=(\Msf,+)$ 
together with a partial order $\geq$ satisfying
\[a\geq b\implies a+c\geq b+c,\]
for any $a,b,c\in \Msf$.
The pair $(\Msf,\geq)$ is often denoted simply by $\Msf$.
If every element of an ordered monoid $(\Msf,\geq)$ is invertible, then we call it a \emph{ordered group}.
A monoid homomorphism $\phi:(\Msf_1,\geq_1)\to (\Msf_2,\geq_2)$ 
between ordered monoids is an \emph{o-homomorphism} if for any $a,b\in \Msf_1$,
\[a\geq_1 b\implies \phi(a)\geq_2\phi(b).\]
If an o-homomorphism $\phi$ is a bijective map whose inverse map is also an o-homomorphism,
then $\phi$ is called an \emph{o-isomorphism}.
If there exists an o-isomorphism between two ordered monoids $\Msf_1$ and $\Msf_2$,
then $\Msf_1$ and $\Msf_2$ are said to be \emph{o-isomorphic}.

The order on $\Msf$ is said to be \emph{trivial} if no two distinct elements of $\Msf$ are comparable.
We often regard any monoid as an ordered monoid with the trivial order.

For an ordered monoid $(\Msf,\geq)$, 
every submonoid $\Msf'$ of $\Msf$ inherits the induced order, 
also denoted by $\geq$.
Unless otherwise stated, 
we equip a submonoid of an ordered monoid $\Msf$, 
such as the kernel or image of an o-homomorphism, with the order induced from $\Msf$.

\medskip

We next define the \emph{localization} of an ordered monoid, 
in particular, the \emph{group completion} of an ordered monoid.
Let $\Msf$ be an ordered monoid and $T$ be a multiplicatively closed subset of $\Msf$.
Recall that the \emph{localization} of $\Msf$ by $T$ is a monoid $T^{-1}\Msf=(\Msf\times T)/\sim$ 
where the relation $\sim$ is given as follows:
\[(a,b)\sim (c,d) \overset{\mathrm{def}}{\iff} \exists t\in T,\ a+d+t= b+c+t.\]
We often denote by $a-b$ the equivalent class of $(a,b)$.
Now we also endow $T^{-1} \Msf$ with the order induced by $\Msf$:
\[a-b\geq c-d \overset{\mathrm{def}}{\iff} \exists t\in T,\ a+d+t\geq b+c+t.\]
Note that the natural homomorphism $\Msf\to T^{-1}\Msf$ is an o-homomorphism. 
In particular, we denote $\Msf^{-1}\Msf$ by $\Msf^{\mathrm{gp}}$, called the \emph{group completion} of $\Msf$.

\medskip

Before giving examples, we introduce some terminology.
Let $X_1,\dots,X_n$ be formal variables.
Throughout this paper, we use the following notation. 
\begin{align}
  \begin{split}\label{monomial notation}
    &\Mp=\{u_1X_1+u_2X_2+\dots +u_nX_n\mid u_1,\dots,u_n\in \Nbb\},\\
    &M=\{u_1X_1+u_2X_2+\dots +u_nX_n\mid u_1,\dots,u_n\in \Zbb\},\\
    &\Rbb\times \Mp=\{(r,u_1X_1+u_2X_2+\dots +u_nX_n)\mid r\in\Rbb,\ u_1,\dots,u_n\in \Nbb\},\\
    &\Rbb\times M=\{(r,u_1X_1+u_2X_2+\dots +u_nX_n)\mid r\in\Rbb,\ u_1,\dots,u_n\in \Zbb\}.
  \end{split}
\end{align}
By abuse of notation, we often write $u_1X_1+\dots+u_nX_n$ as $u\cdot X$ or $X^u$. 
Similarly, we write $(r,u_1X_1+\dots +u_nX_n)$ as $r+u\cdot X$ or $rX^u$.
While we equip $\Mp$ and $M$ with the trivial order,
we define orders on $\Rbb\times \Mp$ and $\Rbb\times M$ as follows:
\begin{align*}
  rX^{u}\geq r'X^{u'}\overset{\mathrm{def}}{\iff} r\geq r' \text{ and }u=u'.
\end{align*}
These ordered monoids serve as the guiding examples throughout the paper.

\begin{example} \label{example:sheaf}
  A real-valued function $f$ on $\Rbb^n$ is said to be \emph{integral affine}
  if it has the form
  \[f(X_1,\dots,X_n)=c+u_1X_1+\dots+u_nX_n\]
  where $c \in \Rbb$ and $u_1,\dots ,u_n \in \Zbb$. 
  Then $\Rbb\times M$ can be seen as the set of all integral affine functions on $\Rbb^n$.

  We endow $\Rbb^n$ with the Euclidean topology.
  We say that a real-valued function $f$ on an open subset $U$ of $\Rbb^n$ is 
  \emph{locally integral affine}
  if for each point $x\in U$, there exists an open neighborhood $U_x\subseteq U$ of $x$
  and an integral affine function $f_x$ such that $f_x|_{U_x}=f|_{U_x}$.
  
  We will denote by $\Aff_{\Rbb^n}$ the sheaf of locally integral affine functions on $\Rbb^n$.
  As a sheaf of groups, any section of the sheaf is isomorphic to $\Rbb\times M$ 
  and any restriction map is an isomorphism.
  However, the sheaf $\Aff_{\Rbb^n}$ becomes nontrivial when regarded as a \emph{sheaf of ordered groups}
  by equipping each section $\Aff_{\Rbb^n}(U)$ with the natural order: for $f,g\in\Aff_{\Rbb^n}(U)$,
  \[f\geq g \overset{\mathrm{def}}{\iff} f(x)\geq g(x) \text{ for all }x\in U.\]
  In particular, the ordered group $\Aff_{\Rbb^n}(\Rbb^n)$ is o-isomorphic to $\Rbb\times M$
  in the sense of \cref{monomial notation}.
\end{example}

\begin{example}\label{example:stalk2}
  For $x=(x_1,\dots,x_n)\in\Rbb^n$,
  define a subspace $H_x$ of $\Rbb\times M$ by
  \[H_x\coloneqq \Rbb\{ (-x_1,X_1),\dots,(-x_n,X_n) \}.\]
  Then $H_x$ defines the surjective homomorphism 
  $\pi_x:\Rbb\times M\to (\Rbb\times M)/H_x\cong\Rbb$.
  This map can be seen as the map that substitutes $X_i=x_i$ for each $i$.

  Now we define a partial order $\geq_x$ on $\Rbb\times M$. 
  For $f,g\in \Rbb\times M$, 
  \[f\geq g \overset{\mathrm{def}}{\iff} \pi_x(f)>\pi_x(g) \text{ or }f=g\]
  Then the ordered group $(\Rbb\times M,\geq_x)$ is o-isomorphic to the stalk $\Aff_{\Rbb^n,x}$ at $x$.
\end{example}

\subsection{$\Bb$-algebras associated to ordered monoids} \label{subsection:B[M]}

We work under the setting of \cref{subsection:2.2}: 
an additively idempotent semiring is called a $\Bb$-algebra and 
its operations are denoted by $\oplus$ and $\odot$.
By the following definition, we construct a $\Bb$-algebra from an ordered monoid $\Msf=(\Msf,+)$,
identifying the monoid operation $+$ with the product $\odot$ of the $\Bb$-algebra.

\begin{definition}\label{definition:B-alg}
  For an ordered monoid $\Msf=(\Msf,+,\geq)$,
  we consider the following set:
  \[\Bb[\Msf]\coloneqq \{A \mid \text{$A$ is a finite subset of } \Msf \}/\sim,\]
  where the relation $\sim$ is given as follows:
  \[A_1\sim A_2 \overset{\mathrm{def}}{\iff} \text{the set of all maximal elements in $A_1$ is equal to that of $A_2$}.\]
  We denote the equivalence class of $A$ by $\sum_{m\in A} m$.
  If $A=\{m_0\}$ is a singleton, then we simply denote $\sum_{m\in A} m$ by $m_0$.
  One checks that the following operations are well-defined:
  \begin{align*}
    \sum_{m\in A_1} m\oplus \sum_{m\in A_2} m &= \sum_{m\in A_1\cup A_2} m,\\
    \sum_{m\in A_1} m\odot \sum_{m\in A_2} m &= \sum_{m_1\in A_1,m_2\in A_2} (m_1+m_2).
  \end{align*}
  Then $\Bb[\Msf]$ is a semiring 
  with additive identity $-\infty\coloneqq \sum_{m\in \emptyset}m$ and multiplicative identity $0=\sum_{m\in \{0\}} m$.
\end{definition}

For an element $f=\sum_{m\in A}m\in\Bb[\Msf]$, 
the set of all maximal elements of $A$ is called the \emph{support} of $f$, denoted $\supp(f)$.
Then $\supp(f)$ is uniquely determined, and two distinct elements of $\Bb[\Msf]$ have distinct supports. 
In other words, $\supp$ gives the following bijection:
\begin{align}\label{eq:bij}
  \Bb[\Msf]\xrightarrow{1:1}\{A\subseteq \Msf\mid A\text{ is finite and any two elements in $A$ are not comparable}\},\, f\mapsto\supp(f).
\end{align}
In particular, there is an injective map $\Msf\to \Bb[\Msf]$.
So we often treat $\Msf$ as a submonoid of $\Bb[\Msf]$.
Recall that a $\Bb$-algebra has a natural partial order (see \eqref{eq:B-alg poset}).
Note that the original order of $\Msf$ and the order induced by that of $\Bb[\Msf]$ are equivalent.
We also note that any o-homomorphism $\Msf_1\to \Msf_2$ induces a semiring homomorphism $\Bb[\Msf_1]\to\Bb[\Msf_2]$.

\begin{remark}
  By definition, $\Bb[\Msf]$ is a semiring with no additional structure.
  If one needs explicitly retain the monomial structure $\Msf$, 
  it is natural to describe it as a blueprint in \cite{lorscheid2012geometry}.
  Let $C$ be the kernel congruence of the natural homomorphism $\Nbb[\Msf]\to \Bb[\Msf]$.
  Then a pair of the monoid $\Msf$ (with its order forgotten) and the pre-addition $C$ 
  is a blueprint whose associated semiring is naturally isomorphic to $\Bb[\Msf]$.
\end{remark}

\begin{example}\label{example:monomials}
  Here we use the notation of \eqref{monomial notation}.
  For example, we consider the $\Bb$-algebra associated to $\Rbb\times M$. 
  A map
  \begin{align*}
    \{A \mid \text{$A$ is a finite subset of } \Rbb\times M \}&\to \lau,\\
    A&\mapsto \sum_{(r,u)\in A} r\odot X_1^{\odot u_1}\odot\dots \odot X_n^{\odot u_n}=\sum_{(r,u)\in A}(r+u_1X_1+\dots +u_nX_n)
  \end{align*}
  induces an isomorphism $\Bb[\Rbb\times M]\cong \lau$ by \cref{proposition:thm of hom}.
  Similarly, we have 
  \begin{align*}
    &\Bb[M]\cong\Blau\\
    &\Bb[\Mp]\cong\Bpol\\
    &\Bb[\Rbb\times \Mp]\cong\pol
  \end{align*}
  In this way, we can treat some semirings arising in tropical geometry uniformly.

  Recall that $\Aff_{\Rbb^n}$ in \cref{example:sheaf} is a sheaf of ordered groups. 
  Each restriction map $\Aff_{\Rbb^n}(\Rbb^n)\to \Aff_{\Rbb^n}(U)$ 
  induces $\lau\to\Bb[\Aff_{\Rbb^n}(U)]$ and a congruence on $\lau$.
  This demonstrates that $\lau$ admits many congruences.
\end{example}

\begin{example}
  In \cref{example:stalk2}, for $x=(x_1,\dots,x_n)\in\Rbb^n$, we have seen two o-homomorphisms:
  \[H_x\subseteq \Rbb\times M\to\Aff_{\Rbb^n,x}.\]
  These induce
  \[\Bb[Z_1^{\pm},\dots,Z_n^{\pm}]\xrightarrow{\iota} \lau\xrightarrow{\pi} \Bb[\Aff_{\Rbb^n,x}],\]
  where each $Z_i$ is a formal variable that maps to $-x_i+X_i$ by $\iota$.
  Note that the image of $\iota$ is the set of all integral affine functions on $\Rbb^n$ 
  which take the value 0 at $x$, except for the constant function $-\infty$.

  For an ideal $I$ of $\lau$, let $I'$ be the ideal of $\Bb[\Aff_{\Rbb^n,x}]$ 
  generated by the image $\pi(I)$ and 
  $I''$ be the pull-back $(\iota\circ\pi)^{-1}(I')\subseteq \Bb[Z_1^{\pm},\dots,Z_n^{\pm}]$ of $I'$.
  Then
  \begin{align*}
    I'&=\left\{\sum_{u\colon a_ux^u=f(x)}a_u X^{u}\ \middle\vert\ f(X)=\sum_{u\in\Zbb^n}a_u X^{ u}\in I \right\}\\
    I''&=\left\{\sum_{u\colon a_ux^u=0}a_u X^{u}\ \middle\vert\ f(X)=\sum_{u\in\Zbb^n}a_u X^{ u}\in I,\ f(x)=0 \right\}\\
    &=\left\{\sum_{u\colon a_ux^u=f(x)}Z^{ u}\ \middle\vert\ f(X)=\sum_{u\in\Zbb^n}a_u X^{ u}\in I \right\}.
  \end{align*}
  In \cite{maclagan2018tropical} and related works, 
  the ideal $I''$ is called the \emph{initial ideal} of $I$ associated to $x$.
\end{example}

We next introduce a basic invariant of prime congruences on $\Bb[\Msf]$, 
which we call \emph{the mobile face}.

\begin{definition}\label{definition:face}
  Let $\Msf$ be an ordered monoid.
  A submonoid $\Fsf\subseteq \Msf$ is a \emph{face} 
  if it satisfies the following conditions: for any $m_1,m_2\in\Msf$,
  \begin{itemize}
    \item If $m_1+m_2\in\Fsf$, then $m_1,m_2\in\Fsf$.
    \item If $m_1\in\Fsf$ and $m_1\leq m_2$, then $m_2\in\Fsf$.
  \end{itemize}
  For a face $\Fsf\subseteq \Msf$, we denote by $\Gamma_{\Fsf}$ 
  the congruence generated by $\{(m,-\infty)\mid m\in \Msf\setminus \Fsf\}$.
\end{definition}

\begin{lemma}\label{lemma:P'}
  Let $\Fsf$ be a face of $\Msf$. Then a subset
  \[I_{\Fsf}=\left\{\sum_{m\in S}m \;\middle\vert \text{ $S\subseteq \Msf\setminus \Fsf$ is a finite subset}\right\}\]
  is an ideal of $\Bb[\Msf]$. 
  Furthermore, the congruence $\Gamma_{\Fsf}$ is equal to
  \[C_0= \left\{(i_1\oplus f,i_2\oplus f)\;\middle\vert\; i_1,i_2\in I_{\Fsf}, f\in\Bb[\Fsf]\right\}.\]
\end{lemma}

\begin{proof}
  By the definition of faces, it is immediate that the subset $I_{\Fsf}$ is an ideal.
  Since $\{(m,-\infty)\mid m\in \Msf\setminus \Fsf\}\subseteq C_0\subseteq \Gamma_{\Fsf}$,
  it is enough to show that $C_0$ is actually a congruence.
  Since $I_{\Fsf}$ is an ideal, $C_0$ is a subsemiring of $\Bb[\Msf]^2$.
  By the second condition of \cref{definition:face}, we have
  \[C_0=\{(g,h)\mid \supp(g)\cap \Fsf=\supp(h)\cap \Fsf\},\]
  which implies $C_0$ is an equivalence relation.
  Thus $C_0$ is a congruence and $\Gamma_{\Fsf}$ is equal to $C_0$.
\end{proof}

Thus, if a face $\Fsf\subseteq \Msf$ is given, 
then we obtain a closed subset $V(\Gamma_{\Fsf})$ of $\CSpec\Bb[\Msf]$.
Note that for faces $\Fsf_1,\Fsf_2$ of $\Msf$, the following are equivalent:
\begin{itemize}
  \item $\Fsf_1\subseteq \Fsf_2$,
  \item $\Gamma_{\Fsf_2}\subseteq \Gamma_{\Fsf_1}$,
  \item $V(\Gamma_{\Fsf_1})\subseteq V(\Gamma_{\Fsf_2})$.
\end{itemize}

\begin{lemma}\label{lemma:sed of prime}
  Let $C$ be a congruence on $\Bb[\Msf]$ and $\pi_C$ be the quotient map. 
  Then 
  \[\Fsf(C)=\{m\in\Msf\mid\pi_C(m)\neq-\infty\}=\{m\in\Msf \mid (m,-\infty)\notin C\}\] 
  is a face of $\Msf$.
  In particular, for a face $\Fsf$ of $\Msf$, we have $\Fsf(\Gamma_{\Fsf})=\Fsf$.
\end{lemma}

\begin{proof}
  Let $m_1,m_2$ be elements of $\Msf$.
  If $m_1\notin \Fsf(C)$, then $\pi_C(m_1+m_2)=-\infty \odot \pi_C(m_2)=-\infty$.
  Thus we have $m_1+m_2\notin \Fsf(C)$.

  Assume that $m_1\in\Fsf(C)$ and $m_1\leq m_2$. 
  If $m_2\notin \Fsf(C)$, then $-\infty=\pi_C(m_2)\geq \pi_C(m_1)\geq-\infty$, 
  which contradicts the assumption that $m_1\in\Fsf(C)$.

  The equality $\Fsf(\Gamma_{\Fsf})=\Fsf$ follows from \cref{lemma:P'}.
\end{proof}

\begin{definition}
  The face uniquely determined by $P$ 
  from \cref{lemma:sed of prime} is called the \emph{mobile face} of $P$. 
  We denote by $O_{\Fsf}$ the set of all points in $\CSpec\Bb[\Msf]$ of mobile face $\Fsf$.
\end{definition}

\begin{lemma}\label{lemma:wah}
  For a face $\Fsf$ of $\Msf$,
  there is a natural isomorphism $\Bb[\Fsf]\cong \Bb[\Msf]/\Gamma_{\Fsf}$.
\end{lemma}

\begin{proof}
  Let $\phi$ be the composition of 
  the semiring homomorphism $\Bb[\Fsf]\to \Bb[\Msf]$ induced 
  by the inclusion $\Fsf\subseteq \Msf$, 
  and the quotient homomorphism $\Bb[\Msf]\to\Bb[\Msf]/\Gamma_{\Fsf}$. 
  We will show $\phi$ is an isomorphism.

  For any $a=\sum_{m\in A}m\in \Bb[\Msf]$, set $a'=\sum_{m\in A\cap \Fsf}m\in\Bb[\Fsf]$. 
  Since $(a,a')\in C_{\Fsf}$, $\phi$ is surjective.
  Assume $a_1,a_2\in\Bb[\Fsf]$ satisfy $(a_1,a_2)\in C_{\Fsf}$.
  By \cref{lemma:P'}, it is clear that $a_1=a_2$. Thus $\phi$ is injective.
\end{proof}

The space $\CSpec\Bb[\Msf]$ can be written as a disjoint union:
\begin{align}\label{eq:B-str}
  \CSpec\Bb[\Msf]=\coprod_{\Fsf}\{P\mid \text{$P$ is of mobile face $\Fsf$}\},
\end{align}
where $\Fsf$ runs over the set of all faces of $\Msf$. 
We call this the \emph{stratification} of $\CSpec\Bb[\Msf]$.
In \cref{section:B[M]cong}, we give a precise description of the stratification.

\begin{example}
  For $a\in \Rbb$, let $\phi_a$ be a $\Tbb$-algebra homomorphism defined by
  \[\phi_a\colon\Tbb[X,Y]\to \Tbb,\ X\mapsto a,\ Y\mapsto -\infty.\]
  Then its kernel congruence $P_a=\Ker\phi_a$ is a prime congruence on $\Tbb[X,Y]$ 
  and the mobile face of $P_a$
  is equal to $\{rX^u\mid r\in \Rbb,\ u\in\Zbb_{\geq 0}\}\subseteq \Rbb\times \Mp$ where $n=2$.
\end{example}

If an ordered monoid $\Msf$ is totally ordered, 
then $\Bb[\Msf]$ is identified with $\Msf\sqcup \{-\infty\}$.
Using this identification, 
we can simply express the quotient of $\Bb[\Msf]$ by its prime congruence as follows.

\begin{lemma}\label{lemma:prime isomorphism}
  Let $P$ be a prime congruence on $\Bb[\Msf]$ whose mobile face is $\Fsf$.
  Then there is an isomorphism $\Bb[\pi_P(\Fsf)]\cong\Bb[\Msf]/P$ induced 
  by the inclusion $\Fsf\hookrightarrow\Msf$.
\end{lemma}

\begin{proof}
  Let $\pi_P\colon \Bb[\Msf]\to\Bb[\Msf]/P$ be the quotient map. 
  By \cref{proposition:JM}, the images $\pi_P(\Msf)$ and $\pi_P(\Fsf)$ are totally ordered monoids.
  Since $\Fsf=\{m\in\Msf\mid\pi_{P}(m)\neq -\infty\}$, we have
  \[\Bb[\Msf]/P=\pi_P(\Msf)\cup \{-\infty\}=\pi_P(\Fsf)\sqcup\{-\infty\}=\Bb[\pi_P(\Fsf)].\]
\end{proof}

\subsection{$\Tbb$-algebras associated to ordered monoids}\label{subsection:T[M]}

We define $\Tbb$-algebras associated to ordered monoids using \cref{definition:B-alg}.
As we have seen in \cref{example:monomials}, we can already construct some $\Tbb$-algebra
using suitable ordered monoids. 
However, in general, even a $\Bb$-algebra of the form $\Bb[\Rbb\times \Msf]$ 
may admit several $\Tbb$-algebra structures.
In order to fix a $\Tbb$-algebra structure, we introduce the following notation.

\begin{definition}\label{definition:T[M]}
  For an ordered monoid $\Msf$, we define the order of $\Rbb\times \Msf$ by
  \[(r_1,m_1)\geq (r_2,m_2) \overset{\mathrm{def}}{\iff} r_1\geq r_2 \text{ and }m_1\geq m_2.\]
  Then $\Rbb\times \Msf$ is an ordered monoid. 
  An o-homomorphism 
  \[\Rbb\to \Rbb\times \Msf,\ r\mapsto (r,0)\]
  induces $\Tbb\to \Bb[\Rbb\times \Msf]$. 
  With this homomorphism, we obtain a $\Tbb$-algebra $\Bb[\Rbb\times \Msf]$,
  denoted by $\Tbb[\Msf]$.
\end{definition}

Faces of $\Rbb\times \Msf$ take a simple form.

\begin{lemma}\label{lemma:T face}
  For any face $\Fsf$ of an ordered monoid $\Rbb\times \Msf$, 
  there exists a face $\Fsf'$ of $\Msf$ such that $\Fsf=\Rbb\times \Fsf'$.
\end{lemma}

\begin{proof}
  Let $\Fsf'$ be the image of $\Fsf$ by the projection $\Rbb\times \Msf\to \Msf$.
  It is easy to check that $\Fsf'$ is a face of $\Msf$.
  Since $(r,0)+(-r,0)=(0,0)\in \Fsf$ holds for any $r\in\Rbb$, we have $\Rbb\times \{0\}\subseteq \Fsf$.
  By the definition of faces, for any $r\in \Rbb$ and $m\in \Msf$, 
  $(r,m)\in\Fsf$ if and only if $(r,0),(0,m)\in \Fsf$.
  Thus we have $\Fsf=\Rbb\times \Fsf'$.
\end{proof}

Similarly to \eqref{eq:B-str}, the space $\CSpec\Tbb[\Msf]$ can be also stratified:
\[\CSpec\Tbb[\Msf]=\coprod_{\Fsf}\{P\mid \text{$P$ is of mobile face $\Rbb\times\Fsf$}\},\]
where $\Fsf$ runs over the set of all faces of $\Msf$. 

\medskip

Recall that the space of prime congruences on a $\Tbb$-algebra $S$ has a special open subset, 
denoted by $\CSpec_{\Tbb}S$ (see \cref{definition:P over T}).

\begin{proposition}\label{proposition:special fiber}
  For an ordered monoid $\Msf$, 
  the complement of an open subset $\CSpec_{\Tbb}\Tbb[\Msf]$ is homeomorphic to $\CSpec \Bb[\Msf]$. 
\end{proposition}

\begin{proof}
  Let $\iota\colon \Tbb\to\Tbb[\Msf]$ be the $\Tbb$-algebra structure and $P_0\in\CSpec\Tbb$ the closed point.
  The complement of $\CSpec_{\Tbb}\Tbb[\Msf]$ is a closed subset $(\iota^{*})^{-1}V(P_0)=V(\iota_{*}P_0)$.
  By \cref{proposition:quotient induces inj}, it is homeomorphic to $\CSpec(\Tbb[\Msf]/\iota_{*}P_0)$.

  The projection $\Pi:\Rbb\times \Msf\to \Msf$ is an o-homomorphism 
  which induces $\Tbb[\Msf]\to \Bb[\Msf]$. Let $C$ be its kernel congruence.
  An inclusion $\iota_{*}P_0\subseteq C$ is clear.
  To show the opposite inclusion, take a pair $(a,b)\in C$. Then $\Pi(\supp(a))$ and $\Pi(\supp(b))$ are same subsets of $\Msf$, denoted by $A$.
  We write $a,b$ by
  \[a=\sum_{m\in A}a_m m,\ b=\sum_{m\in A}b_m m,\]
  where $a_m,b_m \in \Rbb$ for any $m\in A$.
  Then we have
  \[(a,b)=\sum_{m\in A}(a_m,b_m)\odot(m,m)\in \iota_{*}P_0.\]
  Thus we have $\iota_{*}P_0=C$ and $\Tbb[\Msf]/\iota_{*}P_0\cong \Bb[\Msf]$.
\end{proof}

Thus as a set, the space $\CSpec \Tbb[\Msf]$ can be identified with 
a disjoint union of $\CSpec\Bb[\Msf]$ and $\CSpec_{\Tbb}\Tbb[\Msf]$.
We often treat them separately, as in \cref{section:toric}.

\medskip

Recall that a geometric congruence $P\in\CSpec^{\mathrm{geo}}\Tbb[\Msf]$ 
induces a $\Tbb$-algebra homomorphism $\Tbb[\Msf]\to \Tbb$ (see \cref{definition:P over T}).
This map $\Tbb[\Msf]\to \Tbb$ is determined 
by an o-homomorphism $\Msf\to \Tbb$ where we regard $\Tbb$ as a monoid with the operation $\odot$.
In other words, the following map is a bijection:
\begin{align}\label{eq:top emb}
  \Hom(\Msf,\Tbb)\to\CSpec^{\mathrm{geo}}\Tbb[\Msf], \ \phi\mapsto\Ker\overline{\phi},
\end{align}
where $\overline{\phi}:\Tbb[\Msf]\to\Tbb$ is the homomorphism of $\Tbb$-algebras 
induced by 
\[\Rbb\times \Msf \to \Rbb, \ (r,m)\mapsto r\odot \phi(m).\]
We equip the semifield $\Tbb$ with the topology such that 
\begin{align*}
  \Tbb\to \Rbb_{\geq 0}, \ r\mapsto \exp(r)
\end{align*}
is a homeomorphism.
We also equip the set $\Hom(\Msf,\Tbb)$ with the topology of pointwise convergence.
We will show that the map \eqref{eq:top emb} induces 
a topological embedding from $\Hom(\Msf,\Tbb)$ to $\CSpec_{\Tbb}\Tbb[\Msf]$.

\begin{lemma}\label{lemma:TopEmb}
  Assume that the free monoid $\Nbb^k$ is equipped with the trivial order. 
  Then the map 
  \[\iota:\Hom(\Nbb^k,\Tbb)\to\CSpec_{\Tbb}\Tbb[\Nbb^k]\]
  of \eqref{eq:top emb} is a topological embedding onto the subset $\CSpec^{\mathrm{geo}}\Tbb[\Nbb^k]$.
\end{lemma}

\begin{proof}
  By \cite[Proposition 3.4]{rabinoff2012tropical}, 
  the natural bijection $\Hom(\Nbb^k,\Tbb)\to\Tbb^k$ is a homeomorphism.
  Thus it is enough to show that the Euclidean topology on $\Tbb^k$ 
  is equal to the topology on $\Tbb^k$ whose closed sets are of the form
  \[V_{\Tbb^k}(C)\coloneqq \{x\in \Tbb^k\mid f(x)=g(x) \text{ for any }(f,g)\in C\}\]
  as $C$ runs over the set of all congruences on $\Tbb[X_1,\dots,X_k]$.
  Since it is immediate that $V_{\Tbb^k}(C)$ is a closed subset under the Euclidean topology,
  it suffices to show that
  any open subset of $\Tbb^k$ under the Euclidean topology is the complement of $V_{\Tbb^k}(C)$ for some congruence $C$.

  The Euclidean topology on $\Tbb^k$ is generated by the subbasis consisting of open subsets of the form
  \[
  U_{i,a}^{+}=\{x\in \Tbb^k\mid  a<x_i\},\ U_{i,a}^{-}=\{x\in \Tbb^k\mid  x_i<a\},
  \]
  for some $i=1,\dots,k$ and $a\in \Rbb$.
  Let $C^{+}_{i,a}$ and $C^{-}_{i,a}$ be the congruences on $\Tbb[X_1,\dots,X_k]$ 
  generated by $(a,a\oplus X_i)$ and $(X_i,a\oplus X_i)$, respectively.
  Then the complement of $U_{i,a}^{\pm}$ is equal to $V_{\Tbb^k}(C^{\pm}_{i,a})$.
  Thus the assertion is proved.
\end{proof}

\begin{proposition}\label{proposition:TopEmb}
  Let $\Msf$ be an ordered monoid whose underlying monoid is finitely generated.
  Then the map 
  \[\Hom (\Msf,\Tbb)\to \CSpec_{\Tbb}\Tbb[\Msf],\,\phi\mapsto \Ker\overline{\phi}\] 
  is an injective map which induces a homeomorphism onto the subset $\CSpec^{\mathrm{geo}}\Tbb[\Msf]$.
\end{proposition}

\begin{proof}
  Since any $\Msf$ is finitely generated, we can take some $k$ 
  and a surjective o-homomorphism $\Nbb^k\twoheadrightarrow \Msf$ where 
  $\Nbb^k$ is equipped with the trivial order.
  Then we obtain the commutative diagram below.
  \[
  \begin{tikzcd}
    \Hom(\Msf,\Tbb)\arrow[r,hook]\arrow[d,hook] & \CSpec_{\Tbb}\Tbb[\Msf]\arrow[d,hook]\\
    \Hom(\Nbb^k,\Tbb)\arrow[r,hook] & \CSpec_{\Tbb}\Tbb[\Nbb^k]
  \end{tikzcd}
  \]
  By \cref{proposition:quotient induces inj,lemma:TopEmb}, 
  both vertical arrows and the bottom arrow are topological embedding.
  Thus so is the top arrow is also a topological embedding. 
\end{proof}

\begin{remark}
  For a suitable monoid $\Msf$, 
  the space $\Hom(\Msf,\Tbb)$ is used as a local model of a tropical toric variety, 
  which is introduced by \cite{Kajiwara2008TropicalToric} and \cite{payne2009analytification}. 
  We have shown that $\Hom(\Msf,\Tbb)$ can be identified with 
  a subset $\CSpec^{\mathrm{geo}}\Tbb[\Msf]$ of $\CSpec_{\Tbb}\Tbb[\Msf]$.
  In this sense, a tropical toric variety can be regarded as a subspace of the space of prime congruences,
  as we will see in \cref{section:toric}.
\end{remark}

\section{Prime congruences and submonoids}\label{section:B[M]cong}

The aim of this section is to understand $\CSpec\Bb[\Msf]$ as a poset.
We mainly use \cref{proposition:JM}.
In \cref{subsection:The case of ordered groups}, we consider the problem in the case where $\Msf=G$ is an ordered group.
In \cref{subsection:The stratification,subsection:The stratification T}, we consider the problem in the general setting.

\subsection{Convex congruences}\label{subsection:convex congruences}

In this and the following subsection, we assume that $\Msf=G$ is an ordered group.
We first give some observations.
Let $C$ be a congruence on $\Bb[G]$ and $\pi_C:\Bb[G]\to \Bb[G]/C$ be the quotient homomorphism.
Then
\begin{align}\label{eq:right arrow}
  \Lsf(C)\coloneqq \{m\in G\mid \pi_C(m)\geq 0\}
\end{align}
is a submonoid of $G$. 
Since $\pi_C|_G:G\to\pi_{C}(G)$ is an o-homomorphism, 
the submonoid $\Lsf=\Lsf(C)$ satisfies the following conditions.
\begin{align}
  &\{m\in \Msf\mid m\geq 0\}\subseteq \Lsf. \tag{L1}\label{L1}
\end{align}

Conversely, for an arbitrary subset $A$ of $G$, 
we can define a congruence $\<A\geq 0\>$ on $\Bb[G]$ as follows.
\begin{definition}
  For a subset $A$ of $G$, 
  we define a congruence $\<A\geq 0\>$ on $\Bb[G]$ as the congruence generated 
  by $\{(m\oplus 0,m)\mid m\in A\}$, i.e.,
  \[\<A\geq 0\>\coloneqq \<\{(m\oplus 0,m)\mid m\in A\}\>.\]
  A \emph{convex congruence} is a congruence of the form $\<A\geq 0\>$ for a subset $A$ of $G$.
\end{definition}

For a subset $A$ of $G$, let $\Lsf$ be a submonoid generated by
$A$ and $\{m\in G\mid m\geq 0\}$.
Then we have
\begin{align}\label{eq:property1}
  \<A\geq 0\>=\<\Lsf\geq 0\>.
\end{align}
Furthermore, for a family $\{A_i\}_i$ of subsets of $G$, we have
\begin{align}\label{eq:property2}
  \left\<\bigcup_i A_i\geq 0\right\>=\sum_{i}\<A_i\geq 0\>, 
\end{align}
where $\sum_i\<A_i\geq 0\>$ is the congruence generated by $\bigcup_i\<A_i\geq 0\>$.

On the other hand, for an arbitrary submonoid $\Lsf$ of $G$ satisfying \eqref{L1}, 
We define an order of $G/(\Lsf\cap -\Lsf)$ as follows:
\begin{align}\label{eq:order}
  \overline{g_1}\geq \overline{g_2} \overset{\mathrm{def}}{\iff} g_1-g_2\in \Lsf, 
\end{align}
where $\overline{g_i}$ is the image of $g_i\in G$ in $G/(\Lsf\cap -\Lsf)$.
Then $G\to G/(\Lsf\cap -\Lsf)$ is an o-homomorphism and induces the homomorphism 
\begin{align}\label{eq:hoge}
  \pi:\Bb[G]\to \Bb[G/(\Lsf\cap -\Lsf)].
\end{align}

\begin{lemma}\label{lemma:later used}
  Let $A$ be a subset of $G$.
  Then $\<A\geq 0\>$ is equal to the kernel of the homomorphism $\pi$ defined in \eqref{eq:hoge}
  where $\Lsf$ is the submonoid generated by $A$ and $\{m\in G \mid m\geq 0\}$.
\end{lemma}

\begin{proof}
  By \eqref{eq:property1}, we may assume $A=\Lsf$.
  An inclusion $\<\Lsf\geq 0\>\subseteq \Ker\pi$ is immediate from the definitions.
  To show the opposite inclusion, we take an arbitrary element $(f_1,f_2)\in \Ker\pi$.
  Define a subset
  \begin{align*}
    A_{(f_1,f_2)}\coloneqq \{m_{+}-m_{-}\mid m_{+},m_{-}\in \supp(f_1)\cup\supp(f_2)\text{ such that } \pi(m_{+})\geq \pi(m_{-})\}.
  \end{align*}
  Since the o-homomorphism $G\to G/(\Lsf\cap -\Lsf)$ is a restriction of $\pi$, 
  $A_{(f_1,f_2)}$ is a subset of $\Lsf$.
  Hence we have $(f_1,f_2)\in \<A_{(f_1,f_2)}\geq 0\>\subseteq \<\Lsf\geq 0\>$.
\end{proof}

\begin{lemma}\label{lemma:later used2}
  For a submonoid $\Lsf$ of $G$ satisfying \eqref{L1}, we have $\Lsf(\<\Lsf\geq 0\>)=\Lsf$.
  In particular, the map $\Lsf\mapsto \<\Lsf\geq 0\>$ is injective.
\end{lemma}

\begin{proof}
  We have $\Lsf\subseteq \Lsf(\<\Lsf\geq 0\>)$ from the definition of $\<\Lsf\geq 0\>$.
  To show the opposite inclusion, we take an arbitrary element $m\in \Lsf(\<\Lsf\geq 0\>)$.
  Then we have $\pi(m)\geq 0$ in $\Bb[G/(\Lsf\cap -\Lsf)]$.
  By the definition of the order of $G/(\Lsf\cap -\Lsf)$, we have $m\in \Lsf$.
\end{proof}

Thus two maps $  C \mapsto \Lsf(C)$ and $\Lsf \mapsto \<\Lsf\geq 0\>$ give a one-to-one correspondence:
\begin{align}\label{eq:0-th bijection}
  \{C \mid C \text{ is a convex congruence on } \Bb[G]\} \leftrightarrow 
  \{\Lsf \mid \Lsf \text{ is a submonoid of } G \text{ satisfying } \eqref{L1}\}.
\end{align}
Note that the correspondence preserves inclusions.

The following claim is crucial in \cref{section:fg}.

\begin{proposition}\label{proposition:new fg}
  Let $C$ be a convex congruence on $\Bb[G]$.
  Then the following are equivalent:
  \begin{itemize}
    \item $C$ is finitely generated as a congruence,
    \item there exists a finite subset $A$ of $G$ such that 
    the corresponding submonoid $\Lsf(C)$ is generated by $A$ and $\{m\in G\mid m\geq 0\}$.
  \end{itemize}
\end{proposition}

\begin{proof}
  By \eqref{eq:property1}, it is immediate that the latter implies the former.

  Assume that $C$ is generated by a finite set $\Acal\subseteq \Bb[G]^2$.
  As in the proof of \cref{lemma:later used}, we define a finite subset of $G$ by
  \[A_{(f_1,f_2)}\coloneqq \{g_{+}-g_{-}\in G\mid g_{+}, g_{-}\in\supp(f_1)\cup \supp(f_2),\ \overline{g_{+}}\geq \overline{g_{-}}\text{ in }G/(\Lsf(C)\cap -\Lsf(C))\},\]
  for a pair $(f_1,f_2)$.
  Since $\<A_{(f_1,f_2)}\geq 0\>$ contains a pair $(f_1,f_2)$, 
  a finite subset $A\coloneqq\bigcup_{(f_1,f_2)\in \Acal}A_{(f_1,f_2)}$ of $G$ satisfies
  \[\<A\geq 0\>=\sum_{(f_1,f_2)\in \Acal}\<A_{(f_1,f_2)}\geq 0\>\supseteq C.\]
  The opposite inclusion is clear from the definition of $A$. Thus we have $C=\<A\geq 0\>$.
  By \eqref{eq:property1}, the claim holds.
\end{proof}

\subsection{The case of ordered groups}\label{subsection:The case of ordered groups}

%ここからprime congruencesの話．
Now we consider the case of prime congruences.
Since we assume that $G$ is an ordered group, \cref{proposition:JM} can be stated more simply as follows.

\begin{lemma}\label{lemma:simple JM}
  A congruence $C$ on $\Bb[G]$ is prime if and only if $\Bb[G]/C$ is totally ordered.
\end{lemma}

\begin{proof}
  See \cite[Corollary 2.6]{joo2026varieties}.
\end{proof}

\begin{lemma}
  Prime congruences on $\Bb[G]$ are convex.
\end{lemma}

\begin{proof}
  Let $P$ be a prime congruence on $\Bb[G]$ and $\pi_P:\Bb[G]\to \Bb[G]/P$ be the quotient homomorphism.
  By \cref{lemma:prime isomorphism}, there is a natural isomorphism $\Bb[G]/P \cong \Bb[\pi_P(G)]$.

  As we did in \eqref{eq:right arrow}, we can define a submonoid $\Lsf(P)$ of $G$ by 
  \[\Lsf(P)=\{m\in G\mid \pi_P(m)\geq 0\}.\]
  There is an isomorphism of groups:
  \begin{align}\label{eq:cone-thm}
    G/(\Lsf(P)\cap -\Lsf(P))\xrightarrow{\cong} \pi_P(G),\, \overline{g}\mapsto \pi_P(g).
  \end{align}
  Since the order of $G/(\Lsf(P)\cap -\Lsf(P))$ is defined by \eqref{eq:order},
  \[\overline{g_1}\geq \overline{g_2}\iff  g_1-g_2\in \Lsf(P)\iff \pi_P(g_1-g_2)\geq 0\iff \pi_P(g_1)\geq \pi_P(g_2)\]
  holds. Thus \eqref{eq:cone-thm} is an o-isomorphism.
  By \cref{lemma:later used}, we have $P=\<\Lsf(P)\geq 0\>$.
\end{proof}

If $P$ is a prime congruence on $\Bb[G]$, then $\Lsf(P)$ satisfies not only \eqref{L1} 
but also the following condition:
\begin{align}
  &\Lsf\cup -\Lsf=G. \tag{L2}\label{L2}
\end{align}
We denote by $\Lcal_G$ the set of all submonoids of $G$ satisfying \eqref{L1} and \eqref{L2}.

Conversely, if a submonoid $\Lsf$ of $G$ satisfies \eqref{L1} and \eqref{L2}, 
then the congruence $\<\Lsf\geq 0\>$ is prime by \cref{lemma:simple JM}.
Thus we obtain a one-to-one correspondence as a restriction of \eqref{eq:0-th bijection}.

\begin{proposition}\label{proposition:corr inclusion group B}
  The one-to-one correspondence of \eqref{eq:0-th bijection} induces a one-to-one correspondence 
  between $\CSpec \Bb[G]$ and $\Lcal_G$.
  Furthermore, it preserves inclusions.
\end{proposition}

%proofなし

As illustrated in \cref{figure:RM-flag,figure:red blue}, for a typical ordered group, 
submonoids satisfying \eqref{L1} and \eqref{L2} necessarily have a very restricted form.
See \cref{section:Mflag} for the details.

\begin{proposition} \label{proposition:monoid pull-back} 
  Let $\phi:\Bb[G_1]\to\Bb[G_2]$ be the semiring homomorphism induced by an o-homomorphism $\psi:G_1\to G_2$.
  For $P_2\in\CSpec \Bb[G_2]$, we have
  $\Lsf(\phi^{*}P_2)=\psi^{-1}(\Lsf(P_2))$.
\end{proposition}

\begin{proof}
  We have
  \begin{align*}
    g\in \Lsf(\phi^{*}P_2) &\iff (g\oplus 0,g)\in \phi^{*}P_2\\
    &\iff (\phi(g)\oplus 0,\phi(g))=\phi((g\oplus 0,g))\in P_2\\
    &\iff \phi(g)\in \Lsf(P_2)\\
    &\iff g\in \phi^{-1}(\Lsf(P_2))\cap G_1=\psi^{-1}(\Lsf(P_2)).
  \end{align*}
  Thus the equality holds.
\end{proof}

%ここからT係数の場合．まずlying over Tについて．
We consider the case of a $\Tbb$-algebra $\Tbb[G]=\Bb[\Rbb\times G]$.
By \cref{proposition:corr inclusion group B}, we have an inclusion-preserving one-to-one correspondence as follows:
\begin{align}\label{eq:TTTT}
  \CSpec \Tbb[G] \longleftrightarrow \Lcal_{\Rbb\times G}=\{\Lsf\subseteq\Rbb\times G\mid \text{$\Lsf$ is a submonoid satisfying \eqref{L1} and \eqref{L2}} \}.  
\end{align}
To describe the structure of $\CSpec \Tbb[G]$ in more detail,
we next decompose an element of $\Lcal_{\Rbb\times G}$ into two submonoids.

Let $\Lsf$ be a submonoid of $\Rbb\times G$ satisfying \eqref{L1} and \eqref{L2}.
Since the order of $\Rbb\times G$ is given by \cref{definition:T[M]}, the projection 
\[\Pi:\Rbb\times G\to G\] 
is an o-homomorphism.
We define submonoids as follows:
\begin{align}
  \begin{split}\label{eq:Lin}
    &\Lsf^{\mathrm{out}} \coloneqq \Pi(\Lsf),\\
    &\Lsf^{\mathrm{inn}} \coloneqq \Lsf\cap (\Rbb\times (\Lsf^{\mathrm{out}}\cap -\Lsf^{\mathrm{out}})).
  \end{split}
\end{align}
It is straightforward to check that the submonoids 
$\Lsf^{\mathrm{out}}\subseteq G$ and $\Lsf^{\mathrm{inn}}\subseteq \Rbb\times (\Lsf^{\mathrm{out}}\cap -\Lsf^{\mathrm{out}})$ 
both satisfy \eqref{L1} and \eqref{L2}.

By translating the first coordinate,
a group $\Rbb$ acts on $\Rbb\times G$, and thus also on $\Rbb\times (\Lsf^{\mathrm{out}}\cap -\Lsf^{\mathrm{out}})$.
Any orbit of the action on $\Rbb\times (\Lsf^{\mathrm{out}}\cap -\Lsf^{\mathrm{out}})$ 
is written as $\Pi^{-1}(g)$ for some $g\in \Lsf^{\mathrm{out}}\cap -\Lsf^{\mathrm{out}}$.
Then the submonoid $\Lsf^{\mathrm{inn}}$ in \eqref{eq:Lin} satisfies the following condition:
\begin{align}
  \text{Every orbit intersects $\Lsf^{\mathrm{inn}}$.} \tag{L3}\label{L3}
\end{align}

Conversely, we take arbitrary submonoids
$\Lsf^{\mathrm{out}}\in \Lcal_{G}$ and $\Lsf^{\mathrm{inn}}\in \Lcal_{\Rbb\times (\Lsf^{\mathrm{out}}\cap -\Lsf^{\mathrm{out}})}$ 
where $\Lsf^{\mathrm{inn}}$ satisfies the condition \eqref{L3}.
Then we define the submonoid of $\Rbb\times G$ as follows:
\begin{align}\label{eq:Lin conv}
  \Lsf\coloneqq (\Rbb\times (\Lsf^{\mathrm{out}}\setminus (\Lsf^{\mathrm{out}}\cap -\Lsf^{\mathrm{out}})))\cup \Lsf^{\mathrm{inn}}.
\end{align}

The construction \eqref{eq:Lin} and \eqref{eq:Lin conv} yield the following one-to-one correspondence:

\begin{proposition}\label{proposition:corr group T}
  Let $G$ be an ordered group.
  There is a bijection between $\CSpec \Tbb[G]$ and 
  \begin{align*}
    \left\{(\Lsf^{\mathrm{out}},\Lsf^{\mathrm{inn}})\ \middle|\  
      \begin{aligned}
        &\text{$\Lsf^{\mathrm{out}}$ is a submonoid of $G$ satisfying \eqref{L1} and \eqref{L2},}\\
        &\text{$\Lsf^{\mathrm{inn}}$ is a submonoid of $\Rbb\times(\Lsf^{\mathrm{out}}\cap -\Lsf^{\mathrm{out}})$ satisfying \eqref{L1}, \eqref{L2}, and \eqref{L3}.} 
      \end{aligned}
    \right\}
  \end{align*}
\end{proposition}

\begin{proof}
  It is straightforward by the correspondence \eqref{eq:TTTT}.
\end{proof}

Under this decomposition, $\Lsf^{\mathrm{out}}$ and $\Lsf^{\mathrm{inn}}$ are called 
the \emph{outer part} and the \emph{inner part} of $\Lsf$, respectively.
The decomposition of \cref{proposition:corr group T} is particularly useful 
for prime congruences lying over $\Tbb$.

\begin{lemma}\label{lemma:when lying over T}
  Under the correspondence of \cref{proposition:corr group T},
  $P$ is lying over $\Tbb$ 
  if and only if $\Lsf(P)^{\mathrm{inn}}$ is a proper submonoid of $\Rbb\times (\Lsf(P)^{\mathrm{out}}\cap -\Lsf(P)^{\mathrm{out}})$, 
  i.e., $\Lsf(P)^{\mathrm{inn}}\subsetneq \Rbb\times (\Lsf(P)^{\mathrm{out}}\cap -\Lsf(P)^{\mathrm{out}})$ holds.
\end{lemma}

\begin{proof}
  By \eqref{eq:Lin conv}, we have $\Lsf(P)\cap -\Lsf(P)=\Lsf(P)^{\mathrm{inn}}\cap -\Lsf(P)^{\mathrm{inn}}$.
  Since $\pi_P$ is equal to the homomorphism induced by the o-homomorphism 
  \[\Rbb\times G\to (\Rbb\times G)/(\Lsf(P)\cap -\Lsf(P))=(\Rbb\times G)/(\Lsf(P)^{\mathrm{inn}}\cap -\Lsf(P)^{\mathrm{inn}})\]
  Thus $P$ is lying over $\Tbb$ if and only if $\Rbb\times \{0\}\not\subseteq \Lsf(P)^{\mathrm{inn}}\cap -\Lsf(P)^{\mathrm{inn}}$ holds.
  By \eqref{L3}, the latter condition is equivalent to $\Lsf(P)^{\mathrm{inn}}\neq \Rbb\times (\Lsf(P)^{\mathrm{out}}\cap -\Lsf(P)^{\mathrm{out}})$.
\end{proof}

\begin{lemma}\label{lemma:out and limit}
  Let $G$ be an ordered group. For $P\in\CSpec\Tbb[G]$, we have 
  \[\Lsf(P)^{\mathrm{out}}=\{g\in G\mid \lim_{P}(r,g)\in  \Rbb\cup\{+\infty\}\text{ for any (some) $r\in\Rbb$}\}.\]
\end{lemma}

\begin{proof}
  It is immediate from \cref{definition:limit}.
\end{proof}

Therefore the bijection of \cref{proposition:corr group T} 
preserves inclusions in the following sense:

\begin{proposition}\label{proposition:corr inclusion group T over T}
  Let $G$ be an ordered group.
  There is a bijection between the set $\CSpec_{\Tbb} \Tbb[G]$ and 
  \begin{align*}
    \left\{(\Lsf^{\mathrm{out}},\Lsf^{\mathrm{inn}})\ \middle|\  
      \begin{aligned}
        &\text{$\Lsf^{\mathrm{out}}$ is a submonoid of $G$ satisfying \eqref{L1} and \eqref{L2},}\\
        &\text{$\Lsf^{\mathrm{inn}}$ is a proper submonoid of $\Rbb\times(\Lsf^{\mathrm{out}}\cap -\Lsf^{\mathrm{out}})$ satisfying \eqref{L1}, \eqref{L2}, and \eqref{L3}.} 
      \end{aligned}
    \right\}
  \end{align*}
  Furthermore, for $P_1,P_2\in\CSpec_{\Tbb}\Tbb[G]$, 
  an inclusion $P_1\subseteq P_2$ holds if and only if
  $\Lsf_1^{\mathrm{out}}=\Lsf_2^{\mathrm{out}}$ and $\Lsf_1^{\mathrm{inn}}\subseteq \Lsf_2^{\mathrm{inn}}$ hold.
\end{proposition}

\begin{proof}
  If $\Lsf_1^{\mathrm{out}}=\Lsf_2^{\mathrm{out}}$ and $\Lsf_1^{\mathrm{inn}}\subseteq \Lsf_2^{\mathrm{inn}}$, 
  then the inclusion $P_1\subseteq P_2$ follows from the definition of the correspondence.

  Conversely, we assume $P_1\subseteq P_2$.
  By \cref{lemma:limit equal} and \cref{lemma:out and limit}, we have $\Lsf_1^{\mathrm{out}}=\Lsf_2^{\mathrm{out}}$ and
  \begin{align*}
    \Lsf_1^{\mathrm{inn}}&=\Lsf_1\cap(\Rbb\times (\Lsf_1^{\mathrm{out}}\cap -\Lsf_1^{\mathrm{out}}))\\
    &\subseteq \Lsf_2\cap (\Rbb\times (\Lsf_2^{\mathrm{out}}\cap -\Lsf_2^{\mathrm{out}}))\\
    &=\Lsf_2^{\mathrm{inn}}.
  \end{align*}
\end{proof}

%geomに含まれるための条件．
\begin{proposition}\label{proposition:when geom}
  Let $G$ be an ordered group and $P\in\CSpec \Tbb[G]$ lying over $\Tbb$.
  Then there exists a geometric congruence which contains $P$ if and only if $\Lsf(P)^{\mathrm{out}}=G$ holds.
\end{proposition}

\begin{proof}
  The claim follows immediately from \cref{lemma:geom condition,lemma:out and limit}.
\end{proof}

\subsection{The stratification of $\CSpec\Bb[\Msf]$}\label{subsection:The stratification}

Now we focus again on $\CSpec\Bb[\Msf]$, where $\Msf$ is an ordered monoid, 
not necessarily an ordered group. 
Recall that the mobile face of $P\in\CSpec\Bb[\Msf]$ is the face of $\Msf$ given by
\[\Fsf=\{m\in\Msf\mid (m,-\infty)\notin P\}.\]
We will see that $\CSpec\Bb[\Msf]$ has the canonical stratification along the mobile face of prime congruences.

A face $\Fsf$ of $\Msf$ can be seen as a multiplicatively closed subset of $\Bb[\Msf]$.
Recall that the localization $\Bb[\Msf]\to\Fsf^{-1}\Bb[\Msf]$ induces an inclusion
$\CSpec\Fsf^{-1}\Bb[\Msf]\to\CSpec\Bb[\Msf]$ (See \cref{proposition:loc induces inj}).
As in the notation of \eqref{eq:D(M)}, we denote the image of the inclusion by $D(\Fsf)$.
Namely, we define
\begin{align}\label{eq:UM}
  D(\Fsf)=\{P\in\CSpec\Bb[\Msf]\mid (m,-\infty)\notin P \text{ for all }m\in\Fsf\}.
\end{align}
We also recall that for a face $\Fsf$ of $\Msf$, 
the symbol $O_{\Fsf}$ denotes the set of prime congruences on $\Bb[\Msf]$ of mobile face $\Fsf$.

\begin{lemma}
  Let $\Fsf$ be a face of $\Msf$. Then we have
  \[O_{\Fsf}=V(\Gamma_{\Fsf})\cap D(\Fsf).\]
\end{lemma}

\begin{proof}
  For $P\in\CSpec\Bb[\Msf]$, the mobile face of $P$ is $\Fsf$ 
  if and only if the following two conditions hold:
  \begin{align*}
    m\in\Msf\setminus \Fsf\implies (m,-\infty)\in P,\\
    m\in\Fsf\implies (m,-\infty)\notin P.
  \end{align*}
  Each condition is equivalent to $P\in V(\Gamma_{\Fsf})$ and $P\in D(\Fsf)$, respectively
  Hence we have $O_{\Fsf}=V(\Gamma_{\Fsf})\cap D(\Fsf)$.
\end{proof}

From \cref{proposition:used where}, we have
\[O_{\Fsf}\cong \CSpec (\pi_{\Gamma_{\Fsf}}(\Fsf))^{-1}(\Bb[\Msf]/\Gamma_{\Fsf}).\]
From \cref{lemma:wah} and \cref{proposition:used where}, we have a commutative diagram below.
\[
\begin{tikzcd}
  \Bb[\Fsf]\ar[rr,"\cong"] \ar[dr,hook] \ar[ddd]  &  & \Bb[\Msf]/\Gamma_{\Fsf} \ar[ddd]\\
  & \Bb[\Msf]\ar[ur,two heads] \ar[d]& \\
  & \Fsf^{-1}\Bb[\Msf] \ar[dr,two heads] & \\
  \Bb[\Fsf^{\mathrm{gp}}]\ar[rr,"\cong"] \ar[ur,hook]  &  & \pi_{\Gamma_{\Fsf}}(\Fsf)^{-1}(\Bb[\Msf]/\Gamma_{\Fsf})
\end{tikzcd}
\]
Therefore we obtain the stratification of $\CSpec\Bb[\Msf]$:
\begin{align}\label{eq:process}
  \begin{split}
  &\CSpec\Bb[\Msf]=\coprod_{\Fsf\subseteq \Msf}O_{\Fsf},\\
  &O_{\Fsf}\xrightarrow{\cong} \CSpec\Bb[\Fsf^{\mathrm{gp}}],\ P\mapsto \tilde{P}=(P\cap\Bb[\Fsf])\cdot\Bb[\Fsf^{\mathrm{gp}}].
  \end{split}
\end{align}
Note that $\tilde{P}$ satisfies 
\[P=\<(\tilde{P}\cap \Bb[\Fsf])\cdot \Bb[\Msf],\Gamma_{\Fsf}\>.\]
The symbol $\tilde{P}$ is used in the sense above in the rest of this paper. 

%ここまでstratification．ここからさらにinclusionの条件．

By \cref{proposition:corr inclusion group B}, there is a one-to-one correspondence:
\[\CSpec\Bb[\Msf]\longleftrightarrow \coprod_{\Fsf\subseteq \Msf}\Lcal_{\Fsf^{\mathrm{gp}}}.\] 
To understand how these strata $O_{\Fsf}$ form a whole space,
we have to see when a prime congruence contains another in a distinct stratum. 

For a while, we fix $P_1,P_2$ to be points in $\CSpec\Bb[\Msf]$ whose mobile faces are $\Fsf_1,\Fsf_2$, respectively.

\begin{lemma}\label{lemma:P1P2}
  An inclusion $P_1\subseteq P_2$ holds if and only if 
  $\Fsf_2\subseteq \Fsf_1$ and $P_1\cap \Bb[\Fsf_1] \subseteq P_2\cap \Bb[\Fsf_1]$ hold. 
\end{lemma}

\begin{proof}
  Assume that $P_1\subseteq P_2$. 
  For $i=1,2$ and $m\in\Msf$,
  $m\in\Fsf_i$ if and only if $(m,-\infty)\notin P_i$. Thus we have $\Fsf_2\subseteq \Fsf_1$.
  It is also clear that $P_1\cap \Bb[\Fsf_1] \subseteq P_2\cap \Bb[\Fsf_1]$.

  Assume that $\Fsf_2\subseteq\Fsf_1$ and $P_1\cap \Bb[\Fsf_1] \subseteq P_2\cap \Bb[\Fsf_1]$. 
  Define congruences $Q_i=P_i\cdot (\Bb[\Msf]/\Gamma_{\Fsf_1})$ for $i=1,2$.
  Using \cref{lemma:phiphi} for $\Bb[\Msf]\to\Bb[\Msf]/\Gamma_{\Fsf_1}$, 
  we have $Q_i\cap \Bb[\Msf]=P_i$. Taking the pull-back, $Q_i\cap \Bb[\Fsf_1]=P_i\cap\Bb[\Fsf_1]$.
  By \cref{lemma:pushpush} and $\Gamma_{\Fsf_1}\subseteq P_1$, we have
  \[(\Bb[\Msf]/\Gamma_{\Fsf_1})/Q_i\cong \Bb[\Msf]/\<\Gamma_{\Fsf_1},P_i\>\cong \Bb[\Msf]/P_i\text{ for }i=1,2.\] 
  The isomorphism $\Bb[\Fsf_1]\xrightarrow{\cong} \Bb[\Msf]/\Gamma_{\Fsf_1}$ of \cref{lemma:wah} 
  induces the isomorphism 
  \[(\Bb[\Msf]/\Gamma_{\Fsf_1})/Q_i\cong \Bb[\Fsf_1]/(Q_i\cap \Bb[\Fsf_1]) \text{ for }i=1,2.\]
  Therefore, we obtain the isomorphism $\Bb[\Msf]/P_i\cong\Bb[\Fsf_1]/(P_i\cap \Bb[\Fsf_1])$ for $i=1,2$.
  Then the inclusion $P_1\cap \Bb[\Fsf_1] \subseteq P_2\cap \Bb[\Fsf_1]$ induces $P_1\subseteq P_2$.
\end{proof}

\begin{lemma} \label{lemma:triangle roof}
  If $\Fsf_2\subseteq\Fsf_1$, then the following are equivalent.
  \begin{enumerate}
    \item $P_1\cap\Bb[\Fsf_2] \subseteq P_2\cap\Bb[\Fsf_2]$,
    \item $\tilde{P_1}\cap\Bb[\Fsf_2^{\mathrm{gp}}]\subseteq \tilde{P_2}$.
  \end{enumerate}
\end{lemma}

\begin{proof}
  Denote $\pi_{P_i\cap\Bb[\Fsf_2]}:\Bb[\Fsf_2]\to\Bb[\Fsf_2]/(P_i\cap\Bb[\Fsf_2])$ by $\pi_i$ for $i=1,2$.
  \[
  \begin{tikzcd}
    {} & \Bb[\Fsf_2] \ar[ld,two heads] \ar[rd,two heads] & {}\\
    \Bb[\Fsf_2]/(P_1\cap\Bb[\Fsf_2]) \ar[d,hook] \ar[rr,dashed,"\text{induced by }(1)"] & {} & \Bb[\Fsf_2]/(P_2\cap \Bb[\Fsf_2]) \ar[d,hook]\\
    \pi_1(\Fsf_2)^{-1}(\Bb[\Fsf_2]/(P_1\cap\Bb[\Fsf_2])) \ar[d,"\cong"]  & {} & \pi_2(\Fsf_2)^{-1}(\Bb[\Fsf_2]/(P_2\cap \Bb[\Fsf_2])) \ar[d,"\cong"]\\
    \Bb[\Fsf_2^{\mathrm{gp}}]/(\tilde{P_1}\cap\Bb[\Fsf_2^{\mathrm{gp}}]) \ar[rr,dashed,"\text{induced by }(2)"] & {} & \Bb[\Fsf_2^{\mathrm{gp}}]/\tilde{P_2}
  \end{tikzcd}
  \]
  In the diagram, the two vertical isomorphisms are induced by \cref{proposition:used where},
  and the two localization maps are injective 
  by the cancellativity of $\Bb[\Fsf_2]/(P_i\cap\Bb[\Fsf_2])$ (see \cref{proposition:JM}).

  If $(1)$ holds, the bottom arrow is induced by the universal property of the localization.
  If $(2)$ holds, the upper arrow is induced by the injectivity of the localization maps.
\end{proof}

\begin{lemma}\label{lemma:shiage}
  Assume $\Fsf_2\subseteq \Fsf_1$, $P_1\cap\Bb[\Fsf_2]\subseteq P_2\cap\Bb[\Fsf_2]$, and that for any $m_1\in\Fsf_1\setminus \Fsf_2$ and $m_2\in\Fsf_2$, 
  $m_1-m_2\notin \Lsf(\tilde{P_1})$ in $\Fsf_1^{\mathrm{gp}}$. 
  There is a homomorphism $\Bb[\Fsf_1]/(P_1\cap \Bb[\Fsf_1])\to \Bb[\Fsf_2]/(P_1\cap \Bb[\Fsf_2])$
  which makes the following diagram commutative.
  \[
  \begin{tikzcd}
    \Bb[\Fsf_1]\ar[r,"\cong"]\ar[d] & \Bb[\Msf]/\Gamma_{\Fsf_1}\ar[r] & \Bb[\Msf]/\Gamma_{\Fsf_2} \ar[r,"\cong"] & \Bb[\Fsf_2]\ar[d]\\
    \Bb[\Fsf_1]/(P_1\cap \Bb[\Fsf_1]) \ar[rrr] & {} & {} & \Bb[\Fsf_2]/(P_1\cap\Bb[\Fsf_2])
  \end{tikzcd}
  \]
\end{lemma}

\begin{proof}
  The third assumption implies that $\Fsf'_2\coloneqq \pi_{P_1\cap \Bb[\Fsf_1]}(\Fsf_2)$ is 
  a face of $\Fsf'_1\coloneqq \pi_{P_1\cap \Bb[\Fsf_1]}(\Fsf_1)$.
  Thus we have a homomorphism $\Bb[\Fsf'_1]\to \Bb[\Fsf'_1]/\Gamma_{\Fsf'_2}\cong\Bb[\Fsf'_2]$.
  By \cref{lemma:prime isomorphism}, 
  we have $\Bb[\Fsf_1]/(P_1\cap \Bb[\Fsf_1])\to \Bb[\Fsf_2]/(P_1\cap \Bb[\Fsf_2])$.
  The commutativity of the diagram follows from the construction of the homomorphism.
\end{proof}

\begin{proposition}\label{proposition:inclusion monoid B} %inclusion3
  Let $P_1,P_2$ be points in $\CSpec\Bb[\Msf]$ whose mobile faces are $\Fsf_1,\Fsf_2$, respectively.
  For $i=1,2$, we denote by $\tilde{P_i}$ the point in $\CSpec\Bb[\Fsf_i^{\mathrm{gp}}]$ 
  corresponding to $P_i$ by \eqref{eq:process}.
  Then $P_1\subseteq P_2$ if and only if the following hold:
  \begin{enumerate}
    \item $\Fsf_2\subseteq \Fsf_1$,
    \item For any $m_1\in\Fsf_1\setminus \Fsf_2$ and $m_2\in\Fsf_2$, 
    $m_1-m_2\notin \Lsf(\tilde{P_1})$ in $\Fsf_1^{\mathrm{gp}}$, and
    \item $\Lsf(\tilde{P_1})\cap\Fsf_2^{\mathrm{gp}}\subseteq \Lsf(\tilde{P_2})$ in $\Fsf_1^{\mathrm{gp}}$.
  \end{enumerate}
\end{proposition}

\begin{proof}
  From \cref{lemma:P1P2} and \cref{proposition:monoid pull-back}, 
  it is enough to show that the following are equivalent under the assumption that $\Fsf_2\subseteq \Fsf_1$:
  \begin{enumerate}
    \item[(i)] $P_1\cap \Bb[\Fsf_1] \subseteq P_2\cap \Bb[\Fsf_1]$,
    \item[(ii)] $\tilde{P_1}\cap\Bb[\Fsf_2^{\mathrm{gp}}]\subseteq \tilde{P_2}$ and 
    for any $m_1\in\Fsf_1\setminus \Fsf_2$ and $m_2\in\Fsf_2$, 
    $m_1-m_2\notin \Lsf(\tilde{P_1})$ in $\Fsf_1^{\mathrm{gp}}$.
  \end{enumerate}

  First we assume (i).
  Consider the following commutative diagram.
  \[
  \begin{tikzcd}
    \Bb[\Fsf_2]\ar[r,"\iota"]\ar[d,"\phi_2"] & \Bb[\Fsf_1]\ar[d,"\phi_1"]\\
    \Bb[\Fsf_2^{\mathrm{gp}}]\ar[r,"\iota'"] & \Bb[\Fsf_1^{\mathrm{gp}}]
  \end{tikzcd}
  \]
  Then we have
  \[(\phi_2)^{*}(\iota')^{*}\tilde{P_1}=\iota^{*}(\phi_1)^{*}\tilde{P_1}=\iota^{*}(P_1\cap\Bb[\Fsf_1])=P_1\cap\Bb[\Fsf_2].\]
  The second equality follows from \cref{lemma:local phiphi}.
  Taking the push-out by $\phi_2$, we have
  \[(\phi_2)_{*}(\phi_2)^{*}(\iota')^{*}\tilde{P_1}=(\phi_2)_{*}(P_1\cap\Bb[\Fsf_1]).\]
  Applying \cref{lemma:local phiphi} to the left hand side, 
  we have 
  \[\tilde{P_1}\cap\Bb[\Fsf_2^{\mathrm{gp}}]=(P_1\cap\Bb[\Fsf_1])\cdot\Bb[\Fsf_2^{\mathrm{gp}}].\]
  By the assumption (i), we have
  \[\tilde{P_1}\cap\Bb[\Fsf_2^{\mathrm{gp}}]\subseteq (P_2\cap\Bb[\Fsf_2])\cdot\Bb[\Fsf_2^{\mathrm{gp}}]=\tilde{P_2}.\]
  
  Take $m_1\in\Fsf_1\setminus \Fsf_2$ and $m_2\in\Fsf_2$ and assume $m_1-m_2\in \Lsf(\tilde{P_1})$.
  Now we regard $m_1,m_2$ as elements of $\Bb[\Fsf_1]$.
  By \cref{proposition:used where}, the following diagram is commutative.
  \[
  \begin{tikzcd}
    {} & {} & \Bb[\Fsf_1]\ar[d,"\phi_1"] \ar[r,"\pi_{P_1\cap \Bb[\Fsf_1]}"]& \Bb[\Fsf_1]/(P_1\cap \Bb[\Fsf_1])\ar[d,"\phi'_1"]\\
    \Lsf(\tilde{P_1}) \ar[r,phantom,"\subseteq"] & \Fsf_1^{\mathrm{gp}} \ar[r,phantom,"\subseteq"] & \Bb[\Fsf_1^{\mathrm{gp}}]\ar[r,"\pi_{\tilde{P_1}}"] & \Bb[\Fsf_1^{\mathrm{gp}}]/\tilde{P_1}
  \end{tikzcd}
  \]
  Hence we obtain $\pi_{P_1\cap\Bb[\Fsf_1]}(m_1)\geq \pi_{P_1\cap\Bb[\Fsf_1]}(m_2)$ 
  by the cancellativity of $\Bb[\Fsf_1]/(P_1\cap\Bb[\Fsf_1])$.
  By the assumption (i), we have $\pi_{P_2\cap\Bb[\Fsf_1]}(m_1)\geq \pi_{P_2\cap\Bb[\Fsf_1]}(m_2)$.
  However, by the choice of $m_1$ and $m_2$,
  we have $\pi_{P_2\cap\Bb[\Fsf_1]}(m_1)= -\infty$ and $\pi_{P_2\cap\Bb[\Fsf_1]}(m_2)\neq -\infty$, 
  which gives a contradiction.
  Thus we have $m_1-m_2\notin \Lsf(\tilde{P_1})$. 
  Hence we obtain (ii) from (i).

  Conversely, we assume (ii).
  By \cref{lemma:triangle roof} and \cref{lemma:shiage}, 
  there exist homomorphisms 
  \begin{align*}
    \Bb[\Fsf_2]/(P_1\cap\Bb[\Fsf_2])\to \Bb[\Fsf_2]/(P_2\cap\Bb[\Fsf_2]),\\
    \Bb[\Fsf_1]/(P_1\cap\Bb[\Fsf_1])\to\Bb[\Fsf_2]/(P_1\cap\Bb[\Fsf_2]).
  \end{align*}
  The inclusion $\Fsf_2\subseteq \Fsf_1$ induces 
  the homomorphism $\Bb[\Fsf_2]/(P_2\cap \Bb[\Fsf_2])\to\Bb[\Fsf_1]/(P_2\cap\Bb[\Fsf_1])$.
  By composing these homomorphisms, 
  we obtain a homomorphism $\Bb[\Fsf_1]/(P_1\cap\Bb[\Fsf_1])\to \Bb[\Fsf_1]/(P_2\cap\Bb[\Fsf_1])$.
  Thus (i) holds.
\end{proof}

Thus we obtain a way to compare two points belonging to different strata.

\subsection{The stratification of $\CSpec\Tbb[\Msf]$}\label{subsection:The stratification T}

Next we consider a $\Tbb$-algebra $\Tbb[\Msf]$.
Recalling \cref{lemma:T face} and replacing $\Msf$ with $\Rbb\times \Msf$ in \eqref{eq:process}, we have
\begin{align}
  \begin{split}\label{Tprocess} %{proposition:corr, monoid, T} !定理化
    &\CSpec\Tbb[\Msf] = \coprod_{\Fsf\subseteq \Msf}O_{\Rbb\times \Fsf},\\
    &O_{\Rbb\times \Fsf} \xrightarrow{\cong} \CSpec\Tbb[\Fsf^{\mathrm{gp}}],\ P\mapsto \tilde{P}=(P\cap\Tbb[\Fsf])\cdot\Tbb[\Fsf^{\mathrm{gp}}].
  \end{split}
\end{align}

%ここまで普通のstratification．over T，geometricにrefineしていく．

\begin{proposition}\label{proposition:strata over T}
  Let $P$ be a prime congruence on a $\Tbb$-algebra $\Tbb[\Msf]=\Bb[\Rbb\times \Msf]$ over $\Tbb$.
  For the isomorphism in \eqref{Tprocess}, the following hold.
  \begin{itemize}
    \item $P$ is lying over $\Tbb$ if and only if $\tilde{P}$ is lying over $\Tbb$.
    \item $P$ is geometric if and only if $\tilde{P}$ is geometric.
  \end{itemize}
\end{proposition}

\begin{proof}
  By \cref{proposition:used where}, we have a following commutative diagram:
  \[
  \begin{tikzcd}
    \Bb[\Rbb\times \Msf]\ar[rr]\ar[dd] &{} &(\Rbb\times \Fsf)^{-1}\Bb[\Rbb\times \Msf] \ar[dd]\\
    {} & \ar[lu,hook]\ar[ld] \Tbb \ar[ru]\ar[rd]& {}\\
    \Bb[\Rbb\times M]/P\ar[rr] &{}& \pi_P(\Rbb\times \Fsf)^{-1}(\Bb[\Msf]/P) \ar[r,"\cong"] & \Bb[\Rbb\times \Fsf^{\mathrm{gp}}]/\tilde{P}
  \end{tikzcd}
  \]
  Since the localization map $\Bb[\Rbb\times \Fsf]/P\to \pi_P(\Rbb\times \Fsf)^{-1}(\Bb[\Msf]/P)$ is injective 
  on the subset $\pi_P(\Rbb\times \{0\})$ of $\Bb[\Rbb\times \Msf]/P$, the claims follow.
\end{proof}

By \cref{proposition:strata over T}, the stratification \eqref{Tprocess} is refined as follows:
\begin{align}
  \begin{split}\label{Tprocess2} %{proposition:corr, monoid, T, over T}
    &\CSpec_{\Tbb}\Tbb[\Msf]=\coprod_{\Fsf\subseteq \Msf}(O_{\Rbb\times \Fsf}\cap \CSpec_{\Tbb}\Tbb[\Msf]),\\
    &O_{\Rbb\times \Fsf}\cap \CSpec_{\Tbb}\Tbb[\Msf] \xrightarrow{\cong} \CSpec_{\Tbb}\Tbb[\Fsf^{\mathrm{gp}}],\ P\mapsto \tilde{P}=(P\cap\Tbb[\Fsf])\cdot\Tbb[\Fsf^{\mathrm{gp}}].
  \end{split}
\end{align}

\begin{proposition}\label{proposition:inclusion monoid T over T}
  Let $P_1,P_2$ be points in $\CSpec_{\Tbb}\Tbb[\Msf]$ whose mobile faces are 
  $\Rbb\times \Fsf_1,\Rbb\times \Fsf_2$, respectively.
  For $i=1,2$, we denote by $\tilde{P_i}$ the point in $\CSpec_{\Tbb}\Tbb[\Fsf_i^{\mathrm{gp}}]$ 
  corresponding to $P_i$ by \eqref{Tprocess2}.
  Then $P_1\subseteq P_2$ if and only if the following hold:
  \begin{enumerate}[label=(\roman*)]
    \item $\Fsf_2\subseteq \Fsf_1$,
    \item For any $m_1\in\Fsf_1\setminus \Fsf_2$ and $m_2\in\Fsf_2$, 
    $m_1-m_2\notin \Lsf(\tilde{P_1})^{\mathrm{out}}$ in $\Fsf_1^{\mathrm{gp}}$,
    \item $\Lsf(\tilde{P_1})^{\mathrm{out}}\cap \Fsf_2^{\mathrm{gp}}=\Lsf(\tilde{P_2})^{\mathrm{out}}$ 
    and $\Lsf(\tilde{P_1})^{\mathrm{inn}}\cap (\Rbb\times \Fsf_2^{\mathrm{gp}})\subseteq \Lsf(\tilde{P_2})^{\mathrm{inn}}$.
  \end{enumerate}
\end{proposition}

\begin{proof}
  By \cref{proposition:inclusion monoid B}, $P_1\subseteq P_2$ if and only if the following hold.
  \begin{enumerate}
    \item $\Rbb\times \Fsf_2\subseteq \Rbb\times \Fsf_1$,
    \item For any $(r_1,m_1)\in(\Rbb\times \Fsf_1)\setminus (\Rbb\times \Fsf_2)$ and $(r_2,m_2)\in\Rbb\times \Fsf_2$, 
    $(r_1,m_1)-(r_2,m_2)\notin \Lsf(\tilde{P_1})$ in $\Rbb\times \Fsf_1^{\mathrm{gp}}$,
    \item $\Lsf(\tilde{P_1})\cap(\Rbb\times \Fsf_2^{\mathrm{gp}})\subseteq \Lsf(\tilde{P_2})$ in $\Rbb\times \Fsf_1^{\mathrm{gp}}$.
  \end{enumerate}
  The equivalence of (1) and (i) is immediate, and so is that of (2) and (ii).
  By \eqref{eq:Lin conv} and \cref{proposition:corr inclusion group T over T}, 
  the conditions (3) and (iii) are equivalent.
\end{proof}

Note that by \cref{proposition:abstract geom}, 
any prime congruence $P$ on $\Tbb[\Msf]$ is contained in at most one geometric congruence.

\begin{corollary}\label{proposition:abstract proper}
  Let $P\in\CSpec_{\Tbb}\Tbb[\Msf]$ be a point of mobile face $\Fsf\subseteq \Msf$.
  Then there exists a (unique) geometric congruence containing $P$ if and only if
  the image of the localization $\Fsf\to \Fsf^{\mathrm{gp}}$ is contained in $-\Lsf(\tilde{P})^{\mathrm{out}}$.
\end{corollary}

\begin{proof}
  Let $\phi:\Fsf\to\Fsf^{\mathrm{gp}}$ be the localization.
  Assume that there is a geometric congruence $P'$ satisfying $P\subseteq P'$.
  By the condition (i) of \cref{proposition:inclusion monoid T over T},
  the mobile face $\Fsf'$ of $P'$ is a face of $\Fsf$.
  By the condition (iii) of \cref{proposition:inclusion monoid T over T},
  we have 
  \[\Lsf(\tilde{P'})^{\mathrm{out}}=\Lsf(\tilde{P})^{\mathrm{out}}\cap \Fsf'^{\mathrm{gp}}
  \subseteq \Lsf(\tilde{P})^{\mathrm{out}}.\]
  By \cref{proposition:when geom,proposition:strata over T},
  the submonoid $\Lsf(\tilde{P'})^{\mathrm{out}}$ is not proper, 
  that is, it is equal to $\Fsf'^{\mathrm{gp}}$.
  Thus we have 
  \[\phi(\Fsf')\subseteq \Fsf'^{\mathrm{gp}}=\Lsf(\tilde{P'})^{\mathrm{out}}=-\Lsf(\tilde{P'})^{\mathrm{out}}\subseteq -\Lsf(\tilde{P})^{\mathrm{out}}.\]
  On the other hand, by the condition (ii) of \cref{proposition:inclusion monoid T over T},
  we have 
  \[\phi(\Fsf\setminus \Fsf')\subseteq \Fsf^{\mathrm{gp}}\setminus \Lsf(\tilde{P})^{\mathrm{out}}\subseteq -\Lsf(\tilde{P})^{\mathrm{out}}.\]
  Thus we have $\phi(\Fsf)\subseteq -\Lsf(\tilde{P})^{\mathrm{out}}$.

  Conversely we assume $\phi(\Fsf)\subseteq -\Lsf(\tilde{P})^{\mathrm{out}}$.
  Then define a face of $\Fsf$ by 
  \[\Fsf'\coloneqq \Fsf\cap (\Lsf(\tilde{P})^{\mathrm{out}}\cap-\Lsf(\tilde{P})^{\mathrm{out}}).\]
  We define two submonoids by 
  \begin{align*}
    &\Lsf^{\mathrm{out}} \coloneqq \Fsf'^{\mathrm{gp}},\\
    &\Lsf^{\mathrm{inn}} \coloneqq \Lsf(\tilde{P})\cap(\Rbb\times\Fsf'^{\mathrm{gp}}).
  \end{align*}
  Using the correspondence of \cref{proposition:corr inclusion group T over T},
  we obtain a prime congruence $P''$ on $\Tbb[\Msf]$.
  By \cref{proposition:inclusion monoid T over T}, we have $P\subseteq P''$.
  By \cref{proposition:when geom}, there is a geometric congruence $\tilde{P'}$ on $\Tbb[\Fsf'^{\mathrm{gp}}]$ 
  which is contained in $\tilde{P''}$.
  By \cref{proposition:strata over T}, 
  $P'$ is a geometric congruence on $\Tbb[\Msf]$ containing $P''$, and hence $P$.
\end{proof}

\section{Representation of submonoids}\label{section:Mflag}

In this section, we examine the set of prime congruences on typical tropical algebras, such as
\[\Blau,\ \lau,\ \Bpol,\ \pol.\]

%\subsection{The Laurent polynomial semirings over $\Bb$ and $\Tbb$} 
\subsection{The Laurent polynomial semirings over $\Bb$} 

In this subsection, we consider the set of prime congruences 
on the Laurent polynomial semiring $\Blau$.
It is represented as an algebra introduced in \cref{definition:B-alg}:
\[\Blau=\Bb[M].\]
Applying $G=M$ to \cref{proposition:corr inclusion group B},
we obtain a one-to-one correspondence:
\begin{align*}
  \CSpec\Blau \longleftrightarrow \Lcal_M.
\end{align*}
Since $M$ is equipped with the trivial order, 
we may impose only the condition \eqref{L2} on submonoids of $M$. Namely,
\[\Lcal_M=\{\Lsf\subseteq M\mid \Lsf\cup -\Lsf=M\}.\]
As we show in this subsection, 
such a submonoid $\Lsf$ admits a lexicographic stratification described by an \emph{$M$-flag}.

Let $V$ be an $\Rbb$-linear space and $W=\{v\in V\mid l(v)=0\}$ be a linear hyperplane. 
An \emph{orientation} of $W$ is defined as either the set $\{v\in V\mid l(v)> 0\}$ or $\{v\in V\mid l(v)< 0\}$.

\begin{definition}\label{definition:M-flag}
  Let $M$ be the group of Laurent monomials.
  Then $M$ can be naturally embedded into $M_{\Rbb}\cong \Rbb^n$.
  An \emph{$M$-flag} is a collection of the following data:
  \begin{itemize}
    \item $l\in\Zbb_{\geq 0}$;
    \item A family of linear subspaces of $M_{\Rbb}$
    \[\{H_i\}_{i=0}^l:M_{\Rbb}=H_0\supsetneq H_1\supsetneq \dots\supsetneq H_l,\]
    such that $H_i$ is a hyperplane of the linear span $\Rbb(H_{i-1}\cap M)$ of $H_{i-1}\cap M$ for each $i\in \{1,\dots,l\}$;
    \item An orientation $H_{i,+}\subseteq \Rbb(H_{i-1}\cap M)$ of $H_i$ as a hyperplane for each $i\in \{1,\dots,l\}$.
  \end{itemize}
  We often simply write this collection of data as $\{H_i\}_{i=0}^l$ or $H_{\bullet}$.
  The number $l$ is called the \emph{length} of the $M$-flag.
  Denote the set of all $M$-flags by $\Fcal_M$.
\end{definition}

It follows directly that the codimension of $H_i$ in $M_{\Rbb}$ is greater than or equal to $i$.
Thus we have $l\leq n$.

First, for an $M$-flag $H_{\bullet}$, we define a submonoid of $M$:
\[\Psi_1(H_{\bullet})\coloneqq \left(\bigcup_{i=1}^l H_{i,+}\cup H_{l}\right)\cap M.\]

\begin{lemma}
  $\Psi_1(H_{\bullet})$ satisfies \eqref{L2}, i.e. $\Psi_1(H_{\bullet})\cup -\Psi_1(H_{\bullet})=M$.
\end{lemma}

\begin{proof}
  For each $i\in\{1,\dots,l\}$,
  \[\Rbb(H_{i-1}\cap M)=H_{i,+}\cup H_{i}\cup -H_{i,+}\]
  then by taking the intersection with $M$, we have
  \[H_{i-1}\cap M=((H_{i,+}\cup -H_{i,+})\cap M)\cup (H_i\cap M).\]
  By applying the equation inductively, we obtain the claim. 
\end{proof}

Thus we obtain a map
\[\Psi_1\colon \Fcal_M \to \{\Lsf\subseteq M\mid \text{$\Lsf$ is a submonoid satisfying \eqref{L2}}\}.\]

\begin{example}\label{example:red blue}
  Set $M\cong \Zbb^2$. Define an $M$-flag as follows:
  \begin{align*}
    &H_1=\{(0,y)\mid y\in\Rbb\} \text{ with } H_{1,+}=\{(x,y)\mid x,y\in\Rbb,x>0\},\\ 
    &H_2=\{(0,0)\} \text{ with } H_{2,+}=\{(0,y)\mid y\in\Rbb,y>0\}.
  \end{align*}
  Then the corresponding submonoid $\Psi_1(H_0\supsetneq H_1\supsetneq H_2)$ of $M$ is represented by
  the set of all red points and the origin $O$ in \cref{figure:red blue}. 

\begin{figure}[htbp]
  \centering
  \begin{tikzpicture}[scale=0.5]
    \draw[->,>=stealth,semithick] (-4.5,0)--(4.5,0)node[above]{$x$}; %x軸
    \draw[->,>=stealth,semithick] (0,-4.5)--(0,4.5)node[right]{$y$}; %y軸
    \filldraw[black] (0,0) circle (3pt) node[below right]{O}; %原点

    \foreach \x in {1,2,3,4}{
      \foreach \y in {-4,...,4}{
        \filldraw[red] (\x,\y) circle (3pt);
      }
    }
    \foreach \x/\y in {0/1, 0/2, 0/3, 0/4}{
        \filldraw[red] (\x,\y) circle (3pt);
    }

    \foreach \x in {-1,-2,-3,-4}{
      \foreach \y in {-4,...,4}{
        \filldraw[blue] (\x,\y) circle (3pt);
      }
    }
    \foreach \x/\y in {0/-1, 0/-2, 0/-3, 0/-4}{
        \filldraw[blue] (\x,\y) circle (3pt);
    }
  \end{tikzpicture}

  \caption{An example of a submonoid satisfying \eqref{L2}}
  \label{figure:red blue}
\end{figure}
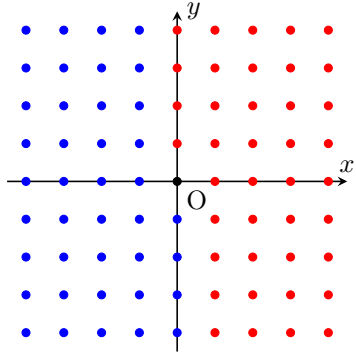

In this case, the length of the $M$-flag is 2. 
\end{example}

\begin{example}
  In \cref{example:red blue}, if $H_1$ is chosen as a hyperplane with an irrational slope, 
  then $H_1\cap M=\{O\}$. Thus the $M$-flag cannot be extended any further.
\end{example}

Next, we construct the inverse map of $\Psi_1$.
Let $\Lsf$ be a submonoid of $M$ satisfying \eqref{L2}.
We denote by $\overline{\cone(\Lsf)}$ the closure of the convex cone spanned by $\Lsf$ in $M_{\Rbb}$.
The subset $\overline{\cone(\Lsf)}\cap-\overline{\cone(\Lsf)}$ is the linear subspace of $M_{\Rbb}$.

\begin{lemma}\label{lemma:codim0or1}
  In the above setting, $\overline{\cone(\Lsf)}\cap-\overline{\cone(\Lsf)}$ 
  is a linear space in $M_{\Rbb}$ of codimension 0 or 1.
\end{lemma}

\begin{proof}
  Set $H=\overline{\cone(\Lsf)}\cap-\overline{\cone(\Lsf)}$.
  Denote the images of $\overline{\cone(\Lsf)}$ in $M_{\Rbb}/H $ by $F$.
  Then $F$ is a closed cone satisfying conditions:
  \[F\cup -F=M_{\Rbb}/H,\, F\cap -F=0.\]
  If the dimension of $M_{\Rbb}/H $ is larger than 1, such $F$ cannot exist.
\end{proof}

If the codimension of $H=\overline{\cone(\Lsf)}\cap-\overline{\cone(\Lsf)}$ is 1, we obtain a new submonoid
\[\Lsf'\coloneqq \Lsf\cap H.\]
It also satisfies \eqref{L2} 
as a submonoid of $M\cap H$.
Furthermore, 
the hyperplane $H$ of $M_{\Rbb}$ has a natural orientation $\overline{\cone(\Lsf)}\setminus H$.

Repeating this procedure yields a sequence of subspaces.
That is, for a submonoid $\Lsf$ of $M$ satisfying \eqref{L2}, we set
\begin{align*}
  &H_0=M_{\Rbb},\\
  &H_{i+1}=\overline{\cone(H_i\cap \Lsf)}\cap -\overline{\cone(H_i\cap \Lsf)},\\
  &H_{i+1,+}=\overline{\cone(H_i\cap \Lsf)}\setminus H_{i+1}.
\end{align*}
Since the dimension of $M_{\Rbb}$ is finite, 
there exists an integer $i$ such that 
\begin{align*}
  &H_i \text{ is a hyperplane of }\Rbb(H_{i-1}\cap M) \text{ if $i\leq l$,}\\
  &H_i=\Rbb(H_{i-1}\cap M)\text{ if $i> l$.}
\end{align*}
Then we obtain a sequence $\{H_i\}_{i=0}^l$, which is actually an $M$-flag.
Thus we obtain the map
\[\Psi_2\colon\{\Lsf\subseteq M\mid \text{$\Lsf$ is a submonoid satisfying \eqref{L2}}\}\to\Fcal_M.\]

\begin{theorem}\label{theorem:flag corr group B}
  There is a bijection between $\CSpec \Blau$ and 
  the set $\Fcal_M$ of all $M$-flags. 
\end{theorem}

\begin{proof}
  By \cref{proposition:corr inclusion group B}, 
  it is enough to show that $\Psi_1$ and $\Psi_2$ give a one-to-one correspondence. 

  First we show that $\Psi_1\Psi_2$ is the identity by induction on the rank $n$ of $M$.
  For a submonoid $\Lsf$ of $M$ satisfying \eqref{L2}, 
  we define an $M$-flag $H_{\bullet}=\{H_i\}_{i=0}^l=\Psi_2(\Lsf)$ and a submonoid $\Lsf'=\Psi_1(H_{\bullet})$ of $M$.
  If $l=0$, then $\Lsf=M$ and the assertion is immediate. Thus we may assume that $l>0$.
  The inclusion $M\cap H_1\subseteq M$ induces $\Fcal_{M\cap H_1}\subseteq \Fcal_M$ 
  by extending $(M\cap H_1)$-flags. 
  In particular, $(M\cap H_1)$-flag $K_{\bullet}=\{K_i\}_{i=0}^{l-1}$ defined by
  \begin{align*}
    &K_{i} =
      \begin{cases*}
        \Rbb(H_1\cap M),  &\text{ if $i=0$},\\
        H_{i+1}, &\text{ if $i>0$},
      \end{cases*}\\
    &K_{i,+} =H_{i+1,+} \text{ for $i=1,\dots,l-1$}.
  \end{align*}
  corresponds to $\{H_i\}_{i=0}^l$ via the inclusion.
  Then we have
  \[\Lsf'\cap H_1=\Psi_1(H_{\bullet})\cap H_1
  =\Psi_1(K_{\bullet})=\Psi_1\Psi_2(\Lsf\cap H_1)=\Lsf\cap H_1.\] 
  The last equality follows from the induction hypothesis.
  Then
  \begin{align*}
    \Lsf'&=M\cap \left(\bigcup_{i=1}^l H_{i,+}\cup H_l\right)\\
    &=(M\cap \left(\bigcup_{i=2}^l H_{i,+}\cup H_l\right))\cup (M\cap H_{1,+})\\
    &=(\Lsf'\cap H_1)\cup (M\cap H_{1,+})\\
    &=(\Lsf\cap H_1)\cup (M\cap H_{1,+})\\
    &=(\Lsf\cap H_1)\cup (\Lsf\setminus H_1)=\Lsf.
  \end{align*}
  Therefore $\Psi_1\Psi_2$ is the identity.

  Second, we show that 
  an $M$-flag $H'_{\bullet }\coloneqq \Psi_2\Psi_1(H_{\bullet})$ 
  is equal to $H_{\bullet}$ for any $H_{\bullet}\in \Fcal_M$.
  To show the claim by induction on $i$, we assume that $H'_{i_0}=H_{i_0}$ holds for $i_0\in\{0,\dots,l-1\}$. 
  Then we have
  \[
  \overline{\cone(H'_{i_0}\cap \Psi_1(H_{\bullet}))}=\overline{\cone(H_{i_0}\cap \Psi_1(H_{\bullet}))}
  =\overline{\cone\left(\left(\bigcup_{i=i_0+1} H_{i,+}\cup H_l\right)\cap M\right)}
  =\overline{\cone(H_{i_0+1,+}\cap M)},
  \]
  and
  \begin{align*}
    H'_{i_0+1}&=\overline{\cone(H'_{i_0}\cap \Psi_1(H_{\bullet}))}\cap -\overline{\cone(H'_{i_0}\cap \Psi_1(H_{\bullet}))}\\
    &=\overline{\cone(H_{i_0+1,+}\cap M)}\cap -\overline{\cone(H_{i_0+1,+}\cap M)}\\
    &=H_{i_0+1},\\
    H'_{i_0+1,+}&=\overline{\cone(H'_{i_0}\cap \Psi_1(H_{\bullet}))}\setminus H'_{i_0+1}\\
    &=\overline{\cone(H_{i_0+1,+}\cap M)}\setminus H_{i_0+1}\\
    &=H_{i_0+1,+}.
  \end{align*}
  Thus $\Psi_2\Psi_1$ is the identity.
\end{proof}

We often denote by $H(P)_{\bullet}$ the $M$-flag corresponding to $P$:
\begin{alignat*}{2}
  \CSpec\Bb[M] &\xrightarrow{1:1} &\Lcal_M \xrightarrow{1:1} &\Fcal_M,\\
  P &\longmapsto &\Lsf(P) \longmapsto &H(P)_{\bullet}.
\end{alignat*}
Note that by definition of $\Psi_1$, we have
\[\Lsf(P)\cap-\Lsf(P)=H(P)_l\cap M.\]

For an $M$-flag $H_{\bullet }=\{H_i\}_{i=0}^{l}$ and $l'\in \{0,\dots,l\}$, 
the sequence $H'_{\bullet}=\{H_i\}_{i=0}^{l'}$ is also an $M$-flag.
We call $H'_{\bullet}$ a \emph{truncation} of $H_{\bullet}$.

\begin{proposition}\label{proposition:flag inclusion group B}
  For $i=1,2$, let $P_i$ be a prime congruence on $\Bb[M]$ and $H(P_i)_{\bullet}$ be the corresponding $M$-flag.
  Then $P_1\subseteq P_2$ if and only if $H(P_2)_{\bullet}$ is a truncation of $H(P_1)_{\bullet}$.
  Furthermore, if the length of the $M$-flag $H(P)_{\bullet}$ corresponding a prime congruence $P$
  is equal to $l$, we have $\het P=\rk(H(P)_l\cap M)$ and $\coht P=l$.
\end{proposition}

\begin{proof}
  Recall that $P_1\subseteq P_2$ holds if and only if $\Lsf(P_1)\subseteq \Lsf(P_2)$.
  Thus the claims follow directly by the definition of $\Psi_1$ and $\Psi_2$.
\end{proof}

We give an alternative proof of the following proposition, which was established \cite{joo2018prime}.

\begin{corollary}\label{corollary:JM result}
  The dimension of $\Blau$ is equal to $n$.
\end{corollary}

\begin{proof}
  By \cref{proposition:flag inclusion group B}, the claim is equivalent to showing that
  every $M$-flag has length at most $n$ and there exists an $M$-flag of length $n$.
  Both assertions are immediate.
\end{proof}

\begin{remark}\label{remark:difference}
  In \cite{joo2018prime} and some research, 
  the calculation is carried out by using (t-)admissible matrix.
  Concretely, they show the existence of a surjective map
  \begin{align} \label{eq:admissible}
    \{\text{admissible matrices}\}\twoheadrightarrow \CSpec\Blau, \   U\mapsto P(U),
  \end{align}
  where $U$ is $l\times n$ admissible matrix with coefficients in $\Rbb$.
  However, some distinct admissible matrices can represent the same prime congruence,
  that is, the map \eqref{eq:admissible} is not injective.

  We construct $M$-flags from these matrices.
  By choosing a basis of $M_{\Rbb}$, 
  each row vector $u_1,\dots,u_l$ of $U$ can be seen as an element of $\Hom_{\Rbb}(M,\Rbb)$.
  Then we can define an $M$-flag as follows:
  \begin{align*}
    &H_0=M_{\Rbb},\\
    &H_{i+1}=\Rbb(H_i\cap M)\cap\{m\in M_{\Rbb}\mid u_i(m)=0\},\\
    &H_{i+1,+}=\Rbb(H_i\cap M)\cap \{m\in M_{\Rbb} \mid u_i(m)>0\}.
  \end{align*}
  Then this $M$-flag corresponds to the prime congruence $P(U)$ determined by $U$.
  That is, the following diagram is commutative:
  \[
  \begin{tikzcd}
    \{\text{admissible matrices}\}\ar[r, two heads]\ar[d,two heads] & \CSpec\Blau\\
    \{\text{$M$-flags}\}\ar[r,leftrightarrow,"\text{\cref{theorem:flag corr group B}}","1:1"'] & \Lcal_{M}\ar[u,leftrightarrow,"1:1","\text{(the proof of) \cref{proposition:corr inclusion group B}}"']
  \end{tikzcd}
  \]
  In this sense, an $M$-flag can be regarded as an equivalence class of admissible matrices 
  representing the same prime congruence.
\end{remark}

\subsection{The Laurent polynomial semirings over $\Tbb$} 

In this subsection, we consider the set of prime congruences 
on the Laurent polynomial semiring $\lau=\Bb[\Rbb\times M]$.
Applying $G=\Rbb\times M$ to \cref{proposition:corr inclusion group B},
we obtain a one-to-one correspondence:
\begin{align*}
  \CSpec\lau\longleftrightarrow \Lcal_{\Rbb\times M}.
\end{align*}
In a manner similar to \cref{theorem:flag corr group B}, 
the space $\CSpec\lau$ can be described by the sequence of linear spaces.
However since $\Rbb\times M$ has the non-trivial order, its non-negative cone is not zero.
For this reason, we have to impose a condition on the choice of each orientation.

\begin{definition}\label{definition:RM-flag}
  An \emph{$(\Rbb\times M)$-flag} is a collection of the following data:
  \begin{itemize}
    \item $l\in\Zbb_{\geq 0}$,
    \item A family of linear subspaces of $\Rbb\times M_{\Rbb}$
    \[\Rbb\times M_{\Rbb}=H_0\supsetneq H_1\supsetneq \dots\supsetneq H_l,\]
    such that $H_i$ is a hyperplane of $\Rbb(H_{i-1}\cap (\Rbb\times M))$ 
    for each $i\in \{1,\dots,l\}$,
    \item An orientation $H_{i,+}$ of $H_i$ as a hyperplane for each $i\in \{1,\dots,l\}$ such that 
  \[\Rbb_{\geq 0}\times \{0\} \subseteq \bigcup_{i=1}^l H_{i,+} \cup H_l.\]
  \end{itemize}
  The number $l$ is called the \emph{length} of the $(\Rbb\times M)$-flag. 
  Denote the set of all $(\Rbb\times M)$-flags by $\Fcal_{\Rbb\times M}$.
\end{definition}

\begin{theorem}\label{theorem:flag corr group T}
  There is a bijection between $\CSpec \lau$ and 
  the set $\Fcal_{\Rbb\times M}$ of all $(\Rbb\times M)$-flags.
\end{theorem}

\begin{proof}
  It is proved in a similar way to \cref{theorem:flag corr group B}.
\end{proof}

As in the case of $M$-flags, we often denote by $H(P)_{\bullet}$ the $(\Rbb\times M)$-flag corresponding to $P$:
\begin{alignat*}{2}
  \CSpec\lau &\xrightarrow{1:1} &\Lcal_{\Rbb\times M} \xrightarrow{1:1} &\Fcal_{\Rbb\times M},\\
  P &\longmapsto &\Lsf(P) \longmapsto &H(P)_{\bullet}.
\end{alignat*}
Note that $P$ is obtained by defining the order of $(\Rbb\times M)/(H_l\cap(\Rbb\times M))$ as follows: 
\begin{align}\label{eq:order by flag over T}
  \overline{m_1}\geq \overline{m_2} \overset{\mathrm{def}}{\iff} m_1-m_2\in \bigcup_{i=1}^l H_{i,+}\cup H_l.
\end{align}

Truncation of an $(\Rbb\times M)$-flag is defined in the same way as those of an $M$-flag.
The following proposition and corollary are proved 
in a similar way to \cref{proposition:flag inclusion group B} and \cref{corollary:JM result}.

\begin{proposition}\label{proposition:flag inclusion group T}
  For $i=1,2$, let $P_i$ be a prime congruence on $\Tbb[M]$ and 
  $H(P_i)_{\bullet}$ be the corresponding $(\Rbb\times M)$-flag.
  Then $P_1\subseteq P_2$ if and only if $H(P_2)_{\bullet}$ is a truncation of $H(P_1)_{\bullet}$.
  Furthermore, if the length of the $(\Rbb\times M)$-flag $H(P)_{\bullet}$ corresponding to a prime congruence $P$
  is equal to $l$, we have $\het P=\dim \Rbb(H(P)_l\cap (\Rbb\times M))$ where the spanning is taken in $\Rbb\times M_{\Rbb}$, and $\coht P=l$.
\end{proposition}

%proofなし

\begin{corollary}\label{corollary:JM result2}
  The dimension of $\lau$ is equal to $n+1$.
\end{corollary}

As the submonoid $\Lsf(P)$ is decomposed in \cref{proposition:corr group T},
an $(\Rbb\times M)$-flag we obtain another description of $\CSpec\Tbb[M]$.

\begin{lemma}\label{lemma:flag L3}
  Let $\Lsf$ be a submonoid of $\Rbb\times M$ satisfying \eqref{L1} and \eqref{L2} 
  and $H_{\bullet}$ the corresponding $(\Rbb\times M)$-flag of length $l$.
  The following are equivalent:
  \begin{enumerate}
    \item $\Lsf$ satisfies the condition \eqref{L3}, that is, 
    $\Lsf\cap (\Rbb\times \{m\})\neq \emptyset$ for any $m\in M$,
    \item $l=0$ or $\Rbb\times \{0\}\not\subseteq H_1$ holds.
  \end{enumerate}
\end{lemma}

\begin{proof}
  The implication (2) $\Rightarrow$ (1) is clear.
  We show (1) $\Rightarrow$ (2) by contraposition.
  Assume that $l\neq 0$ and $\Rbb\times \{0\}\subseteq H_1$.
  Then there exist a hyperplane $H'_1$ of $M_{\Rbb}$ and its orientation $H'_{1,+}$
  satisfying $H_1=\Rbb\times H'_1$ and $H_{1,+}=\Rbb\times H'_{1,+}$.
  For an element $m\in -H'_{1,+}$, 
  the orbit $\Rbb\times\{m\}$ does not intersect with $H_1\cup H_{1,+}\supseteq \Lsf$ and 
  hence $\Lsf$ does not satisfy \eqref{L3}.
\end{proof}

\begin{proposition}\label{proposition:flag corr group T innout} %proposition:flag corr, group, T
  There is a bijection between the set $\CSpec\Tbb[M]$ and the following set:
  \begin{align*}
    \left\{
    (H_{\bullet}^{\mathrm{out}},H_{\bullet}^{\mathrm{inn}})\ \middle|\  
      \begin{aligned}
        &H_{\bullet}^{\mathrm{out}}\in\Fcal_M, \ H_{\bullet}^{\mathrm{inn}}\in\Fcal_{\Rbb\times (H_{l'}^{\mathrm{out}}\cap M)}\\ 
        &\text{the length $l''$ of $H_{\bullet}^{\mathrm{inn}}$ is $0$ or $\Rbb\times \{0\}\not\subseteq H_1^{\mathrm{inn}}$ holds.} 
      \end{aligned}
    \right\}
  \end{align*}
\end{proposition}

\begin{proof}
  The correspondence is obtained by describing 
  the submonoids $\Lsf^{\mathrm{out}}$ and $\Lsf^{\mathrm{inn}}$ of \cref{proposition:corr group T}
  in terms of an $M$-flag and an $(\Rbb\times (\Lsf^{\mathrm{out}}\cap-\Lsf^{\mathrm{out}}))$-flag, respectively.
  The condition \eqref{L3} is translated by \cref{lemma:flag L3}.
\end{proof}

%ここに先にlying over Tのcorrespondenceを入れてもいいかもしれない

By  \cref{proposition:corr inclusion group T over T},
the decomposition of \cref{proposition:flag corr group T innout} is refined as follows:

\begin{proposition}\label{proposition:flag corr group T over T}
  There is a bijection between the set $\CSpec_{\Tbb}\Tbb[M]$ and the following set:
  \begin{align*}
    \left\{(H_{\bullet}^{\mathrm{out}},H_{\bullet}^{\mathrm{inn}})\ \middle|\  
    \begin{aligned}
      &H_{\bullet}^{\mathrm{out}}\in \Fcal_{M}, H_{\bullet}^{\mathrm{inn}}\in \Fcal_{\Rbb\times (H_{l'}^{\mathrm{out}}\cap M)},\\
      &\text{ $l''>0$, and $\Rbb\times \{0\}\not\subseteq H_1^{\mathrm{inn}}$ holds.} 
    \end{aligned}
    \right\}
  \end{align*}
  For $P_1,P_2\in \CSpec_{\Tbb}\Tbb[M]$,
  the inclusion $P_1\subseteq P_2$ holds if and only if the following hold:
  \begin{itemize}
    \item $H(P_1)_{\bullet}^{\mathrm{out}}=H(P_2)_{\bullet}^{\mathrm{out}}$,
    \item $H(P_2)_{\bullet}^{\mathrm{inn}}\text{ is a truncation of } H(P_1)_{\bullet}^{\mathrm{inn}}$.
  \end{itemize}
\end{proposition}

Thus we obtain a corollary.

\begin{corollary}\label{cor:good dimension}
  The dimension of $\CSpec_{\Tbb}\lau$ is equal to $n$.
\end{corollary}

We will describe the decomposition of flags explicitly. 
Let $\Pi\colon \Rbb\times M_{\Rbb}\to M_{\Rbb}$ be the projection.
For an $(\Rbb\times M)$-flag $H_{\bullet}$ of length $l$,
let $l'\geq 0$ be the largest integer such that $\Rbb\times \{0\}\subseteq H_{l'}$.
For the following lemma, we temporarily define an $M$-flag $H_{\bullet}^{\mathrm{(out)}}$ by
\begin{align*}
  &H_i^{\mathrm{(out)}}=\Pi(H_i)\text{ for }i=0,\dots,l'\\
  &H_{i,+}^{\mathrm{(out)}}=\Pi(H_{i,+})\text{ for }i=1,\dots,l',
\end{align*}
and an $(H_{l'}\cap(\Rbb\times M))$-flag $H_{\bullet}^{\mathrm{(inn)}}$ by
\begin{align*}
  &H_i^{\mathrm{(inn)}}=
  \begin{cases*}
    \Rbb\times \Rbb(H_{l'}^{\mathrm{out}}\cap M) &\text{ if $i=0$},\\
    H_{l'+i} &\text{ if $i=1,\dots,l-l'$},
  \end{cases*}\\
  &H_{i,+}^{\mathrm{(inn)}}=H_{l'+i,+}\text{ for }i=1,\dots,l-l'.
\end{align*}

\begin{lemma}\label{lemma:out flag}
  Let $\Lsf$ be the submonoid of $\Rbb\times M$ corresponding to $H_{\bullet}$.
  The sequence $H^{\mathrm{(out)}}_{\bullet}$ is the $M$-flag of length $l'$
  which corresponds to $\Lsf^{\mathrm{out}}$.
  The sequence $H^{\mathrm{(inn)}}_{\bullet}$ is the $(\Rbb\times (H_{l'}^{\mathrm{out}}\cap M))$-flag of length $l-l'$
  which corresponds to $\Lsf^{\mathrm{inn}}$.
\end{lemma}

\begin{proof}
  Since $H_{\bullet}$ is an $(\Rbb\times M)$-flag, 
  $H_i=\Rbb\times H_{i}^{\mathrm{(out)}}$ is a hyperplane of 
  \begin{align*}
    \Rbb(H_{i-1}\cap (\Rbb\times M))&=\Rbb((\Rbb\times H_{i-1}^{\mathrm{(out)}})\cap (\Rbb\times M))\\
    &=\Rbb(\Rbb\times (H_{i-1}^{\mathrm{(out)}}\cap M))\\
    &=\Rbb\times \Rbb(H_{i-1}^{\mathrm{(out)}}\cap M).
  \end{align*}
  for $i=1,\dots,l'$. Thus $H_i^{\mathrm{(out)}}$ is a hyperplane of $\Rbb(H_{i-1}^{\mathrm{(out)}}\cap M)$, and hence
  $H_{\bullet}^{\mathrm{(out)}}$ is an $M$-flag. Since 
  \begin{align*}
    \Pi(H_{i,+}\cap (\Rbb\times M))
    \begin{cases*}
      =H_{i,+}^{\mathrm{(out)}}\cap M &\text{ if $i<l'+1$},\\
      =H_{l'}^{\mathrm{(out)}}\cap M &\text{ if $i=l'+1$},\\
      \subseteq H_{l'}^{\mathrm{(out)}}\cap M &\text{ if $i>l'+1$}
    \end{cases*}
  \end{align*}
  holds, we have
  \begin{align*}
    \Lsf^{\mathrm{out}}
    &=\Pi\left(\left(\bigcup_{i=1}^{l}H_{i,+}\cup H_l\right)\cap (\Rbb\times M)\right)\\
    &=\Pi\left(\bigcup_{i=1}^{l}(H_{i,+}\cap (\Rbb\times M))\cup (H_{l'}\cap (\Rbb\times M))\right)\\
    &=\bigcup_{i=1}^{l}\Pi(H_{i,+}\cap (\Rbb\times M))\cup \Pi(H_{l'}\cap (\Rbb\times M))\\
    &=\bigcup_{i=1}^{l'}(H_{i,+}^{\mathrm{(out)}}\cap M)\cup (H_{l'}^{\mathrm{(out)}}\cap M)\\
    &=\left(\bigcup_{i=1}^{l'}H_{i,+}^{\mathrm{(out)}}\cup H_{l'}^{\mathrm{(out)}}\right)\cap M.
  \end{align*}
  By \cref{theorem:flag corr group B}, $\Lsf^{\mathrm{out}}$ is the submonoid represented by $H_{\bullet}^{\mathrm{(out)}}$.

  It is clear that the sequence $H_{\bullet}^{\mathrm{(inn)}}$ is $\Rbb\times (H_{l'}^{\mathrm{out}}\cap M)$-flag.
  \begin{align*}
    \Lsf^{\mathrm{out}}\cap -\Lsf^{\mathrm{out}}&=\left(\left(\bigcup_{i=1}^{l'}H_{i,+}^{\mathrm{out}}\cup H_{l'}^{\mathrm{out}}\right)\cap M \right)\cap -\left(\left(\bigcup_{i=1}^{l'}H_{i,+}^{\mathrm{out}}\cup H_{l'}^{\mathrm{out}}\right)\cap M \right)\\
    &=H_{l'}^{\mathrm{out}}\cap M,\\
    \Lsf^{\mathrm{inn}}&=\Lsf\cap (\Rbb\times (\Lsf^{\mathrm{out}}\cap -\Lsf^{\mathrm{out}}))\\
    &=\Lsf\cap H_{l'}\\
    &=\left(\bigcup_{i=1}^{l}H_{i,+}\cup H_l\right)\cap (\Rbb\times M)\cap H_{l'}\\
    &=\left(\bigcup_{i=l'+1}^{l}H_{i,+}\cup H_l\right)\cap (\Rbb\times M).
  \end{align*}
  Then $\Lsf^{\mathrm{inn}}$ corresponds to $H_{\bullet}^{\mathrm{(inn)}}$.
\end{proof}

Then the bijection of  \cref{proposition:flag corr group T innout} is described as follows:
\begin{align}
  \begin{aligned}\label{eq:bij inn out flag}
    H_{\bullet}=\{H_i\}_{i=0}^l &\longmapsto (H_{\bullet}^{\mathrm{(out)}},H_{\bullet}^{\mathrm{(inn)}})=(\{\Pi(H_i)\}_{i=0}^{l'},\{H_{l'+i}\}_{i=0}^{l''}),\\
    [\Rbb\times H_0^{\mathrm{out}}\supsetneq \dots\supsetneq \Rbb\times H_{l'}^{\mathrm{out}}\supsetneq H_1^{\mathrm{inn}}\supsetneq \dots\supsetneq H_{l''}^{\mathrm{inn}}] &\longmapsfrom (H_{\bullet}^{\mathrm{out}},H_{\bullet}^{\mathrm{inn}}).
  \end{aligned}
\end{align}
Note that via the bijection, the length $l$ of $H_{\bullet}$ is equal to $l'+l''$.

%以下はexplicitなbijectionより前に置くべきかも
\begin{proposition}\label{proposition:flag when geom}
  Let $(H_{\bullet}^{\mathrm{out}},H_{\bullet}^{\mathrm{inn}})$ be an $(\Rbb\times M)$-flag 
  which corresponds to $P\in \CSpec\Tbb[M]$ by  \cref{proposition:flag corr group T innout}.
  Then the following hold:
  \begin{enumerate}
    \item $P$ is contained in some geometric congruence if and only if $l'=0$.
    \item $P$ is geometric if and only if $l'=0$ and $l''=1$.
  \end{enumerate}
\end{proposition}

\begin{proof}
  The claim (1) follows from \cref{proposition:when geom}.
  If $l'=0$ and $l''=1$ hold, then $P$ is the kernel congruence of
  \[\lau=\Bb[\Rbb\times M]\to \Bb[(\Rbb\times M)/(H_1\cap (\Rbb\times M))]\cong \Bb[\Rbb]\cong \Tbb,\]
  and hence $P$ is geometric.

  To show the remaining part, assume that $P$ is geometric.
  By (1) and  \cref{proposition:flag corr group T over T}, we have $l''>0$ and $l'=0$.
  Let $(r_1,m_1)$ be an arbitrary element of $H_{1}\cap (\Rbb\times M)$.
  Since the order of $(\Rbb\times M)/(H_{l''}\cap(\Rbb\times M))$ 
  is defined by \eqref{eq:order by flag over T},
  \[-\epsilon<\pi_P((r_1,m_1))<\epsilon\text{ for any }\epsilon \in\Rbb_{>0},\]
  holds.
  Since $P$ is geometric, we have $\pi_P((r_1,m_1))=0$ and hence 
  \[H_1=\{(r,m)\in \Rbb\times M_{\Rbb}\mid \pi_P((r,m))=0\}=H_{l''}\]
  where $\pi_P$ is linearly extended to $\Rbb\times M_{\Rbb}$.
  Thus we have $H_1=H_{l''}$ and $l''=1$.
\end{proof}

\begin{corollary}\label{corollary:laurent geom dim}
  For a point $P\in X= \CSpec_{\Tbb}\lau$, $P$ is geometric if and only if $\dim_P X=n$.
\end{corollary}

\begin{proof}
  By \cref{proposition:flag inclusion group T},
  \[\het P= \dim  \Rbb(H_l\cap (\Rbb\times M))\]
  holds. Since $X$ is an open subset of $\CSpec\lau$ containing $P$, we have
  \[\dim_P X=\dim_P \CSpec \lau=\het P,\]
  by \cref{lemma:dim height} and \eqref{eq:dim at point open nbhd}.
  Thus the claim follows immediately by \cref{proposition:flag when geom}.
\end{proof}

\subsection{The polynomial semirings over $\Bb$ and $\Tbb$}\label{subsection:Bpol Tpol}

In this subsection, we consider the set of prime congruences
on the polynomial semiring $\Bpol$ and $\pol$.
By the notation \eqref{monomial notation}, we have
\[\Bpol=\Bb[\Mp],\ \pol=\Bb[\Rbb\times \Mp].\]

For $I\subseteq \{1,\dots,n\}$, we define a submonoid $\Fsf_I$ of $\Mp$ as follows:
\[\Fsf_I=\{uX=u_1X_1+\dots+u_nX_n \mid u_i=0 \text{ for all }i\notin I\},\]
which is isomorphic to $(\Nbb)^{I}$.

\begin{lemma}
  Any face of $\Mp$ is equal to $\Fsf_I$ for some $I$.
  Any face of $\Rbb\times \Mp$ is equal to $\Rbb\times \Fsf_I$ for some $I$. 
\end{lemma}

\begin{proof}
  Take a face $\Fsf$ of $\Mp$ and an element $uX=u_1X_1+\dots+u_nX_n$ of $\Fsf$.
  By the definition of faces, $X_i\in\Fsf$ holds if $u_i>0$.
  So we have $\Fsf=\Fsf_{I}$ for $I=\{i\mid X_i\in\Fsf\}$.  
  The claim about $\Rbb\times M$ follows from \cref{lemma:T face}. 
\end{proof}

The stratification \eqref{eq:process} is described as follows: For $\Msf=\Mp$, 
\begin{align}
  \begin{aligned}\label{eq:flag corr monoid B}
    &\CSpec\Bb[X_1,\dots,X_n]=\coprod_{I\subseteq \{1,\dots,n\}}O_{\Fsf_I},\\
    &O_{\Fsf_I}\xrightarrow{\cong} \CSpec\Bb[\Fsf_I^{\mathrm{gp}}],\,P\mapsto \tilde{P}=(P\cap \Bb[\Fsf_I])\cdot \Bb[\Fsf_I^{\mathrm{gp}}],\\
    &\CSpec\Bb[\Fsf_I^{\mathrm{gp}}]=\CSpec\Bb[X_i^{\pm}\mid i\in I]\longleftrightarrow \Fcal_{\Zbb^I},
  \end{aligned}
\end{align}
and for $\Msf=\Rbb\times \Mp$,
\begin{align}
  \begin{aligned}\label{eq:flag corr monoid T}
    &\CSpec\Tbb[X_1,\dots,X_n]=\coprod_{I\subseteq \{1,\dots,n\}}O_{\Rbb\times\Fsf_I},\\
    &O_{\Rbb\times \Fsf_I}\xrightarrow{\cong} \CSpec\Tbb[\Fsf_I^{\mathrm{gp}}],\,P\mapsto \tilde{P}=(P\cap \Tbb[\Fsf_I])\cdot \Tbb[\Fsf_I^{\mathrm{gp}}],\\
    &\CSpec\Tbb[\Fsf_I^{\mathrm{gp}}]=\CSpec\Tbb[X_i^{\pm}\mid i\in I]\longleftrightarrow \Fcal_{\Rbb\times \Zbb^I},
  \end{aligned}
\end{align}

\cref{theorem:flag corr group B} provides 
the complete description of each stratum $O_{\Fsf_I}$ and $O_{\Rbb\times\Fsf_I}$.
Moreover, by \cref{proposition:inclusion monoid B}, 
we know precisely how these strata form the whole space.

We introduce \emph{restriction} of flags, which is a basic operation on flags on a par with truncation.
For a prime congruence $P$ on $\Bb[M]$, let $P'$ be the pull-back of $P$ by the inclusion $\Bb[M']\subseteq \Bb[M]$,
where $M'$ is a subgroup of $M$ satisfying $M'_{\Rbb}\cap M=M'$.
We compare the $M$-flag $H_{\bullet}$ corresponding to $P$ and 
the $M'$-flag corresponding to $P'$.

By \cref{proposition:monoid pull-back}, we have $\Lsf(P')=\Lsf(P)\cap M'$.
\begin{align}\label{eq:restriction}
  \Lsf(P')\cap M'=\left(\bigcup_{i=1}^{l}H_{i,+}\cup H_l \right)\cap M'=\left(\bigcup_{i=1}^{l}(H_{i,+}\cap {M'_{\Rbb}})\cup (H_l\cap M'_{\Rbb}) \right)\cap M'.
\end{align}
Although the sequence $\{H_{i}\cap M'_{\Rbb}\}_{i=0}^l$ is not an $M'$-flag in general,
the codimension of $H_i\cap M'_{\Rbb}$ in $\Rbb((H_{i-1}\cap M'_{\Rbb})\cap M')$ 
is at most 1 for $i=1,\dots,l$.
Then we obtain an $M'$-flag by eliminating $H_i\cap M'_{\Rbb}$ such that
\[H_i\cap M'_{\Rbb}=\Rbb((H_{i-1}\cap M'_{\Rbb})\cap M')\]
from the sequence $\{H_{i}\cap M'_{\Rbb}\}_{i=0}^l$. 

\begin{definition}\label{definition:restriction}
  The $M'$-flag defined above is denoted by $H_{\bullet}\cap^{*}M'$, 
  called the \emph{restriction} of $H_{\bullet}$ to $M'$.
\end{definition}

By the assumption that $M'_{\Rbb}\cap M=M'$ and \eqref{eq:restriction},
the $M'$-flag $H_{\bullet}\cap^{*}M'$ corresponds to 
the submonoid $\Lsf\cap M'$ and the prime congruence $P'=P\cap \Bb[\Msf']$.
An $M'$-flag $H_{\bullet}\cap^{*}M'$ has length at most the length of $H_{\bullet}$.
%In particular, if $M'\subsetneq M$, then the inequality is strict.

For an $(\Rbb\times M)$-flag $H_{\bullet}$ and a subgroup $M'$ of $M$ satisfying $M'_{\Rbb}\cap M=M'$,
we define an $(\Rbb\times M')$-flag $H_{\bullet}\cap^{*}(\Rbb\times M')$ similarly.

\begin{lemma}\label{lemma:strict inclusion}
  For $i=1,2$, let $P_i$ be a point of $\CSpec \Bpol$ and $l_i$ the length of 
  the corresponding $\Fsf^{\mathrm{gp}}$-flag $H(P_i)_{\bullet}$ by \eqref{eq:flag corr monoid B}.
  If $P_1\subsetneq P_2$ holds, we have $l_1>l_2$.
  The same holds for points of $\CSpec \pol$.
\end{lemma}

\begin{proof}
  By \cref{proposition:inclusion monoid B}, we have
  \[\Lsf(\tilde{P_1})\cap\Fsf_2^{\mathrm{gp}}\subseteq \Lsf(\tilde{P_2}).\]
  Thus $H(P_2)_{\bullet}$ is a truncation of 
  the restriction $H(P_1)_{\bullet}\cap^{*}\Fsf_2^{\mathrm{gp}}$.

  To derive a contradiction, we assume $P_1\subsetneq P_2$ and $l_1=l_2$.
  Then a truncation is trivial, that is, $H(P_2)_{\bullet}=H(P_1)_{\bullet}\cap^{*}\Fsf_2^{\mathrm{gp}}$.
  If the inclusion $\Fsf_2\subseteq \Fsf_1$ is proper, then $\Fsf_2^{\mathrm{gp}}$ is contained in $H(P_1)_1$
  by the condition (ii) of \cref{proposition:inclusion monoid B}.
  It contradicts the assumption $l_1=l_2$, and thus $\Fsf_2$ is equal to $\Fsf_1$.
  Hence we have $P_1=P_2$, which also contradicts the assumption.
  In the case of $\pol$, the claim can be proved similarly by replacing $\Fsf_i$ with $\Rbb\times \Fsf_i$.
\end{proof}

\begin{corollary}\label{corollary:JM result3}
  The dimension of $\Bpol$ is equal to $n$.
  The dimension of $\pol$ is equal to $n+1$.
\end{corollary}

\begin{proof}
  The space $\CSpec \Bpol$ contains open subset $O_{\Fsf_{\{1,\dots,n\}}}$ 
  which is homeomorphic to the space $\CSpec \Blau$.
  By \cref{lemma:dim height} and \cref{corollary:JM result}, we have 
  \[\dim \Bpol=\dim \CSpec \Bpol\geq \dim\CSpec \Blau =\dim \Blau=n.\]

  Since the length of flags appeared in \eqref{eq:flag corr monoid B} is at most $n$,
  the dimension of $\Bpol$ is at most $n$ by \cref{lemma:strict inclusion}.
  The equality $\dim\pol=n+1$ is proved similarly.  
\end{proof}

In fact, these assertions are generalized as follows:
\begin{align*}
  \dim\Bb[\Msf_{\sigma}]=n,\ \dim\Tbb[\Msf_{\sigma}]=n+1,
\end{align*}
where $\sigma$ is a cone in $N$ and $\Msf_{\sigma}=M\cap \sigma^{\vee}$ 
(see \cref{section:toric} for details).
In particular, properties of the open subset $\CSpec_{\Tbb}\lau\subseteq \CSpec \lau$ are
examined in \cref{subsection:affine toric}. 

Here we present a concrete example.

\begin{example}\label{example:affine line}
  \cref{figure:A^1} is the Hasse diagram of the affine line 
  $X=\CSpec \Tbb[X]$.
  Each arrow represents an inclusion.

  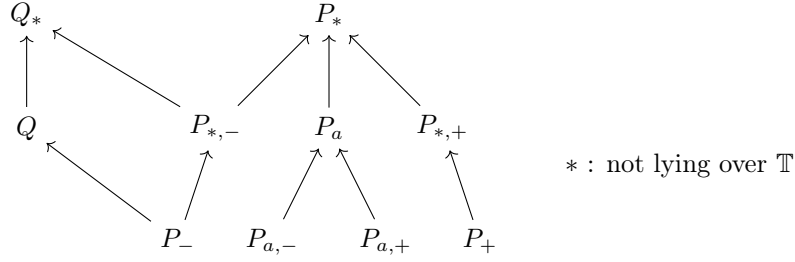
\begin{figure}[h]
  \centering
  \begin{tikzpicture}[<-]
  \node (Q) {$Q_{\ast}$}
    child { node (Q') {$Q$} 
  };
  \node at (4,0) {$P_{\ast}$}
    child { node (P-) {$P_{\ast,-}$} 
      child { node (P'-) [xshift=-5mm]{$P_{-}$} }
    }
    child { node {$P_a$} 
      child { node {$P_{a,-}$} }
      child { node {$P_{a,+}$} }
    }
    child { node {$P_{\ast,+}$} 
      child { node [xshift=5mm]{$P_{+}$} }
    };

  % added lines
  \draw (Q) -> (P-);
  \draw (Q') -> (P'-);

  \node[right] at (7,-2) {$\ast$ : not lying over $\Tbb$};

  \end{tikzpicture}
  \caption{The Hasse diagram of the tropical affine line.}
  \label{figure:A^1}
  \end{figure}

  We explain each point (see also \cref{figure:RM-flag}).
  The space $X$ is decomposed into two strata:
  \begin{align*}
    O_{\Fsf_{\emptyset}}\cong \CSpec \Tbb\longleftrightarrow \Fcal_{\Rbb\times \{0\}},\quad
    O_{\Fsf_{\{1\}}}\cong \CSpec \Tbb[X^{\pm}]\longleftrightarrow \Fcal_{\Rbb\times \Zbb},
  \end{align*}

  The $(\Rbb\times \{0\})$-flags, which correspond to points on $O_{\Fsf_{\emptyset}}$, 
  are classified as follows:
  \begin{itemize}
    \item length 0, which is uniquely determined, denoted by $Q_{\ast}$.
    \item length 1, which is uniquely determined, denoted by $Q$.
  \end{itemize}
  Via the homeomorphism $O_{\Fsf_{\emptyset}}\cong \CSpec \Tbb$, 
  the points $Q_{\ast}$ and $Q$ correspond to
  the kernel congruence of the surjective homomorphism $\Tbb\to \Bb$ and the diagonal congruence $\Delta$ on $\Tbb$, respectively.

  Similarly, the $(\Rbb\times \Zbb)$-flags, which correspond to points on $O_{\Fsf_{1}}$, 
  are classified into the following types:
  \begin{itemize}
    \item length 0, which is uniquely determined, denoted by $P_{\ast}$ 
    Via the homeomorphism $O_{\Fsf_{1}}\cong \CSpec\Tbb[X^{\pm}]$, 
    $P_{\ast}$ corresponds to $\Ker(\Tbb[X^{\pm}]\to\Bb)$.
    \item length 1 and $H_1$ does not contain $\Rbb\times \{0\}$, 
    which are determined by a singleton $\{(a,-1)\}=H_1\cap (\Rbb\times \{-1\})$.
    We denote the corresponding point by $P_a$. 
    Via the homeomorphism $O_{\Fsf_{1}}\cong \CSpec\Tbb[X^{\pm}]$, $P_a$ corresponds to $\Ker(\Tbb[X^{\pm}]\to \Tbb,\, X\mapsto a)$.
    \item length 1 and $H_1=\Rbb\times \{0\}$, which are determined by the orientation of $H_1$.
    We denote corresponding points by $P_{\ast,+}$ and $P_{\ast,-}$, respectively.
    \item length 2 and $H_1$ does not contain $\Rbb\times \{0\}$, 
    which is determined by $a$ of (2) and the orientation of $H_2=\{0\}$.
    We denote corresponding points by $P_{a,+}$ and $P_{a,-}$, respectively.
    \item length 2 and $H_1=\Rbb\times \{0\}$, which are uniquely determined by the orientation of $H_1$.
    We denote corresponding points by $P_{+}$ and $P_{-}$, respectively. 
  \end{itemize}

  By \cref{proposition:inclusion monoid B}, we know the inclusion relation of these points.
  In this way, we obtain \cref{figure:A^1}.
  It follows that $\dim\Tbb[X]$ and $\dim\CSpec \Tbb[X]$ are equal to $2$.
  Since the points without $\ast$ are lying over $\Tbb$, $\dim\CSpec_{\Tbb}\Tbb[X]$ is $1$.
  Note that in $\CSpec_{\Tbb}\Tbb[X]$, $P_a$ and $Q$ are geometric points and hence closed points,
  but a closed point $P_{+}$ is not geometric.
  This is analogous to the fact that the classical affine line is not proper.
\end{example}

\section{Toric construction}\label{section:toric}

In this section, we glue together sets of prime congruences using techniques of toric geometry.
Most of the discussion below generalizes that of \cref{section:Mflag}.

First we recall the usual setting in toric geometry.
Let $M$ and $N$ be free abelian groups of rank $n$ with a pairing 
\[\<\ \cdot\ ,\ \cdot\ \>\colon M\times N\to \Zbb, \, (m,n)\mapsto \<m,n\>.\]
Set $M_{\Rbb}=M\otimes_{\Zbb} \Rbb$ and $N_{\Rbb}=N\otimes_{\Zbb} \Rbb$.
A \emph{(convex) cone} $\sigma$ is a subset of $N_{\Rbb}$ represented by
\[\sigma=\left\{\sum_{v\in S}a_v v\  \middle|\  a_v\in\Rbb_{\geq0}\right\},\]
for some finite subset $S\subseteq N$, 
and its \emph{dimension} is the dimension of the linear span of $\sigma$ in $N_{\Rbb}$.
A cone $\sigma$ is \emph{strongly convex} if $\sigma\cap -\sigma=\{0\}$.
A subset $\tau$ of $\sigma$ is called a \emph{face} of $\sigma$ 
if the following condition holds for any $v_1,v_2\in \sigma$:
\[v_1+v_2\in \tau\implies v_1,v_2\in \tau.\]
We denote it by $\tau\prec\sigma$.
Note that if we regard cones as ordered monoids with the trivial order, 
\cref{definition:face} agrees with the present definition.

A \emph{fan} $\Sigma$ is a polyhedral complex in $N_{\Rbb}$ consisting of strongly convex cones.
For a cone $\sigma\subseteq N_{\Rbb}$, define subsets of $M$:
\begin{align*}
  &\sigma^{\perp}\coloneqq \{m\in M_{\Rbb}\mid \<m,n\>= 0 \text{ for all }n\in \sigma\},\\
  &\sigma^{\vee}\coloneqq \{m\in M_{\Rbb}\mid \<m,n\>\geq 0 \text{ for all }n\in \sigma\},\\
  &\Msf_{\sigma}\coloneqq \sigma^{\vee}\cap M.
\end{align*}

\subsection{Affine toric varieties}\label{subsection:affine toric}

As in standard toric geometry, we construct local models from cones in $N_{\Rbb}$.
For a cone $\sigma$, we define
\[U_{\sigma}^{\con}\coloneqq \CSpec_{\Tbb}\Tbb[\Msf_{\sigma}]=\CSpec_{\Tbb}\Bb[\Rbb\times \Msf_{\sigma}].\]

\begin{lemma}\label{lemma:toric face}
  Let $\sigma$ be a cone in $N_{\Rbb}$.
  For any face $\Fsf$ of the monoid $\Rbb\times \Msf_{\sigma}$, there exists a face $\tau$ of $\sigma$ such that
  \[\Fsf=\Rbb\times(\Msf_{\sigma}\cap\tau^{\perp}),\]
  and its group completion $\Fsf^{\mathrm{gp}}$ is equal to $\Rbb\times (\tau^{\perp}\cap M)$.
\end{lemma}

\begin{proof}
  By \cref{lemma:T face}, for any face $\Fsf$ of $\Rbb\times \Msf_{\sigma}$, 
  there exists a face $\Fsf'$ of $\Msf_{\sigma}$ such that $\Fsf=\Rbb\times \Fsf'$.
  Thus the claim follows from \cite[Proposition 1.2.10]{cox2011toric} and $(\sigma^{\vee})^{\mathrm{gp}}=M_{\Rbb}$.
\end{proof}

By \eqref{Tprocess2} and \cref{lemma:toric face},
we obtain the stratification of $U_{\sigma}^{\con}$:
\begin{align}\label{eq:toric process}
  \begin{split}
    &U_{\sigma}^{\con} = \coprod_{\tau\prec \sigma}(O_{\Rbb\times(\Msf_{\sigma}\cap\tau^{\perp})}\cap \CSpec_{\Tbb}\Tbb[\Msf_{\sigma}])\\
    &O_{\Rbb\times (\Msf_{\sigma}\cap\tau^{\perp})}\cap \CSpec_{\Tbb}\Tbb[\Msf_{\sigma}] \cong \CSpec_{\Tbb}\Tbb[\tau^{\perp}\cap M] 
  \end{split}
\end{align}
We write each stratum $O_{\Rbb\times (\Msf_{\sigma}\cap\tau^{\perp})}\cap \CSpec_{\Tbb}\Tbb[\Msf_{\sigma}]$ briefly as $N^{\con}(\tau)$.

Each element of $N^{\con}(\tau)$ corresponds to a pair of two flags by 
\cref{proposition:flag corr group T innout}:
\begin{align}\label{eq:toric process2}
  \begin{split}
    &N^{\con}(\tau) \longleftrightarrow \left\{(H_{\bullet}^{\mathrm{out}},H_{\bullet}^{\mathrm{inn}})=(\{H_{i}^{\mathrm{out}}\}_{i=0}^{l'},\{H_{i}^{\mathrm{inn}}\}_{i=0}^{l''})\ \middle|\  
  \begin{aligned}
    &H_{\bullet}^{\mathrm{out}}\in \Fcal_{\tau^{\perp}\cap M}, H_{\bullet}^{\mathrm{inn}}\in \Fcal_{\Rbb\times (H_{l'}^{\mathrm{out}}\cap M)},\\
    &\text{$l''>0$ and $\Rbb\times \{0\}\not\subseteq H_1^{\mathrm{inn}}$ holds.} 
  \end{aligned}
  \right\}
  \end{split}
\end{align}

To characterize points of maximal height,
we provide a way to construct a sequence of prime congruences.

\begin{lemma}\label{lemma:strict inclusion2}
  For $i=1,2$, let $P_i$ be a point of $\CSpec_{\Tbb}\Tbb[\Msf_{\sigma}]$ 
  and denote by $l'_i$ and $l''_i$ the lengths of the outer part and the inner part, respectively.
  If $P_1\subseteq P_2$, 
  then $l_2'\leq l_1'$ and $l_2''\leq l_1''$. 
  If $P_1\subsetneq P_2$, 
  then at least one of the inequalities is strict.
\end{lemma}

\begin{proof}
  Let $\Rbb\times \Fsf_i$ be the mobile face of $P_i$ for $i=1,2$.
  By the conditions (i) and (iii) of \cref{proposition:flag corr group T innout}, 
  the inclusion $\Fsf_2\subseteq \Fsf_1$ and
  the following two conditions hold:
  \begin{itemize}
    \item $\Lsf(\tilde{P_1})^{\mathrm{out}}\cap \Fsf_2^{\mathrm{gp}}=\Lsf(\tilde{P_2})^{\mathrm{out}}$,
    \item $\Lsf(\tilde{P_1})^{\mathrm{inn}}\cap (\Rbb\times \Fsf_2^{\mathrm{gp}})\subseteq \Lsf(\tilde{P_2})^{\mathrm{inn}}$.
  \end{itemize}
  In terms of flags,
  \begin{itemize}
    \item $H(\tilde{P_2})^{\mathrm{out}}_{\bullet}$ is 
    the restriction $H(\tilde{P_1})^{\mathrm{out}}_{\bullet}\cap^{*} \Fsf_2^{\mathrm{gp}}$,
    \item $H(\tilde{P_2})^{\mathrm{inn}}_{\bullet}$ is 
    a truncation of the restriction $H(\tilde{P_1})^{\mathrm{inn}}_{\bullet}\cap^{*}(\Rbb\times \Fsf_2^{\mathrm{gp}})$.
  \end{itemize}
  Thus we have $l_2'\leq l_1'$ and $l_2''\leq l_1''$.

  Assume $P_1\subsetneq P_2$.
  If the inclusion $\Fsf_2\subseteq \Fsf_1$ is strict, 
  the length of $H(P_1)^{\mathrm{out}}_{\bullet}\cap^{*} \Fsf_2^{\mathrm{gp}}$ is smaller than 
  that of $H(P_1)^{\mathrm{out}}_{\bullet}$. 
  Thus we have $l_2'<l_1'$.
  If $\Fsf_2=\Fsf_1$ holds, then
  \begin{itemize}
    \item $H(\tilde{P_2})^{\mathrm{out}}_{\bullet}=H(\tilde{P_1})^{\mathrm{out}}_{\bullet}$,
    \item $H(\tilde{P_2})^{\mathrm{inn}}_{\bullet}$ is a truncation of $H(\tilde{P_1})^{\mathrm{inn}}_{\bullet}$.
  \end{itemize}
  By the assumption $P_1\neq P_2$, this truncation is nontrivial and thus $l_2''<l_1''$ holds.
\end{proof}

Let $P$ be a point of a stratum $N^{\con}(\tau)$ where $\tau$ is a face of $\sigma$.
We assume that $P$ is geometric.
By \cref{proposition:strata over T,proposition:flag when geom}, $\tilde{P}$ is geometric and 
corresponds to an $(\Rbb\times (\tau^{\perp}\cap M))$-flag of length 1:
\begin{align*}
  &H(\tilde{P})_0=\Rbb\times \tau^{\perp}\supsetneq H(\tilde{P})_1,
\end{align*}
where $H(\tilde{P})_1$ is a hyperplane of $\Rbb\times \tau^{\perp}$ 
satisfying $\Rbb\times \{0\}\not\subseteq H(\tilde{P})_1$.

%outer part
We will show that there exists a sequence of prime congruences of length $n$ contained in $P$.
Take a sequence 
\[\{0\}=\tau_{0}\prec \tau_{1}\prec \dots\prec \tau_{\dim\tau}=\tau,\ \dim\tau_i=i,\]
of faces in $\tau$. Then we obtain a sequence of linear subspaces of $M_{\Rbb}$
\[M_{\Rbb}=\tau_{0}^{\perp}\supsetneq \tau_{1}^{\perp}\supsetneq \dots\supsetneq \tau_{\dim\tau-1}^{\perp} \supsetneq \tau^{\perp}.\]
We define an $M$-flag by
\begin{alignat*}{2}
  &H_i^{\mathrm{out}}\coloneqq\tau_i^{\perp} &\text{ for }i=0,\dots,\dim\tau,\\
  &H_{i,+}^{\mathrm{out}}\coloneqq\tau_{i-1}^{\perp}\cap(-\tau_i^{\vee}\setminus \tau_i^{\perp}) &\text{ for }i=1,\dots,\dim\tau,
\end{alignat*}
For each $j=0,\dots,\dim\tau$, the restriction $H_{\bullet}^{\mathrm{out}}\cap^{*}(\tau_{\dim\tau-j}^{\perp}\cap M)$ is of the form
\[\tau_{\dim\tau-j}^{\perp}\supsetneq \dots\supsetneq \tau^{\perp}\]
and has length $j$.

%inner part
On the other hand, since the rank of $H(\tilde{P})_1\cap (\Rbb\times M)\cong \tau^{\perp}\cap M$ 
is equal to $\dim\tau^{\perp}$, the flag $H(\tilde{P})_{\bullet}$ 
can be extended to one of length $1+\dim\tau^{\perp}$ arbitrarily.
That is, there exists an $(\Rbb\times (\tau^{\perp}\cap M))$-flag $H^{\mathrm{inn}}_{\bullet}$ 
of length $1+\dim\tau^{\perp}$ such that $H^{\mathrm{inn}}_1=H(\tilde{P})_1$.
In particular, for each $k=1,\dots,\dim\tau^{\perp}+1$, 
we obtain an $(\Rbb\times (\tau^{\perp}\cap M))$-flag $\{H^{\mathrm{inn}}_{i}\}_{i=0}^{k}$ of length $k$.

Thus we obtain a family of pairs of flags 
\[(H_{\bullet}^{\mathrm{out}}\cap^{*}(\tau_{\dim\tau-j}^{\perp}\cap M), \{H^{\mathrm{inn}}_{i}\}_{i=0}^{k}), 
\text{ for }j=0,\dots,\dim\tau, \text{ and }k=1,\dots,\dim\tau^{\perp}+1.\]
Let $P_{j,k}$ be the corresponding point in $N^{\con}(\tau_{\dim\tau-j})$ 
under the correspondence \eqref{eq:toric process2}.
Note that the mobile face of $P_{j,k}$ is $\Rbb\times (\tau_{\dim\tau-j}^{\perp}\cap \Msf_{\sigma})$.

\begin{lemma}\label{lemma:a sequence of maximal length}
  In the above setting, if $j'\leq j$ and $k'\leq k$, we have $P_{j,k}\subseteq P_{j',k'}$.
\end{lemma}

\begin{proof}
  It is enough to show the following two claims:
  \begin{itemize}
    \item $P_{j,k+1}\subseteq P_{j,k}$ for $0\leq j\leq \dim \tau$ and $1\leq k\leq \dim \tau^{\perp}$,
    \item $P_{j+1,k}\subseteq P_{j,k}$ for $0\leq j\leq \dim \tau-1$ and $1\leq k\leq \dim \tau^{\perp}+1$.
  \end{itemize}

  To show the first claim,
  among the conditions from \cref{proposition:inclusion monoid T over T},
  only the condition 
  \[\Lsf(\tilde{P_{j,k+1}})^{\mathrm{inn}}\cap (\Rbb\times (\tau_{\dim\tau-j}^{\perp}\cap M))\subseteq \Lsf(\tilde{P_{j,k}})^{\mathrm{inn}}\] 
  is nontrivial.
  Since $\Lsf(\tilde{P_{j,k+1}})$ is a submonoid of $\Rbb\times (\tau^{\perp}\cap M)$, 
  which is a subgroup of $\Rbb\times (\tau_{\dim\tau-j}^{\perp}\cap M)$, 
  the left hand side is equal to $\Lsf(\tilde{P_{j,k+1}})^{\mathrm{inn}}$
  and thus the above condition holds.
  
  The second claim is slightly complicated. There are three nontrivial conditions to check:
  \begin{enumerate}[label=(\roman*)]
    \item $\tau_{\dim\tau-j}^{\perp}\cap \Msf_{\sigma}\subseteq \tau_{\dim\tau-j-1}^{\perp}\cap \Msf_{\sigma}$,
    \item For any $m\in(\tau_{\dim\tau-j-1}^{\perp}\setminus \tau_{\dim\tau-j}^{\perp})\cap \Msf_{\sigma}$ 
    and $m'\in\tau_{\dim\tau-j}^{\perp}\cap \Msf_{\sigma}$, 
    $m-m'\notin \Lsf(\tilde{P_{j+1,k}})^{\mathrm{out}}$,
    \item $\Lsf(\tilde{P_{j+1,k}})^{\mathrm{out}}\cap (\tau_{\dim\tau-j}^{\perp}\cap M)=\Lsf(\tilde{P_{j,k}})^{\mathrm{out}}$.
  \end{enumerate}
  The conditions (i) and (iii) are immediate from the definitions (see \cref{definition:restriction}). 
  For $m$ and $m'$ of the condition (ii), we have
  \begin{align*}
    &m\in (\tau_{\dim\tau-j-1}^{\perp}\setminus \tau_{\dim\tau-j}^{\perp})\cap \Msf_{\sigma}
    \subseteq (\tau_{\dim\tau-j-1}^{\perp}\setminus \tau_{\dim\tau-j}^{\perp})\cap \tau_{\dim\tau-j}^{\vee}=-H_{\dim\tau-j,+}^{\mathrm{out}},\\
    &m'\in \tau_{\dim\tau-j}^{\perp}\cap \Msf_{\sigma}\subseteq \tau_{\dim\tau-j}^{\perp}=H_{\dim\tau-j}^{\mathrm{out}}.
  \end{align*}
  By \cref{lemma:out flag},
  \[m-m'\in -H_{\dim\tau-j,+}^{\mathrm{out}}\cap M\subseteq M\setminus \Lsf(\tilde{P_{j+1,k}})^{\mathrm{out}},\]
  holds. Thus the condition (ii) follows.
\end{proof}

\begin{theorem}\label{theorem:affine toric geometric ht}
  For a point $P$ of $\CSpec_{\Tbb}\Tbb[\Msf_{\sigma}]$, $P$ is geometric if and only if $\het P=n$.
\end{theorem}

\begin{proof}
  Assume that $P$ is geometric and lies in $N^{\con}(\tau)$.
  By \cref{lemma:a sequence of maximal length}, we have prime congruences:
  \[
  \begin{array}{ccccccc}
  P_{0,\dim\tau^{\perp}+1} & \subsetneq & \dots & \subsetneq & P_{0,1} &=& P \\[0.8em]
  \rotatebox{90}{$\subsetneq$} &&&& \rotatebox{90}{$\subsetneq$} &&\\
  \vdots &&\iddots&& \vdots &&\\[0.8em]
  \rotatebox{90}{$\subsetneq$} &&&& \rotatebox{90}{$\subsetneq$}&& \\
  P_{\dim\tau,\dim\tau^{\perp}+1} & \subsetneq & \dots & \subsetneq & P_{\dim\tau,1}&&
  \end{array}
  \]
  Note that these inclusions are strict since these prime congruences correspond to different flags.
  Any path from $P_{\dim\tau,\dim\tau^{\perp}+1}$ to $P_{0,1}$ has length $\dim\tau+\dim\tau^{\perp}=n$.
  Thus we have $\het P\geq n$.
  Recall that the length of $H(\tilde{P})_{\bullet}$ is equal to 1 and 
  the value $l'+l''$ of flags in \eqref{eq:toric process2} is at most $n+1$. 
  By \cref{lemma:strict inclusion2}, we have $\het P\leq n$.
  Thus every geometric congruence is of height $n$.

  Conversely, we assume that $\het P$ is equal to $n$.
  Let $l'\geq 0$ and $l''>0$ be the lengths of \eqref{eq:toric process2}.
  Since the length of flags appeared in \eqref{eq:toric process2} is at most $n+1$,
  we have $l'+l''\leq 1$ by \cref{lemma:strict inclusion2}.
  Thus $(l',l'')=(0,1)$ holds, 
  which implies that $\tilde{P}$ is geometric by \cref{proposition:flag when geom}.
  By \cref{proposition:strata over T}, $P$ is also geometric.
\end{proof}

We provide a technical lemma.

\begin{lemma}\label{lemma:cone}
  For an $M$-flag $\{H_i\}_{i=0}^l$, we have
  \[\cone \left(\left(\bigcup_{i=1}^{l}H_{i,+}\cup H_l\right)\cap M\right)=\bigcup_{i=1}^{l}H_{i,+}\cup \Rbb(H_l\cap M).\]
\end{lemma}

\begin{proof}
  Since the right hand side is closed under addition and multiplication by nonnegative real numbers, 
  it clearly contains the left hand side.
  The opposite inclusion is proved by the following inclusions.
  \begin{align*}
    &\cone \left(\left(\bigcup_{i=1}^{l}H_{i,+}\cup H_l\right)\cap M\right)\supseteq \cone (H_{j,+}\cap M)=H_{j,+}\cup \{0\} \text{ for }j=1,\dots,l,\\
    &\cone \left(\left(\bigcup_{i=1}^{l}H_{i,+}\cup H_l\right)\cap M\right)\supseteq \cone ( H_l\cap M)=\Rbb(H_l\cap M).
  \end{align*}
\end{proof}

Note that by \cref{proposition:abstract geom}, for any point $P\in U_{\sigma}$, 
its closure $\overline{\{P\}}$ contains at most one geometric congruence.

\begin{proposition}\label{proposition:local proper}
  Let $P$ be a point of $N^{\con}(\tau)\subseteq U_{\sigma}^{\con}$.
  Then $P$ is contained in a (unique) geometric congruence 
  if and only if the following condition holds:
  \[\sigma^{\vee}\cap\tau^{\perp}\subseteq \bigcup_{i=1}^{l'}-H(\tilde{P})_{i,+}^{\mathrm{out}}\cup H(\tilde{P})_{l'}^{\mathrm{out}}.\]
\end{proposition}

\begin{proof}
  Set $H_{\bullet}=H(\tilde{P})_{\bullet}^{\mathrm{out}}$.
  By \cref{proposition:abstract proper}, it is enough to show that 
  the following are equivalent:
  \begin{enumerate}
    \item $\sigma^{\vee}\cap\tau^{\perp}\subseteq \bigcup_{i=1}^{l'}-H_{i,+}\cup H_{l'}$,
    \item $\Msf_{\sigma}\cap \tau^{\perp}\subseteq \left(\bigcup_{i=1}^{l'}-H_{i,+}\cup H_{l'}\right)\cap M$.
  \end{enumerate}
  The inclusion (2) is obtained by taking the intersection of both sides of (1) with $M$.
  
  Assume the inclusion (2).
  Since $\sigma$ and $\tau$ are rational cones, we have
  \[\cone(\Msf_{\sigma}\cap \tau^{\perp})=\sigma^{\vee}\cap\tau^{\perp}.\]
  By \cref{lemma:cone}, we have
  \[\sigma^{\vee}\cap\tau^{\perp}\subseteq \bigcup_{i=1}^{l'}-H_{i,+}\cup \Rbb(H_{l'}\cap M)
  \subseteq \bigcup_{i=1}^{l'}-H_{i,+}\cup H_{l'}.\]
\end{proof}

\subsection{Toric varieties and their properties}

%まず従来のtrop toric
We now turn to the study of global objects.
The usual tropical toric varieties are introduced 
by \cite{Kajiwara2008TropicalToric} and \cite{payne2009analytification}.
We follow the definition of \cite{maclagan2015introduction}.

\begin{definition}[\cite{maclagan2015introduction}]\label{definition:trop toric}
  Let $\Sigma$ be a fan in $N_{\Rbb}$.
  As a set, the tropical toric variety $X_{\Sigma}^{\trop}$ is the following:
  \[X_{\Sigma}^{\trop}=\coprod_{\sigma\in\Sigma}N(\sigma),\]
  where $N(\sigma)=N_{\Rbb}/span(\sigma)$.
  The topology of $X_{\Sigma}^{\trop}$ is defined as follows:
\end{definition}

To give this a topology, we define the space of all homomorphisms of monoids 
from $(\sigma^{\vee}\cap M,+)$ to $(\Tbb,\odot)$:
\[U_{\sigma}^{\trop}=\Hom(\Msf_{\sigma},\Tbb),\]
where $\Msf_{\sigma}=\sigma^{\vee}\cap M$.
We equip $U_{\sigma}^{\trop}$ with the topology of pointwise convergence 
and $X_{\Sigma}^{\trop}$ with the topology induced by the covering 
$X_{\Sigma}^{\trop}=\bigcup_{\sigma\in\Sigma}U_{\sigma}^{\trop}$.

\medskip

Now we construct a topological space $X_{\Sigma}^{\con}$ from spaces of prime congruences.
As defined in the previous subsection, define $U_{\sigma}^{\con}$ for each cone $\sigma$ in $\Sigma$:
\[U_{\sigma}^{\con}=\CSpec_{\Tbb} \Tbb[\sigma^{\vee}\cap M].\]
If $\tau$ is a face of $\sigma$, 
we can find $m\in \sigma^{\vee}\cap M$ such that $\tau =m^{\perp}\cap \sigma$ 
where $m^{\perp}=\{n\in N_{\Rbb}\mid \<m,n\>=0\}$.
Then $\Msf_{\sigma}+\Zbb(-m)=\Msf_{\tau}$ holds.
The localization induces 
\[
\Tbb[\Msf_{\sigma}]\to\{0,m,2m,\dots\}^{-1}\Tbb[\Msf_{\sigma}]=\Tbb[\Msf_{\tau}].
\]
It induces an open immersion $U_{\tau}^{\con}\hookrightarrow U_{\sigma}^{\con}$.
Using these maps, we obtain a topological space $X_{\Sigma}^{\con}$.

As in the definition of $X_{\Sigma}^{\trop}$, 
the space $X_{\Sigma}^{\con}$ also has the stratification.
Recall that $U_{\sigma}^{\con}$ admits a stratification given by \eqref{eq:toric process}.
If $\tau$ is a face of $\sigma_1$ and $\sigma_2$, the construction of $X_{\Sigma}^{\con}$
identifies two strata $N^{\con}(\tau_i)$ for $i=1,2$, denoted by $N^{\con}(\tau)$.
Then we have the stratification:
\begin{align}\label{eq:strati toric}
  \begin{split}
    &X_{\Sigma}^{\con} = \coprod_{\tau\in\Sigma}N^{\con}(\tau),\\
    &N^{\con}(\tau) \xrightarrow{\cong} \CSpec_{\Tbb}\Tbb[\tau^{\perp}\cap M], \ P\mapsto \tilde{P}.
  \end{split}
\end{align}
Thus we obtain a stratification of $X_{\Sigma}^{\con}$ analogous to that of $X_{\Sigma}^{\trop}$.
By \eqref{eq:toric process2}, $X_{\Sigma}^{\con}$ can also be seen as a set of flags.

\begin{remark}
  In \eqref{eq:strati toric}, 
  one should be careful that different points may correspond to the same flags.
  Indeed, for any maximal cone $\tau'$, 
  the correspondence of \eqref{eq:strati toric} maps points on $N^{\con}(\tau')$ to $\CSpec_{\Tbb}\Tbb$.
  To avoid this ambiguity, we should record from which stratum each flag is mapped.
  More precisely, for a point $P$ on $N^{\con}(\tau)$, the corresponding objects must be a pair $(\tau,\tilde{P})$ and a triple
  $(\tau,H(\tilde{P})^{\mathrm{out}}_{\bullet},H(\tilde{P})^{\mathrm{inn}}_{\bullet})$.
  In this way, the correspondence \eqref{eq:strati toric} remains bijective, 
  although we often omit this additional data when it is clear from the context.
\end{remark}

Recall the map \eqref{eq:top emb}:
\[\iota_{\sigma}\colon\Hom(\Msf_{\sigma},\Tbb)\to \CSpec_{\Tbb}\Tbb[\Msf_{\sigma}].\]
By \cref{proposition:TopEmb} and \cref{theorem:affine toric geometric ht}, 
the image of $\iota_{\sigma}$ is the set of points of height $n$.
The morphisms $\{\iota_{\sigma}\}_{\sigma\in\Sigma}$ glue together 
to give an injective map $\iota:X_{\Sigma}^{\trop}\to X_{\Sigma}^{\con}$.
Recalling \cref{definition:dim of top space}, the map $\iota$ satisfies the following theorem.
\begin{theorem}\label{theorem:trop is height n toric}
  The inclusion $\iota\colon X_{\Sigma}^{\trop}\to X_{\Sigma}^{\con}$ induces 
  a homeomorphism 
  \[X_{\Sigma}^{\trop}\cong\{P\in X_{\Sigma}^{\con}\mid \dim_P X_{\Sigma}^{\con}=n\}. \]
\end{theorem}

In what follows, we regard $X_{\Sigma}^{\trop}$ as a subspace of $X_{\Sigma}^{\con}$.
A point $P$ of $X_{\Sigma}^{\con}$ is called a \emph{closed point} 
if a singleton $\{P\}$ is a closed subset.
Since $\dim X_{\Sigma}^{\con}$ is equal to $n$, any point $P$ in $X_{\Sigma}^{\trop}$ is closed.
However, a closed point is not of height $n$ in general.
Indeed, in \cref{example:affine line}, a point $P_{+}$ is closed in $\CSpec_{\Tbb}\Tbb[X^{\pm}]$ but of height 0.
Next, we examine when the two notions coincide, that is, when every closed point is of height $n$.

\begin{lemma}\label{lemma:help}
  For $i=1,2$, let $P_i$ be a point in $N^{\con}(\tau_i)$.
  If $P_2\in \overline{\{P_1\}}$, then $\tau_1$ is a face of $\tau_2$.
\end{lemma}

\begin{proof}
  Assume $P_2\in \overline{\{P_1\}}$.
  By the construction of $X_{\Sigma}^{\con}$, there exists a cone $\sigma\in \Sigma$ such that 
  $P_2\in U_{\sigma}^{\con}$. Since $U_{\sigma}^{\con}$ is an open subset of $X_{\Sigma}^{\con}$ 
  and $P_2\in \overline{\{P_1\}}$, 
  then $P_1\in U_{\sigma}^{\con}$. 
  For each $i=1,2$, $\Rbb\times (\Msf_{\sigma}\cap\tau_i^{\perp})$ be the mobile face
  of $P_i$ in $U_{\sigma}$.
  By $P_1\subseteq P_2$ and \cref{proposition:inclusion monoid T over T}, 
  we have $\Msf_{\sigma}\cap \tau_2^{\perp}\subseteq \Msf_{\sigma}\cap \tau_1^{\perp}$.
  Thus $\tau_1\prec \tau_2$ holds.
\end{proof}

For $m\in M$, we define $m^{\perp}=\{n\in N_{\Rbb}\mid \<m,n\>=0\}$.

\begin{lemma}\cite[Lemma 1.2.13]{cox2011toric} \label{lemma:cox lemma}
  Let $\sigma_1$, $\sigma_2$ be cones in $N_{\Rbb}$.
  Then
  \[\sigma_1\cap \sigma_2=m^{\perp}\cap \sigma_1=m^{\perp}\cap \sigma_2\]
  for any $m\in \relint(\sigma_1^{\vee}\cap -\sigma_2^{\vee})$.
\end{lemma}

%separated
\begin{theorem}\label{theorem:separatedness}
  Let $\Sigma$ be a fan in $N_{\Rbb}$ and $P$ an arbitrary point in $X_{\Sigma}^{\con}$.
  Then there is at most one point $P_0$ in $X_{\Sigma}^{\trop}$ contained 
  in $\overline{\{P\}}$.
\end{theorem}

\begin{proof}
  Assume that $\overline{\{P\}}\cap X_{\Sigma}^{\trop}$ contains two points $P_0$ and $P_1$.
  Let $\tau,\tau_0,$ and $\tau_1$ be the cones of $\Sigma$ giving the strata containing these points:
  \[P\in N^{\con}(\tau),\ P_0\in N^{\con}(\tau_0),\ P_1\in N^{\con}(\tau_1).\]
  By \cref{lemma:help}, we have $\tau$ is a face of $\tau_0$ and $\tau_1$.
  By \cref{lemma:cox lemma}, for any $m\in \relint(\tau_0^{\vee}\cap -\tau_1^{\vee})$,
  \begin{align}\label{eq:separation}
    \tau_0\cap\tau_1=\tau_0\cap m^{\perp}=\tau_1\cap m^{\perp}
  \end{align}
  holds. Now we fix $m\in M\cap \relint(\tau_0^{\vee}\cap -\tau_1^{\vee})$. 
  Note that $m\in \tau^{\perp}\cap M$ holds since $\tau$ is contained in $\tau_0\cap \tau_1$.

  To derive a contradiction, 
  we assume that $\tau_0$ and $\tau_1$ are not comparable, i.e., $\tau_0\not\prec\tau_1$ and $\tau_1\not\prec\tau_0$ hold.
  From \eqref{eq:separation}, we have $m\neq 0$.
  Considering in the chart $U_{\tau_0}$,
  $P$ and $P_0$ are prime congruences on $\Tbb[\Msf_{\tau_0}]$ and satisfy $P\subseteq P_0$.
  Since $m\in (\tau^{\perp}\cap M)\setminus (\tau_0^{\perp}\cap M)$ holds,
  \[m=m-0\notin \Lsf(\tilde{P})^{\mathrm{out}}\]
  by the condition (ii) of \cref{proposition:inclusion monoid T over T}.
  By applying the same argument to $U_{\tau_1}$, we have 
  \[-m\notin \Lsf(\tilde{P})^{\mathrm{out}}.\]
  This contradicts the condition \eqref{L2} of $\Lsf(\tilde{P})^{\mathrm{out}}$. 
  Thus $\tau_0$ and $\tau_1$ are comparable and 
  hence $P_0$ and $P_1$ lie in the same chart $U_{\tau_0}$ or $U_{\tau_1}$.
  By \cref{proposition:abstract geom}, we have $P_0=P_1$.
\end{proof}

%proper
\begin{theorem}\label{theorem:properness}
  Let $\Sigma$ be a fan in $N_{\Rbb}$.
  Then the following are equivalent:
  \begin{enumerate}
    \item for any $P\in X_{\Sigma}^{\con}$, there is a point $P_0$ in $X_{\Sigma}^{\trop}$ contained 
    in $\overline{\{P\}}$, 
    \item the fan $\Sigma$ is complete.
  \end{enumerate}
\end{theorem}

\begin{proof}
  First we show that (2) implies (1).
  By \cref{proposition:local proper}, it is enough to show that 
  for a cone $\tau\in \Sigma$ and a $(\tau^{\perp}\cap M)$-flag $H_{\bullet}$ of length $l$, 
  there exists a cone $\sigma\in\Sigma$ containing $\tau$ such that
  \begin{align}\label{eq:tau}
    \sigma^{\vee}\cap\tau^{\perp}\subseteq \bigcup_{i=1}^{l}-H_{i,+}\cup H_{l}.
  \end{align}

  If $l=0$, then the assertion is immediate. Thus we assume that $l>0$. 
  We show the claim by induction on the codimension $k$ of $\tau$.
  If $k$ is equal to $0$, then the claim clearly holds.
  Assume that the claim holds if the codimension is smaller than $k$.
  Note that the pairing $M\times N\to \Zbb$ induces 
  a pairing $(\tau^{\perp}\cap M)\times N/(\Rbb\tau\cap N)\to \Zbb$.
  With this pairing, consider a normal vector $v\in N_{\Rbb}/\Rbb\tau$ of $H_1$ pointing $-H_{1,+}$.
  Let $\pi\colon N_{\Rbb}\to N_{\Rbb}/\Rbb\tau$ be the projection.
  Since the fan 
  \[\Sigma/\tau\coloneqq \{\pi(\tau')\mid \tau'\in \Sigma,\ \tau\subseteq \tau'\}\] 
  is complete, 
  there is a unique cone $\tau_1$ of $\Sigma$ satisfying $\tau\prec \tau_1$ and $v\in\relint\pi(\tau_1)$.

  Using the induction assumption 
  for a $(\tau_1^{\perp}\cap M)$-flag $H_{\bullet}\cap^{*}(\tau_1^{\perp}\cap M)$,
  we obtain a cone $\sigma\in \Sigma$ satisfying $\tau_1\prec \sigma$ and
  \begin{align}\label{eq:tau_1}
    \sigma^{\vee}\cap\tau_1^{\perp}\subseteq \left(\bigcup_{i=1}^{l}-H_{i,+}\cup H_{l}\right)\cap \tau_1^{\perp}.
  \end{align}

  Since $\tau_1\prec\sigma$ and $v\in\relint\pi(\tau_1)$ hold, we have
  \begin{align}\label{eq:tautau}
    (\sigma^{\vee}\cap\tau^{\perp})\setminus \tau_1^{\perp}\subseteq \tau^{\perp}\cap\tau_1^{\vee}\setminus \tau_1^{\perp}\subseteq -H_{1,+}.
  \end{align}
  Combining \eqref{eq:tau_1} and \eqref{eq:tautau}, we obtain the desired inclusion \eqref{eq:tau}.

  Conversely we assume that the fan $\Sigma$ is not complete.
  Let $\Sigma'$ be a complete fan which contains $\Sigma$ as a subfan.
  Then $X_{\Sigma'}^{\con}$ contains $X_{\Sigma}^{\con}$ as an open subset.

  We fix a maximal cone $\tau'\in\Sigma'\setminus \Sigma$.
  Since $\tau'$ is maximal, $\tau'^{\vee}$ is a strictly convex cone of $M_{\Rbb}$.
  Hence we can take a hyperplane $H$ of $M_{\Rbb}$ and its orientation $H_{+}$
  satisfying $H\cap \tau'^{\vee}=\{0\}$ and $H_{+}\cap \tau'^{\vee}=\emptyset$.
  Define an $M$-flag $H^{\mathrm{out}}_{\bullet}$ of length 1 by 
  \begin{align*}
    &H_1^{\mathrm{out}}\coloneqq H,\\
    &H_{1,+}^{\mathrm{out}}\coloneqq H_{+}.
  \end{align*}
  Let $H_{\bullet}^{\mathrm{inn}}$ be an arbitrary $\Rbb\times (H\cap M)$-flag satisfying 
  $\Rbb\times \{0\}\not \subseteq H_1^{\mathrm{inn}}$.

  By \eqref{Tprocess2}, we obtain a point $P$ of $N^{\con}(\{0\})\subseteq X_{\Sigma}^{\con}$.
  Using \cref{proposition:local proper} in the chart $U_{\tau'}^{\con}$,
  there exists a point $P_0\in  U_{\tau'}^{\trop}$ which is contained in $\overline{\{P\}}$.
  As can be seen from the proof of \cref{proposition:local proper}, $P_0$ lies on a stratum $N^{\con}(\tau')$, 
  which is not contained in $X_{\Sigma}^{\con}$.
  By \cref{theorem:separatedness}, there is no point of $X_{\Sigma}^{\trop}$ 
  which is contained in $\overline{\{P\}}$.
\end{proof}

\begin{corollary}\label{corollary:closed points}
  Let $\Sigma$ be a fan in $N_{\Rbb}$.
  Then the subset $X_{\Sigma}^{\trop}\subseteq X_{\Sigma}^{\con}$ is equal to the set of closed points
  if and only if $\Sigma$ is complete.
\end{corollary}

\section{Finitely generated prime congruences}\label{section:fg}

Finally, we consider a purely algebraic problem: determining which prime congruences are finitely generated.
First we refer to a general result.

\begin{lemma} \label{lemma:fg of ext}
  Let $\phi\colon S_1\to S_2$ be a homomorphism of semirings and 
  $C$ be a finitely generated congruence on $S_1$.
  Then the push-out $C\cdot S_2$ is also finitely generated.
\end{lemma}

\begin{proof}
  The surjective case is \cite[Corollary 2.15]{BertramEaston2017TropicalNullstellensatz}.
  The general case follows by factoring $S_1\to S_2$ as $S_1\twoheadrightarrow \Im\phi\hookrightarrow S_2$.
\end{proof}

Let $\Msf$ be an ordered monoid and $P$ a point of $\CSpec \Bb[\Msf]$ whose mobile face is $\Fsf\subseteq \Msf$. 
Recall the relation of $P$ and $\tilde{P}$:
\begin{align*}
  &\tilde{P}=(P\cap \Bb[\Fsf])\cdot \Bb[\Fsf^{\mathrm{gp}}],\\
  &P=\<(\tilde{P}\cap \Bb[\Fsf])\cdot \Bb[\Msf],\Gamma_{\Fsf}\>.
\end{align*}
In general, the congruence $\Gamma_{\Fsf}$ is not finitely generated.
However, in the case of $\Msf=\Msf_{\sigma}$ or $\Rbb\times \Msf_{\sigma}$ (see \cref{section:toric}),
$\Gamma_{\Fsf}$ is finitely generated.
Indeed, the quotient of $\Msf$ by any face is a finitely generated monoid.

\begin{lemma}\label{lemma:consider orbit}
  Let $P$ be a prime congruence on $\Bb[\Msf_{\sigma}]$ or $\Tbb[\Msf_{\sigma}]$.
  Then $P$ is finitely generated if and only if $\tilde{P}$ is finitely generated.
\end{lemma}

\begin{proof}
  We only show the assertion for $\Bb[\Msf_{\sigma}]$ since the other case is similar.
  Let $\Fsf$ be the mobile face of $P$. 
  By the isomorphism $\Bb[\Fsf]\stackrel{\cong}{\longrightarrow}\Bb[\Msf_{\sigma}]/\Gamma_{\Fsf}$, 
  $P\cap \Bb[\Fsf]$ is pushed forward to $P\cdot (\Bb[\Msf_{\sigma}]/\Gamma_{\Fsf})$.
  Then by \cref{lemma:fg of ext}, 
  if $P$ is finitely generated, then $P\cap \Bb[\Fsf]$ and $\tilde{P}$ are also finitely generated.

  Conversely, assume that $\tilde{P}$ is finitely generated.
  Since $\Gamma_{\Fsf}$ is finitely generated, 
  it is enough to show that $\tilde{P}\cap \Bb[\Fsf]$ is finitely generated.
  By \cref{proposition:new fg}, the submonoid $\Lsf=\Lsf(\tilde{P})$ is a finitely generated monoid.
  Since $\Msf_{\sigma}$ is torsion-free, $\Fsf$ can be seen as a submonoid of $\Fsf^{\mathrm{gp}}$.
  By Gordan's Lemma (\cite[Proposition 1.2.17]{cox2011toric}), a submonoid
  \[\Delta\coloneqq \{(m,m')\in \Fsf^2\mid m-m'\in \Lsf\}\]
  of $(\Fsf^{\mathrm{gp}})^2$ is also finitely generated.
  Assume that $\Delta$ is generated by $(m_1,m'_1),\dots,(m_k,m'_k)$.
  We define a congruence $C$ on $\Bb[\Fsf]$ generated by
  \[\{(m_i\oplus m'_i,m_i)\mid 1\leq i\leq k\}.\]
  Then $\tilde{P}\cap \Bb[\Fsf]$ is equal to $C$ and hence finitely generated.
\end{proof}

By \cref{lemma:consider orbit}, it is enough to consider the case of $\Blau$ and $\lau$.

\begin{theorem} \label{theorem:fg}
  For a prime congruence $P$ on $\Blau$, the following are equivalent:
  \begin{enumerate}
    \item $P$ is finitely generated,
    \item $\Lsf(P)$ is a finitely generated monoid,
    \item $(\het P,\coht P)$ is equal to $(n,0)$ or $(n-1,1)$,
    \item $P$ is one of the following two: (i) the maximum congruence or (ii) $P$ corresponds to an $M$-flag of length 1
    where $H_1$ is a rational hyperplane.
  \end{enumerate}
\end{theorem}

\begin{proof}
  The equivalence of (1) and (2) is already proven in \cref{proposition:new fg}.
  The equivalence of (3) and (4) is also immediate by \cref{theorem:flag corr group B}.

  Then we will show the equivalence of (2) and (4).
  Assume the condition (2). 
  Then $\cone(\Lsf(P))$ is a closed cone. 
  Let $\{H_i\}_{i=0}^l$ be the $M$-flag corresponding to $P$.
  If the length $l$ is equal to $0$, then $P$ is the maximum congruence. 
  If the length $l$ is larger than $0$, by \cref{theorem:flag corr group B}, we have 
  \begin{align*}
    H_1&=\overline{\cone(\Lsf(P))}\cap\overline{-\cone(\Lsf(P))}\\
    &=\cone(\Lsf(P))\cap-\cone(\Lsf(P))\\
    &=\cone\left(\left(\bigcup_{i=1}^l H_{i,+}\cup H_l\right)\cap M\right)\cap-\cone\left(\left(\bigcup_{i=1}^l H_{i,+}\cup H_l\right)\cap M\right)\\
    &=\Rbb(H_l\cap M),
  \end{align*}
  which implies that $H_1$ is rational and the length $l$ is equal to $1$.
  Thus we have one implication (2)$\implies$(4).

  Finally we show the opposite implication.
  For (i) the maximum congruence $P_0$, 
  the corresponding submonoid $\Lsf(P_0)$ is equal to $M$, which is finitely generated.
  Let $P$ be (ii) a prime congruence whose $M$-flag is $M_{\Rbb}=H_0\supsetneq H_1$ 
  where $H_1$ is a rational hyperplane.
  Then we can take a finite subset $\{m_1,\dots,m_n\}$ of $M$ such that
  \begin{itemize}
    \item $m_1\in H_{1,+}$ whose image generates $M/(H_1\cap M)\cong \Zbb$ as a group, and
    \item $m_2,\dots,m_n$ generate $H_1\cap M$,
  \end{itemize}
  as groups.
  Thus the submonoid $\Lsf(P)$ of $M$ is generated by $\{m_1,\pm m_2,\dots,\pm m_n\}$.
\end{proof}

%% ここまで\Bb[M]の有限生成について．

The situation is similar in the case of $G=\Rbb\times M$.
Note that in this case, $\{m\in G \mid m\geq 0\}$ is equal to $\{(r,0)\mid r\geq 0\}$. 
Recall that a geometric congruence on $\lau$ is a prime congruence associated to a point $x\in\Rbb^n$.

\begin{theorem} \label{theorem:fg for T}
  For a prime congruence $P$ on $\lau$, the following are equivalent:
  \begin{enumerate}
    \item $P$ is finitely generated,
    \item $\Lsf(P)$ is a submonoid of $\Rbb\times M$ generated by finitely many elements and $\Rbb_{\geq 0}\times \{0\}$,
    \item $(\het P,\coht P)$ is equal to $(n+1,0)$ or $(n,1)$,
    \item $P$ is one of the following three: (i) the maximum congruence, (ii) a geometric congruence, 
    or (iii) a congruence which corresponds to an $(\Rbb\times M)$-flag $H_{\bullet}$ of length 1 
    and $H_1=\Rbb\times H$ where $H$ is a rational hyperplane of $M_{\Rbb}$.
  \end{enumerate}
\end{theorem}

\begin{proof}
  The proof is similar to that of \cref{theorem:fg}.
\end{proof}

Note that congruences in (4) of \cref{theorem:fg for T} 
are not lying over $\Tbb$ except for geometric congruences.
Thus the following corollary is immediate.

\begin{corollary}\label{corollary:fg for lying over T}
  Let $P$ be a prime congruence on $\lau$ lying over $\Tbb$.
  Then the following are equivalent:
  \begin{itemize}
    \item $P$ is finitely generated,
    \item $(\het P,\coht P)$ is equal to $(n,1)$, and
    \item $P$ is geometric.
  \end{itemize}
\end{corollary}

The following claim follows from \cref{lemma:consider orbit} and \cref{corollary:fg for lying over T}.

\begin{corollary}
  Let $P$ be a prime congruence on $\Tbb[\Msf_{\sigma}]$ lying over $\Tbb$.
  Then $P$ is geometric if and only if $P$ is finitely generated.
\end{corollary}

Thus we have obtained another characterization of the subset $U^{\trop}_{\sigma}$ of $U^{\con}_{\sigma}$.

\printbibliography

@article{giansiracusa2016equations,
  title={Equations of tropical varieties},
  author={Giansiracusa, Jeffrey and Giansiracusa, Noah},
  journal={Duke Mathematical Journal},
  volume={165},
  number={18},
  pages={3379--3433},
  year={2016},
  publisher={Duke University Press}
}

@article{joo2018prime,
  title={Prime congruences of additively idempotent semirings and a Nullstellensatz for tropical polynomials},
  author={Jo{\'o}, D{\'a}niel and Mincheva, Kalina},
  journal={Selecta Mathematica},
  volume={24},
  number={3},
  pages={2207--2233},
  year={2018},
  publisher={Springer}
}

@article{joo2026varieties,
  author  = {Jo{\'o}, D{\'a}niel and Mincheva, Kalina},
  title   = {Varieties of prime tropical ideals and the dimension
             of the coordinate semiring},
  journal = {Mathematical Proceedings of the Cambridge Philosophical Society},
  volume  = {181},
  number  = {1},
  pages   = {853--868},
  year    = {2026},
  doi     = {10.1017/S0305004126101960}
}

@article{maclagan2018tropical,
  title={Tropical ideals},
  author={Maclagan, Diane and Rinc{\'o}n, Felipe},
  journal={Compositio Mathematica},
  volume={154},
  number={3},
  pages={640--670},
  year={2018},
  publisher={London Mathematical Society}
}

@article{BertramEaston2017TropicalNullstellensatz,
  author       = {Aaron Bertram and Robert Easton},
  title        = {The tropical Nullstellensatz for congruences},
  journal      = {Advances in Mathematics},
  volume       = {308},
  pages        = {36--82},
  year         = {2017},
}

@book{maclagan2015introduction,
  title={Introduction to tropical geometry},
  author={Maclagan, Diane and Sturmfels, Bernd},
  volume={161},
  year={2015},
  publisher={American Mathematical Society}
}

@article{rabinoff2012tropical,
  title={Tropical analytic geometry, Newton polygons, and tropical intersections},
  author={Rabinoff, Joseph},
  journal={Advances in Mathematics},
  volume={229},
  number={6},
  pages={3192--3255},
  year={2012},
  publisher={Elsevier}
}

@article{friedenberg2025geometric,
  author  = {Friedenberg, Netanel and Mincheva, Kalina},
  title   = {Geometric interpretation of valuated term (pre)orders},
  journal = {Michigan Mathematical Journal},
  pages   = {1--33},
  year    = {2025},
  note    = {Advance publication},
  doi     = {10.1307/mmj/20236430}
}

@misc{song2024geometricinterpretationkrulldimensions,
      title={A geometric interpretation of Krull dimensions of $\boldsymbol{T}$-algebras}, 
      author={JuAe Song and Yasuhito Nakajima},
      year={2024},
      eprint={2408.02366v2},
      archivePrefix={arXiv},
      primaryClass={math.AG},
      url={https://arxiv.org/abs/2408.02366}, 
}

@book{hebisch1998semirings,
  title={Semirings: algebraic theory and applications in computer science},
  author={Hebisch, Udo and Weinert, Hanns Joachim},
  volume={5},
  year={1998},
  publisher={World Scientific}
}

@book{golan2013semirings,
  title={Semirings and their Applications},
  author={Golan, Jonathan S},
  year={2013},
  publisher={Springer Science \& Business Media}
}

@book{hartshorne2013algebraic,
  title={Algebraic geometry},
  author={Hartshorne, Robin},
  volume={52},
  year={2013},
  publisher={Springer Science \& Business Media}
}

@book{cox2011toric,
  title={Toric varieties},
  author={Cox, David A and Little, John B and Schenck, Henry K},
  volume={124},
  year={2011},
  publisher={American Mathematical Society}
}

@incollection{Kajiwara2008TropicalToric,
  author    = {Kajiwara, Takeshi},
  title     = {Tropical toric geometry},
  booktitle = {Toric Topology},
  series    = {Contemporary Mathematics},
  volume    = {460},
  publisher = {American Mathematical Society},
  address   = {Providence, RI},
  year      = {2008},
  pages     = {197--207}
}

@article{payne2009analytification,
  title={Analytification is the limit of all tropicalizations},
  author={Payne, Sam},
  journal={Mathematical research letters},
  volume={16},
  number={3},
  pages={543--556},
  year={2009},
  publisher={International Press of Boston, Inc. Somerville, MA 02143, USA}
}

@misc{mikami2025tropicalcohomologysmoothalgebraic,
      title={On tropical cohomology of smooth algebraic varieties}, 
      author={Ryota Mikami},
      year={2025},
      eprint={2009.04690v6},
      archivePrefix={arXiv},
      primaryClass={math.AG},
      url={https://arxiv.org/abs/2009.04690}, 
}

@misc{ito2026nonfinitegeneratednesscongruencesdefined,
      title={Non-finite generatedness of the congruences defined by tropical varieties}, 
      author={Takaaki Ito},
      year={2026},
      eprint={2512.21565v2},
      archivePrefix={arXiv},
      primaryClass={math.AC},
      url={https://arxiv.org/abs/2512.21565}, 
}

@article{lorscheid2012geometry,
  title={The geometry of blueprints: Part I: Algebraic background and scheme theory},
  author={Lorscheid, Oliver},
  journal={Advances in Mathematics},
  volume={229},
  number={3},
  pages={1804--1846},
  year={2012},
  publisher={Elsevier}
}

@article{jarra2025strong,
  title={Strong congruence spaces and dimension in F1-geometry},
  author={Jarra, Manoel},
  journal={Journal of Pure and Applied Algebra},
  volume={229},
  number={2},
  pages={107862},
  year={2025},
  publisher={Elsevier}
}

\end{document}